\documentclass{article} 
\usepackage{amsmath,amssymb,amsthm} 
\usepackage[OT2,T1]{fontenc}
\usepackage[utf8]{inputenc}
\usepackage{enumitem}
\usepackage{fixltx2e}
\usepackage{graphicx}
\usepackage{geometry}
\usepackage{hyperref}
\usepackage{tikz}
\usetikzlibrary{cd}
\usepackage{tikz-cd}
\usepackage{xcolor}
\usepackage{ stmaryrd }
\usepackage{mathrsfs}
\usepackage{titling}
\usepackage{fullpage}
\usepackage{authblk}
\usepackage{mathtools}

\usepackage[alphabetic,lite]{amsrefs}

\usetikzlibrary{arrows,positioning,decorations.pathmorphing,
	decorations.markings
}
\tikzset{inner sep=0pt,
	root/.style={circle,draw,minimum size=7pt,thick},
	fatroot/.style={circle,draw,minimum size=10pt,thick},
	short root/.style={circle,fill,minimum size=7pt},
	doublearrow/.style={postaction={decorate},
		decoration={markings,mark=at position .7
			with {\arrow{angle 60}}},double distance=3pt,thick}
}

\DeclareUnicodeCharacter{00A0}{ }

\newtheorem{proposition}{Proposition}[section]
\newtheorem{definition}[proposition]{Definition}
\newtheorem{remark}[proposition]{Remark}

\newtheorem{theorem}[proposition]{Theorem}
\newtheorem{lemma}[proposition]{Lemma}
\newtheorem{corollary}[proposition]{Corollary}

\numberwithin{equation}{subsection}

\newcommand{\G}{\mathbb{G}}
\DeclareMathOperator{\GL}{GL}
\DeclareMathOperator{\PGL}{PGL}
\DeclareMathOperator{\SL}{SL}

\DeclareMathOperator{\Sp}{Sp}
\DeclareMathOperator{\PSp}{PSp}

\DeclareMathOperator{\liesl}{\mathfrak{sl}}
\DeclareMathOperator{\liesp}{\mathfrak{sp}}

\DeclareMathOperator{\lieh}{\mathfrak{h}}
\DeclareMathOperator{\liet}{\mathfrak{t}}
\DeclareMathOperator{\Ad}{Ad}
\DeclareMathOperator{\Lie}{Lie}

\DeclareMathOperator{\Id}{Id}

\DeclareMathOperator{\Hom}{Hom}
\DeclareMathOperator{\Aut}{Aut}

\DeclareMathOperator{\Sym}{Sym}
\DeclareMathOperator{\vol}{vol}

\DeclareMathOperator{\tr}{tr}

\DeclareMathOperator{\image}{image}

\DeclareMathOperator{\Isom}{Isom}

\DeclareMathOperator{\Spec}{Spec}

\DeclareMathOperator{\Frac}{Frac}

\DeclareMathOperator{\Pic}{Pic}

\newcommand{\sh}[1]{\mathscr{#1}}

\newcommand{\A}{\mathbb{A}}
\renewcommand{\P}{\mathbb{P}}

\renewcommand{\O}{\mathcal{O}}
\newcommand{\GIT}{\mathbin{/\mkern-6mu/}}
\newcommand{\HH}{\mathrm{H}}

\newcommand{\Real}{\mathbb{R}}
\newcommand{\C}{\mathbb{C}}
\newcommand{\Q}{\mathbb{Q}}
\newcommand{\Z}{\mathbb{Z}}
\newcommand{\F}{\mathbb{F}}

\DeclareMathOperator{\Sel}{Sel}
\DeclareMathOperator{\rk}{rk}

\DeclareSymbolFont{cyrletters}{OT2}{wncyr}{m}{n}
\DeclareMathSymbol{\Sha}{\mathalpha}{cyrletters}{"58}

\newcommand{\bwedge}[1]{\bigwedge\nolimits^{#1}}

\newcommand{\extp}{\@ifnextchar^\@extp{\@extp^{\,}}}
\def\@extp^#1{\mathop{\bigwedge\nolimits^{\!#1}}}
\newcommand{\overbar}[1]{\mkern 1.5mu\overline{\mkern-1.5mu#1\mkern-1.5mu}\mkern 1.5mu}
\DeclareUnicodeCharacter{00A0}{ }
\newcommand{\define}[1]{{\fontfamily{cmss}\selectfont{#1}}}

\DeclareMathOperator{\GrLie}{GrLie}
\DeclareMathOperator{\GrLieE}{GrLieE}
\newcommand{\intbigG}{\underline{\mathsf{G}}_{\mathrm{E}}}
\newcommand{\intbigH}{\underline{\mathsf{H}}_{\mathrm{E}}}
\newcommand{\intbigV}{\underline{\mathsf{V}}_{\mathrm{E}}}

\newcommand{\bigH}{\mathsf{H}_{\mathrm{E}}}
\newcommand{\bigh}{\lieh_{\mathrm{E}}}
\newcommand{\bigG}{\mathsf{G}_{\mathrm{E}}}
\newcommand{\bigB}{\mathsf{B}_{\mathrm{E}}}
\newcommand{\bigP}{\mathsf{P}_{\mathrm{E}}}
\newcommand{\bigT}{\mathsf{T}_{\mathrm{E}}}
\newcommand{\bigt}{\mathfrak{t}_{\mathrm{E}}}
\renewcommand{\bigg}{\mathfrak{g}_{\mathrm{E}}}
\newcommand{\bigV}{\mathsf{V}_{\mathrm{E}}}
\newcommand{\bigtheta}{\theta_{\mathrm{E}}}
\newcommand{\bigsigma}{\sigma_{\mathrm{E}}}
\newcommand{\bigkappa}{\kappa_{\mathrm{E}}}
\newcommand{\bigLambda}{\Lambda_{\mathrm{E}}}
\newcommand{\bigpi}{\pi_{\mathrm{E}}}
\newcommand{\bigA}{\mathsf{A}_{\mathrm{E}}}

\newcommand{\intsmallH}{\underline{\mathsf{H}}}
\newcommand{\intsmallV}{\underline{\mathsf{V}}}

\newcommand{\intsmallG}{\underline{\mathsf{G}}}
\newcommand{\intsmallB}{\underline{\mathsf{B}}}
\newcommand{\intsmallh}{\underline{\mathfrak{h}}}
\newcommand{\smallH}{\mathsf{H}}
\newcommand{\smallh}{\mathfrak{h}}
\newcommand{\smallP}{\mathsf{P}}
\newcommand{\smallN}{\mathsf{N}}
\newcommand{\smallG}{\mathsf{G}}
\newcommand{\smallg}{\mathfrak{g}}
\newcommand{\smallV}{\mathsf{V}}
\newcommand{\smallW}{\mathsf{W}}
\newcommand{\smallT}{\mathsf{T}}
\newcommand{\smallt}{\mathfrak{t}}
\newcommand{\smallB}{\mathsf{B}}
\newcommand{\smallkappa}{\kappa}
\newcommand{\smalltheta}{\theta}
\newcommand{\smallsigma}{\sigma}

\newcommand{\smallpi}{\pi}

\newcommand{\smalldisc}{\Delta}
\newcommand{\smallPopp}{\overbar{\mathsf{P}}}
\newcommand{\smallNopp}{\overbar{\mathsf{N}}}

\newcommand{\genB}{\mathcal{B}}

\newcommand{\intrhoV}{\underline{\mathsf{V}}^{\star}}
\newcommand{\intrhoG}{\underline{\mathsf{G}}^{\star}}
\newcommand{\intrhoB}{\underline{\mathsf{B}}^{\star}}

\newcommand{\rhoG}{\mathsf{G}^{\star}}

\newcommand{\rhoV}{\mathsf{V}^{\star}}
\newcommand{\rhoB}{\mathsf{B}^{\star}}

\newcommand{\rhopi}{\pi^{\star}}
\newcommand{\rhodisc}{\Delta^{\star}}

\newcommand{\intsmallcurve}{\mathcal{C}}

\newcommand{\bigprojcurve}{C_{\mathrm{E}}}

\newcommand{\smallprojcurve}{C}
\newcommand{\ellcurve}{E}
\newcommand{\Cellcurve}{\overbar{E}}
\newcommand{\intellcurve}{\mathcal{E}}
\newcommand{\Neronellcurve}{\mathscr{E}}
\newcommand{\Jac}{J}
\newcommand{\intJac}{\mathcal{J}}
\newcommand{\intCJac}{\bar{\mathcal{J}}}
\newcommand{\NeronJac}{\mathscr{J}}

\newcommand{\Prym}{P}
\newcommand{\intPrym}{\mathcal{P}}
\newcommand{\NeronPrym}{\mathscr{P}}
\newcommand{\curvezeta}{\tau}
\newcommand{\CJac}{\bar{J}}
\newcommand{\CPrym}{\overbar{P}}
\newcommand{\intCPrym}{\overbar{\mathcal{P}}}
\newcommand{\bigJac}{J_{\mathrm{E}}}
\newcommand{\bigCJac}{\bar{J}_{\mathrm{E}}}
\newcommand{\intCellcurve}{\overbar{\mathcal{E}}}
\newcommand{\bigresolv}{\mathcal{Q}}
\newcommand{\rhoDelta}{\Delta^{\star}}

\newcommand{\splitLambda}{\Lambda}

\newcommand{\Siegel}{\mathfrak{S}}

\newcommand{\height}{\text{ht}}

\DeclareMathOperator{\ord}{ord}
\DeclareMathOperator{\rs}{rs}
\DeclareMathOperator{\reg}{reg}
\DeclareMathOperator{\Eq}{Eq}
\newcommand{\rhoBrs}{\mathsf{B}^{\star,\rs}}
\newcommand{\intrhoBrs}{\underline{\mathsf{B}}^{\star,\rs}}
\newcommand{\rhoVrs}{\mathsf{V}^{\star,\rs}}
\newcommand{\intrhoVrs}{\underline{\mathsf{V}}^{\star,\rs}}
\newcommand{\rhodual}{\hat{\rho}}

\usepackage{microtype}

\title{Arithmetic statistics of Prym surfaces}
\author{Jef Laga}

\begin{document}
\maketitle
\begin{abstract}
	We consider a family of abelian surfaces over $\Q$ arising as Prym varieties of double covers of genus-$1$ curves by genus-$3$ curves. 
	These abelian surfaces carry a polarization of type $(1,2)$ and we show that the average size of the Selmer group of this polarization equals $3$.
	Moreover we show that the average size of the $2$-Selmer group of the abelian surfaces in the same family is bounded above by $5$. 
	This implies an upper bound on the average rank of these Prym varieties, 
	and gives evidence for the heuristics of Poonen and Rains for a family of abelian varieties which are not principally polarized. 
	
	The proof is a combination of an analysis of the Lie algebra embedding $F_4\subset E_6$, invariant theory, a classical geometric construction due to Pantazis, a study of Néron component groups of Prym surfaces and Bhargava's orbit-counting techniques. 
\end{abstract}
\tableofcontents	
\section{Introduction}

\subsection{Context}

Let $\lambda : A \rightarrow B$ be an isogeny of abelian varieties over $\Q$. The $\lambda$-\define{Selmer group} of $A$ is defined by 
\begin{equation*}
	\Sel_{\lambda}A \coloneqq \ker\left(\HH^1(\Q,A[\lambda]) \rightarrow \prod_v \HH^1(\Q_v,A) \right),
\end{equation*}
where $A[\lambda]$ denotes the kernel of $\lambda$, the cohomology groups are Galois cohomology and the product runs over all places $v$ of $\Q$.
It is a finite group defined by local conditions and fits in an exact sequence 
$$
0 \rightarrow B(\Q)/\lambda(A(\Q)) \rightarrow \Sel_{\lambda}A \rightarrow \Sha(A/\Q)[\lambda]\rightarrow 0.
$$
The determination of $\Sel_{\lambda}A$, known as performing a $\lambda$-descent, is often the first step towards determining the finitely generated abelian groups $A(\Q)$ and $B(\Q)$. 
One is therefore led to ask how $\Sel_{\lambda}A$ behaves on average as $\lambda$ varies in families.
When $A=B$ ranges over a family of Jacobian varieties and $\lambda$ is multiplication by an integer, the last ten years have seen spectacular progress in this direction; see for example \cite{BS-2selmerellcurves, BS-3Selmer, Bhargava-Gross-hyperellcurves, ShankarWang-hypermarkednonweierstrass, Thorne-Romano-E8} for works of particular relevance to this paper.
There are some results when $A$ is not a Jacobian variety (see for example \cite{BhargavaKlagsbrunZevOliver, Morgan-Quadratictwistsabelianvarietiesselmer, MorganPaterson-2Selmergroupstwistquadraticextension}) but they concern twists of a single abelian variety over $\Q$, therefore considering only an isotrivial family in the relevant moduli space.
By contrast in this paper we study for the first time a non-isotrivial family of abelian varieties which are not Jacobians. 


\subsection{Statement of results}

Let $\sh{E} \subset \Z^4$ be the subset of $4$-tuples of integers $b=(p_2,p_6,p_8,p_{12})$ such that the projective closure of the equation
\begin{equation}\label{equation: intro f4 equation}
y^4+p_2xy^2+p_6y^2 = x^3+p_8x+p_{12}
\end{equation}
defines a smooth genus-$3$ curve $\smallprojcurve_b$ over $\Q$.
The quotient of $\smallprojcurve_b$ by the involution $\tau(x,y) = (x,-y)$ is an elliptic curve $\ellcurve_b$ given by the equation 
\begin{equation}\label{equation: intro ell curve}
y^2+p_2xy+p_6y = x^3+p_8x+p_{12}.
\end{equation}
The associated morphism $f:\smallprojcurve_b\rightarrow \ellcurve_b$ is a double cover ramified at four points, namely the ones with $y=0$ and the point at infinity.
The families of curves $C_b$ and $E_b$ parametrized by such $b$ have a moduli interpretation, see Remark \ref{remark: moduli interpretation curves}.

Let $\Jac_b$ be the Jacobian variety of $\smallprojcurve_b$ and let $\Prym_b$ be the kernel of the norm map $f_* \colon \Jac_b \rightarrow \ellcurve_b$. 
Then $\Prym_b$ is an abelian surface carrying a polarization $\rho: \Prym_b \rightarrow \Prym_b^{\vee}$ of type $(1,2)$. (This means that $\Prym_b[\rho](\overbar{\Q})\simeq (\Z/2\Z)^2$.)
It is called the \define{Prym variety} associated to the double cover $\smallprojcurve_b \rightarrow \ellcurve_b$. 
The abelian threefold $\Jac_b$ is isogenous to $\Prym_b \times \ellcurve_b$. 

For $b\in \sh{E}$ we define the \define{height} of $b$ as 
$$\height(b) = \max |p_i(b)|^{1/i}.    $$
Note that for every $X\in \Real_{>0}$, the set $\{b\in \sh{E} \mid \height(b) < X  \}$ is finite.


\begin{theorem}[Theorem \ref{theorem: average size rho selmer group}]   \label{theorem: intro rho selmer}
	The average size of $\Sel_{\rho} \Prym_b$ for $b\in \sh{E}$, when ordered by height, equals $3$. 
	More precisely, we have
	\begin{equation*}
	\lim_{X\rightarrow\infty} \frac{ \sum_{b\in \sh{E},\; \height(b)<X }\# \Sel_{\rho}\Prym_b  }{\#  \{b \in \sh{E}\mid \height(b) < X\}}	= 3.
	\end{equation*}
\end{theorem}

\begin{theorem}[Theorem \ref{theorem: the average size of the 2-Selmer group}]   \label{theorem: intro 2 selmer}
	The average size of $\Sel_{2} \Prym_b$ for $b\in \sh{E}$, when ordered by height, is bounded above by $5$. 
	More precisely, we have
	\begin{equation*}
	\limsup_{X\rightarrow\infty} \frac{ \sum_{b\in \sh{E},\; \height(b)<X }\# \Sel_{2}\Prym_b   }{\#  \{b \in \sh{E}\mid \height(b) < X\}}	\leq 5.
	\end{equation*}
\end{theorem}

In fact, both theorems also hold when $\sh{E}$ is replaced by a subset defined by finitely many congruence conditions.

\begin{remark}
We expect that the limit in Theorem \ref{theorem: intro 2 selmer} exists and equals $5$, see the end of \S\ref{subsection: intro methods}.
	
\end{remark}

We mention a few standard consequences of the above theorems. 
The first one concerns the Mordell--Weil rank $\rk(\Prym_b)$ of $\Prym_b$.
Using the inequalities $2\rk(\Prym_b) \leq 2^{\rk(\Prym_b)}\leq \#\Sel_2\Prym_b$, Theorem \ref{theorem: intro 2 selmer} immediately implies:

\begin{corollary}\label{corollary: intro bound av rank prym}
	The average rank of $\Prym_b$ for $b\in \sh{E}$, when ordered by height, is bounded above by $5/2$. 
\end{corollary}

Because the rank of $\Jac_b$ equals the sum of the ranks of its isogeny factors $\Prym_b$ and $\ellcurve_b$, Corollary \ref{corollary: intro bound av rank prym} also gives a bound on the average rank of the family of Jacobians $\Jac_b$ for $b\in \sh{E}$, once a bound for the average rank of $\ellcurve_b$ is known. 
Since the statistical properties of Selmer groups of the family of elliptic curves $\ellcurve_b$ reduce to those of the family of elliptic curves in short Weierstrass form (see Remark \ref{remark: changing variables elliptic curves}), we may use the previously obtained estimates in the case of elliptic curves \cite[Theorem 3]{BS-5Selmer} to obtain: 

\begin{corollary}
	The average rank of $\Jac_b$ for $b\in \sh{E}$, when ordered by height, is $<5/2+0.885 = 3.385$.
\end{corollary}



\subsection{Methods}\label{subsection: intro methods}

The basic proof strategy is the same as the one employed in previous works: for each of the isogenies $\rho$ and $[2]$, we construct a representation of a reductive group over $\Q$ whose integral orbits parametrize Selmer elements and then count those orbits using the geometry-of-numbers techniques pioneered by Bhargava and his collaborators.
Given the robustness of these counting techniques, the crux of the matter is finding the right representation in the first place and showing that its rational orbits relate to the arithmetic of our isogeny of interest.

Previous cases suggest that relevant representations can very often be constructed using graded Lie algebras.
In the special case of $\Z/2\Z$-gradings on simply laced Lie algebras, Thorne \cite{Thorne-thesis} has made this very explicit using the connection with simple singularities \cite{Slodowy-simplesingularitiesalggroups}, paving the way for studying the $2$-Selmer groups of certain families of curves using orbit-counting techniques. (See the introduction of \cite{Laga-E6paper} for a more detailed exposition.)
In the classical cases $A_n$ or $D_n$ the families of curves in question are hyperelliptic with marked points and most of these results were already obtained using different methods (where the papers \cite{Bhargava-Gross-hyperellcurves, ShankarWang-hypermarkednonweierstrass, Shankar-2Selmerhyperell2markedpoints} handle the cases $A_{2n}, A_{2n+1}$, $D_{2n+1}$ respectively), but in the exceptional cases $E_6, E_7, E_8$ the curves are not hyperelliptic and this framework has led to new results: see \cite{Thorne-E6paper, Romano-Thorne-ArithmeticofsingularitiestypeE, Thorne-Romano-E8, Laga-E6paper}.

The present work is a first attempt to incorporate non-simply laced Dynkin diagrams in the above picture, and more specifically the Dynkin diagram of type $F_4$. 
Since non-simply laced Dynkin diagrams have a more complicated relationship to geometry (as can be seen in the work of Slodowy \cite{Slodowy-simplesingularitiesalggroups} which forms the basis of Thorne's framework), this introduces various difficulties.
The starting observation is the following.
If $\bigh$ is a simple complex Lie algebra of type $E_6$, then there exists an involution $\zeta\colon \bigh \rightarrow \bigh$ whose fixed point subalgebra $\bigh^{\zeta}$ is a simple complex Lie algebra of type $F_4$. 
This procedure is somewhat informally depicted as \emph{folding the Dynkin diagram} of $E_6$:

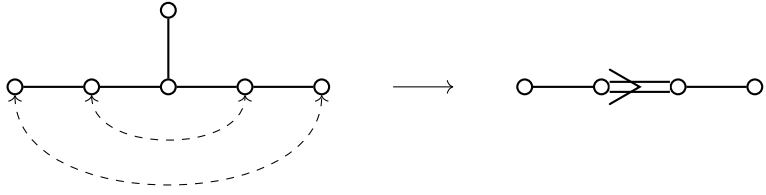
\begin{center}
	\begin{tikzpicture}[transform shape, scale=0.8]
	\node[root] (a1) {};
	\node[root] (a2) [right=of a1]{};
	\node[root] (a3) [right=of a2]{};
	\node[root] (a4) [above=of a3]{};
	\node[root] (a5) [right=of a3]{};
	\node[root] (a6) [right=of a5]{};
	\draw[thick] (a1) -- (a2) ;
	\draw[thick] (a2) -- (a3) ;
	\draw[thick] (a3) -- (a4) ;
	\draw[thick] (a3) -- (a5) ;
	\draw[thick] (a5) -- (a6) ;
	\draw[<->, dashed] (a1) to [out=-90, in=-90] (a6) ;
	\draw[<->, dashed] (a2) to [out=-90, in=-90] (a5) ;
	\node (s1)  [right=of a6]{};
	\node (s2)  [right=of s1]{};	
	\draw[->] (s1) to (s2);
	\node[root] (b) [right=of s2]{};
	\node[root] (c) [right=of b] {};
	\node[root] (d) [right=of c] {};
	\node[root] (e) [right=of d] {};
	\draw[thick] (b) -- (c) ;
	\draw[thick] (d) -- (e);
	\draw[doublearrow] (c)--(d);
	\end{tikzpicture}
\end{center} 

It suggests that studying the $F_4$ case should correspond to studying the $E_6$ case equivariantly with respect to the symmetry of the Dynkin diagram. This viewpoint is already present in the work of Slodowy \cite{Slodowy-simplesingularitiesalggroups} where he identifies the restriction of the adjoint quotient of the $F_4$ Lie algebra to a subregular transverse slice as the semi-universal deformation of the $E_6$ surface singularity with `fixed symmetries', and analogously for other non-simply laced Lie algebras.
We will approach Theorems \ref{theorem: intro rho selmer} and \ref{theorem: intro 2 selmer} similarly.


In more detail, we will define an involution $\zeta$ on the representation $(\bigG,\bigV)$ constructed by Thorne in the $E_6$ case, whose fixed points give rise to a representation $\smallV$ of a reductive group $\smallG$. 
The family $\smallprojcurve$ of Equation (\ref{equation: intro f4 equation}) is then the subfamily of the semi-universal deformation of the $E_6$ curve singularity (explicitly given by Equation (\ref{equation : E6 family middle of paper})) to which the involution $\tau(x,y) =(x,-y)$ lifts.
In our previous work \cite{Laga-E6paper} we have constructed an embedding of $\Sel_2 \Jac_b$ in the set of $\bigG(\Q)$-orbits of $\bigV(\Q)$.
The techniques of that paper combined with a detailed study of the actions of $\tau$ and $\zeta$ allow us to embed $\Sel_2 \Prym_b$ into the set of $\smallG(\Q)$-orbits of $\smallV(\Q)$.
In that same paper, a general construction of integral orbit representatives was given using properties of compactified Jacobians. 
A similar construction works here using a compactified Prym variety instead.

It then seems that Theorem \ref{theorem: intro 2 selmer} follows from geometry-of-numbers arguments to count integral orbits in $\smallV$, but there is a catch: such arguments will only allow us to count `strongly irreducible' elements of $\Sel_2 \Prym_b$.
To explain what this means, note that there exists a unique isogeny $\rhodual \colon \Prym_b^{\vee} \rightarrow \Prym_b$ such that $[2] = \rhodual \circ \rho$, giving rise to the exact sequence
$$\Sel_{\rho} \Prym_b \rightarrow \Sel_2 \Prym_b \rightarrow \Sel_{\rhodual} \Prym_b^{\vee}.   $$
We say an element of $\Sel_2 \Prym_b$ is \define{strongly irreducible} if it has nontrivial image in $\Sel_{\rhodual} \Prym_b^{\vee}$. 
Estimating $\Sel_2\Prym_b$ then breaks up into two parts: estimating the strongly irreducible elements (which can be done using the representation $\smallV$), and $\Sel_{\rho}\Prym_b$.
This is not unlike the situation of \cite{BS-4Selmer}, where the representation used in that paper only counts elements of the $4$-Selmer group of an elliptic curve of \emph{exact} order $4$, i.e. having nontrivial image in the $2$-Selmer group. 


Therefore to prove Theorem \ref{theorem: intro 2 selmer} it remains to prove Theorem \ref{theorem: intro rho selmer}, which we focus on now. 
Using a classical geometric construction going back to Pantazis, we may reduce to estimating the size of $\Sel_{\rhodual}\Prym_b^{\vee}$ instead.
A construction in invariant theory which we call the `resolvent binary quartic' allows us to embed $\Sel_{\rhodual}\Prym^{\vee}_b$ in the set of $\PGL_2(\Q)$-orbits of binary quartic forms with rational coefficients.
Counting orbits of integral binary quartic forms using the techniques of \cite{BS-2selmerellcurves} and modifying the local conditions leads to the determination of the average size of $\Sel_{\rhodual}\Prym^{\vee}_b$, proving Theorem \ref{theorem: intro rho selmer} and consequently Theorem \ref{theorem: intro 2 selmer}.


We end this introduction by discussing some limitations, questions and remarks. 
We only obtain an upper bound in Theorem \ref{theorem: intro 2 selmer} because we are unable to prove a uniformity estimate similar to \cite[Theorem 2.13]{BS-2selmerellcurves} hence we cannot apply the so-called square-free sieve to obtain an equality in Theorem \ref{theorem: counting infinitely many congruence conditions rho case}. We expect that a similar such estimate holds and that the average size of $\Sel_2 \Prym_b$ equals $5$.
For proving an equality in Theorem \ref{theorem: intro rho selmer}, we bypassed proving such a uniformity estimate by reducing it to the one established by Bhargava and Shankar \cite[Theorem 2.13]{BS-2selmerellcurves}.
The crucial ingredient for this reduction step is Corollary \ref{corollary: squarefree implies unramified selmer condition rhovee descent} which is based on a detailed analysis of N\'eron component groups of certain Prym varieties in \S\ref{subsection: neron component groups pryms}.

The fact that the $\rhodual$-Selmer group of $\Prym_b^{\vee}$ (and so consequently, by the `bigonal construction' of Theorem \ref{theorem: summary bigonal construction}, the $\rho$-Selmer group of $\Prym_b$) has an interpretation in terms of binary quartic forms (Theorem \ref{theorem: inject rho-descent orbits}) might be of independent interest. 
It seems conceivable that a further analysis would make the computation of $\Sel_{\rho}\Prym_b$ possible using binary quartic forms, similar to the computation of the $2$-Selmer group of an elliptic curve.

We compare our results with the heuristics of Poonen and Rains \cite{PoonenRains-maximalisotropic}, which provide a framework for statistics of Selmer groups using random matrix models.
The self-dual isogeny $\rho\colon \Prym_b\rightarrow \Prym^{\vee}_b$ is defined by a symmetric line bundle, so\footnote{It is the group $\Sel_{\rho}\Prym_b/\Sha(\Q,\Prym_b[\rho])$ that is the intersection of two maximal isotropic subspaces, but because $\Prym_b[\rho] \simeq \ellcurve_b[2]$, $\Sha(\Q,\Prym_b[\rho])$ vanishes by \cite[Proposition 3.3]{PoonenRains-maximalisotropic}.} \cite[Theorem 4.13]{PoonenRains-maximalisotropic} shows that $\Sel_{\rho}\Prym_b$ is the intersection of two maximal isotropic subspaces of an infinite-dimensional quadratic space over $\F_2$.
It is therefore natural to ask whether the distribution of $\#\Sel_{\rho}\Prym_b$ coincides with the one modelling $2$-Selmer groups of elliptic curves (Conjecture 1.1 of op. cit.); Theorem \ref{theorem: intro rho selmer} provides evidence for this.
On the other hand, the isogeny $[2]:\Prym_b \rightarrow \Prym_b$ is not self-dual and a different type of matrix model is needed. 

\begin{remark}\label{remark: moduli interpretation curves}
    The families of curves considered here have a moduli interpretation. Loosely speaking, Equation (\ref{equation: intro ell curve}) defines the universal family of elliptic curves with a marked line in its Weierstrass embedding (here given by intersecting with the line $\{y = 0\}$) not meeting the origin $\infty$, and Equation (\ref{equation: intro f4 equation}) describes the double cover of this elliptic curve branched along the marked line and $\infty$. 
\end{remark}

\begin{remark}
Stable gradings on nonsimply laced Lie algebras have played an implicit role before in arithmetic statistics. In \cite{BhargavaElkiesShnidman}, the authors study the $3$-isogeny Selmer group of the family of cubic twist elliptic curves $y^2= x^3+k$. They use a representation associated to a $\Z/3\Z$-grading on a Lie algebra of type $G_2$. 
This forms the starting point of the previously cited results of \cite{BhargavaKlagsbrunZevOliver}, so graded Lie algebras play a role there too.
\end{remark}

\begin{remark}
	Bhargava and Ho have studied the representation $\smallV$ before in the context of invariant theory of genus-$1$ curves (cf. Entry 10 of \cite[Table 1]{BhargavaHo-coregularspacesgenusone}). It would be interesting to relate their geometric constructions to ours, and to see how the Prym variety fits in their description.
\end{remark}

\subsection{Organization}

In \S\ref{section: representation theory} we define the representation $(\smallG,\smallV)$, summarize its invariant theory and describe it explicitly. Moreover we describe the resolvent binary quartic of an element of $\smallV$.
In \S\ref{section: geometry}, we start by establishing a link between stable orbits in $\smallV$ and the family of curves $\smallprojcurve \rightarrow \smallB$. Then we introduce the family of Prym varieties $\Prym \rightarrow \smallB$ and study its geometry.
The construction of orbits associated with Selmer elements is the content of \S\ref{section: orbit parametrization}. We start by embedding the $2$-Selmer group inside the space of rational orbits of the representation $\smallV$. 
We then define a new representation $\rhoV$ of $\rhoG$ (very closely related to binary quartic forms) and embed the $\rhodual$-Selmer group inside the space of rational orbits of $\rhoV$. 
In \S\ref{section: integral orbit representatives}, we prove that orbits coming from Selmer elements admit integral representatives away from small primes. 
Then we count integral orbits of $\smallV$ and $\rhoV$ using geometry-of-numbers techniques in \S\ref{section: counting orbits in V} and \S\ref{section: counting orbits in rhoV} respectively.
Finally in \S\ref{section: proof of the main theorems} we combine all of the above ingredients and prove Theorems \ref{theorem: intro rho selmer} and \ref{theorem: intro 2 selmer}.

\subsection{Acknowledgements}
This research has been carried while the author was a PhD student under the supervision of Jack Thorne. 
	I want to thank him for suggesting the problem, providing many invaluable suggestions and his constant encouragement. 
	I am also grateful to Beth Romano for useful discussions.
	Finally, I would like to thank the anonymous referee for their helpful comments.
	This project has received funding from the European Research Council (ERC) under the European Union’s Horizon 2020 research and innovation programme (grant agreement No. 714405).

\subsection{Notation}

For a field $k$ we write $\bar{k}$ for a fixed algebraic closure of $k$.

If $X$ is a scheme over $S$ and $T\rightarrow S$ a morphism we write $X_T$ for the base change of $X$ to $T$. If $T = \Spec A$ is an affine scheme we also write $X_A$ for $X_T$. We write $X(S)$ for the set of sections of the structure map $X\rightarrow S$ and $X(T) = X_T(T)$. 

If $\lambda\colon A\rightarrow B$ is a morphism between group schemes we write $A[\lambda]$ for the kernel of $\lambda$.

If $T$ is a torus over a field $k$ and $V$ a representation of $T$, we write $\Phi(V,T)\subset X^*(T)$ for the set of weights of $T$ on $V$. 
If $H$ is a group scheme over $k$ containing $T$, we write $\Phi(H,T)$ for $\Phi(\Ad H, T)$, where $\Ad H$ denotes the adjoint representation of $H$.

If $G$ is a smooth group scheme over $S$ we write $\HH^1(S,G)$ for the set of isomorphism classes of \'etale sheaf torsors under $G$ over $S$, which is a pointed set coming from non-abelian \v{C}ech cohomology. If $S = \Spec R$ we write $\HH^1(R,G)$ for the same object. If $G\rightarrow S$ is affine then every sheaf torsor under $G$ is representable by a scheme.

If $G\rightarrow S$ is a group scheme acting on $X\rightarrow S$ and $x \in X(T)$ is a $T$-valued point, we write $Z_G(x) \rightarrow T$ for the centralizer of $x$.
If $x$ is an element of a Lie algebra $\lieh$, we write $\mathfrak{z}_{\lieh}(x)$ for the centralizer of $x$, a subalgebra of $\lieh$.

A $\Z/2\Z$-\define{grading} on a Lie algebra $\lieh$ over a field $k$ is a direct sum decomposition 
$$\lieh = \bigoplus_{i\in \Z/2\Z} \lieh(i) $$
of linear subspaces of $\lieh$ such that $[h(i),h(j)] \subset \lieh(i+j)$ for all $i,j \in \Z/2\Z$. 
If $2$ is invertible in $k$ then giving a $\Z/2\Z$-grading is equivalent to giving an involution of $\lieh$.

If $V$ is a finite free $R$-module over a ring $R$ we write $R[V]$ for the graded algebra $\Sym(R^{\vee})$. Then $V$ is naturally identified with the $R$-points of the scheme $\Spec R[V]$, and we call this latter scheme $V$ as well.
If $G$ is a group scheme over $R$ we write $V \GIT G\coloneqq \Spec R[V]^G$ for the \define{GIT quotient} of $V$ by $G$.

\begin{table}
\centering
\begin{tabular}{|c | c | c |}
	\hline
	   Symbol & Description & Reference in paper \\
	\hline       
		$\smallH$ & Split adjoint group of type $F_4$ & \S\ref{subsection: definition of the representations} \\
		$\smalltheta$ & Stable involution of $\smallH$ & \S\ref{subsection: definition of the representations} \\
		$\smallG$ & Fixed points of $\smalltheta$ on $\smallH$ & \S\ref{subsection: definition of the representations} \\
		$\smallV$ & $(-1)$-part of action of $\smalltheta$ on $\smallh$ & \S\ref{subsection: definition of the representations}\\
		$\smallB$ & GIT quotient $\smallV\GIT \smallG$ & \S\ref{subsection: definition of the representations} \\
		$\Delta \in \Q[\smallB]$ & Discriminant polynomial & \S\ref{subsection: definition of the representations} \\
		$\pi: \smallV \rightarrow \smallB$ & Invariant map & \S\ref{subsection: definition of the representations} \\
		$\sigma: \smallB \rightarrow \smallV$ & Kostant section & \S\ref{subsection: distinguished orbit} \\
		$\bigH$ & Split adjoint group of type $E_6$ & \S\ref{subsection: definition of the representations} \\
		$\zeta: \bigH \rightarrow \bigH$ & Pinned automorphism of $\bigH$ & \S\ref{subsection: definition of the representations} \\
		$\bigtheta, \bigG,\bigV$ & Analogous objects of $\bigH$ & \S\ref{subsection: definition of the representations} \\
		$\bigB, \bigpi, \bigsigma$ & Analogous objects of $\bigH$ & \S\ref{subsection: definition of the representations},\ref{subsection: distinguished orbit}  \\
		$Q_v$ & Resolvent binary quartic of $v\in \smallV$ & \S\ref{subsection: the resolvent binary quartic} \\
		$p_2,p_6,p_8, p_{12}$ & $\smallG$-invariant polynomials of $\smallV$ & \S\ref{subsection: a family of curves}  \\
		$\smallprojcurve \rightarrow \smallB$ & Family of projective curves &  \S\ref{subsection: a family of curves},  Eq.(\ref{equation : F4 family middle of paper}) \\
		$\tau \colon \smallprojcurve \rightarrow \smallprojcurve$ & Involution $(x,y)\mapsto (x,-y)$ &  \S\ref{subsection: a family of curves} \\
		$\Jac \rightarrow \smallB^{\rs}$ & Jacobian variety of $\smallprojcurve^{\rs} \rightarrow \smallB^{\rs}$ & \S\ref{subsection: a family of curves} \\
		$\Lambda, W_{\mathrm{E}}$ & $E_6$ root lattice and its Weyl group & \S \ref{subsection: monodromy of J[2]}\\
		$\Cellcurve \rightarrow \smallB$ & Quotient of $\smallprojcurve$ by $\tau$ & \S\ref{subsection: def prym variety}, Eq.(\ref{equation: elliptic curve}) \\
		$\Prym \rightarrow \smallB^{\rs}$ & Prym variety of the cover $\smallprojcurve^{\rs} \rightarrow \ellcurve$ & \S\ref{subsection: def prym variety} \\
		$\rho\colon \Prym \rightarrow \Prym^{\vee}$ & Polarization of type $(1,2)$ & \S\ref{subsection: def prym variety} \\
		$\chi\colon \smallB\rightarrow \smallB$ & Automorphism arising from bigonal construction & \S\ref{subsection: the bigonal construction} \\
				$\hat{X}$ & Pullback of a $\smallB$-scheme $X$ along $\chi$ & \S\ref{subsection: the bigonal construction} \\
				$\CPrym \rightarrow \smallB$ & Compactified Prym variety & \S\ref{subsection: compactifications} \\
		$\rhoG$ & $\PGL_2$ over $\Q$ & \S\ref{subsection: the representation rhoV} \\
		$\rhoV$ & $\rhoG$-representation $\Q\oplus \Q\oplus \Sym^4(2)$ & \S\ref{subsection: the representation rhoV} \\
		$\rhoB$ & GIT quotient $\rhoV\GIT \rhoG $ & \S\ref{subsection: the representation rhoV} \\
		$\bigresolv \colon \smallV \rightarrow \rhoV$ & Map $v\mapsto (p_2(v),p_6(v),Q_v)$ & \S\ref{subsection: the representation rhoV} \\
		$S$ & $\Z[1/N]$, where $N$ sufficiently large integer   &\S\ref{subsection: integral structures} \\
		$\intsmallH, \intsmallG, \intsmallV, \intsmallB, \intrhoG, \dots$ & Extensions of above objects over $\Z$ &\S\ref{subsection: integral structures} \\
		$\intsmallcurve \rightarrow \intsmallB$, $\intCellcurve \rightarrow \intsmallB$ & Extension of $\smallprojcurve$ and $\Cellcurve$ over $\Z$ &  \S\ref{subsection: integral structures}\\
		$\intJac \rightarrow \intsmallB_S^{\rs}$ & Jacobian of $\intsmallcurve^{\rs}_S \rightarrow \intsmallB^{\rs}_S$ & \S\ref{subsection: integral structures} \\
		$\intPrym \rightarrow \intsmallB_S^{\rs}$ & Prym variety of $\intsmallcurve_S^{\rs} \rightarrow \intellcurve$ & \S\ref{subsection: integral structures} \\

	\hline
\end{tabular}
\caption{Notation used throughout the paper}
\label{table 1}
\end{table}

\section{Representation theory} \label{section: representation theory}

\subsection{Definition of the representation \texorpdfstring{$\smallV$}{V}}\label{subsection: definition of the representations}

In this section we define the pair $(\smallG,\smallV)$ using a $\Z/2\Z$-grading on a Lie algebra of type $F_4$. We will define it by embedding it in a larger representation $(\bigG,\bigV)$ defined in \cite[\S2.1]{Laga-E6paper} using a $\Z/2\Z$-grading on a Lie algebra of type $E_6$, which we recall first. Objects related to $\bigV$ will usually denoted by a subscript $(-)_{\mathrm{E}}$.

Let $\bigH$ be a split adjoint semisimple group of type $E_6$ over $\Q$ with Lie algebra $\bigh$. We suppose that $\bigH$ comes with a pinning $(\bigT,\bigP,\{Y_{\alpha}\})$. So $\bigT \subset \bigH$ is a split maximal torus (which determines a root system $\Phi(\bigH,\bigT) \subset X^*(\bigT)$), $\bigP\subset \bigH$ is a Borel subgroup containing $\bigT$ (which determines a root basis $S_{\bigH} \subset \Phi(\bigH,\bigT)$) and $Y_{\alpha}$ is a generator for each root space $(\bigh)_{\alpha}$ for $\alpha \in S_{\bigH}$. The group $\bigH$ is of dimension $78$.

Let $\check{\rho}_{\mathrm{E}}\in X_*(\bigT)$ be the sum of the fundamental coweights with respect to $S_{\bigH}$, defined by the property that $\langle \check{\rho}_{\mathrm{E}},\alpha \rangle = 1$ for all $\alpha \in S_{\bigH}$. 
Write $\zeta\colon \bigH \rightarrow \bigH$ for the unique nontrivial automorphism preserving the pinning: it is an involution inducing the order-$2$ symmetry of the Dynkin diagram of $E_6$. Let $$\bigtheta \coloneqq \zeta \circ \Ad(\check{\rho}_{\mathrm{E}}(-1)) = \Ad(\check{\rho}_{\mathrm{E}}(-1)) \circ\zeta.$$ Then $\bigtheta$ defines an involution of $\bigh$ and thus by considering $(\pm1)$-eigenspaces it determines a $\Z/2\Z$-grading 
$$\bigh = \bigh(0) \oplus \bigh(1).$$
Let $\bigG \coloneqq \bigH^{\bigtheta}$ be the centralizer of $\bigtheta$ in $\bigH \subset \Aut(\bigh)$ and write $\bigV\coloneqq\bigh(1)$; the space $\bigV$ defines a representation of $\bigG$ and its Lie algebra $\bigg$ by restricting the adjoint representation. The pair $(\bigG,\bigV)$ has been studied extensively in \cite{Laga-E6paper}.

We now consider the $\zeta$-fixed points of the above objects.
Let $\smallH \coloneqq \bigH^{\zeta}$ and $\smallh \coloneqq \bigh^{\zeta}$. Then $\smallH$ is a split adjoint semisimple group of type $F_4$ with Lie algebra $\smallh$, and the pinning of $\bigH$ induces a pinning of $\smallH$, cf. \cite[\S3.1]{Reeder-torsion}. Indeed, $\smallT \coloneqq \bigT^{\zeta}$ is a split maximal torus and $\smallP\coloneqq \bigP^{\zeta}$ is a Borel subgroup containing $\smallT$.
They determine a root system $\Phi(\smallH,\smallT) \subset X^*(\smallT)$ and a root basis $S_{\smallH} \subset \Phi(\smallH,\smallT)$ respectively. 
The natural map $X^*(\bigT) \rightarrow X^*(\smallT)$ restricts to a surjection $S_{\bigH} \rightarrow S_{\smallH}$ where two different elements $\beta,\beta'\in S_{\bigH}$ define the same element of $S_{\smallH}$ if and only if $\beta' = \zeta(\beta)$. 
(The map $S_{\bigH}\rightarrow S_{\smallH}$ can be seen as `folding' the $E_6$ Dynkin diagram alluded to in the introduction.)
If $\alpha\in S_{\smallH}$ we write $[\alpha]$ for its inverse image in $S_{\bigH}$ under this map, and we define $X_{\alpha} \coloneqq \sum_{\beta\in [\alpha]} Y_{\beta} \in \smallh_{\alpha}$.
Then the triple $(\smallT,\smallP,\{X_{\alpha}\})$ is a pinning of $\smallH$. 
Since $\bigtheta$ commutes with $\zeta$, the restriction $\smalltheta \coloneqq \bigtheta|_{\smallH}$ defines an involution $\smallH \rightarrow \smallH$. 
We have $\smalltheta = \Ad \check{\rho}(-1)$, where $\check{\rho}\in X_*(\smallT)$ is the sum of the fundamental coweights with respect to $S_{\smallH}$. 
As before this determines a $\Z/2\Z$-grading
$$\smallh = \smallh(0) \oplus \smallh(1).$$
Let $\smallG = \smallH^{\smalltheta}$ be the centralizer of $\smalltheta$ in $\smallH$ and write $\smallV\coloneqq\smallh(1)$. Again $\smallV$ defines a representation of $\smallG$ and its Lie algebra $\smallg$.
The pair $(\smallG,\smallV)$ is the central object of study in this paper. We summarize some of its basic properties here.

\begin{proposition}\label{proposition: table properties algebraic groups}
	The groups $\bigG,\smallH,\smallG$ are split connected semisimple groups over $\Q$ with maximal torus $\smallT$. Their properties are listed in Table \ref{table 2}. The vector spaces $\bigV$ and $\smallV$ have dimension $42$ and $28$ respectively. 
\end{proposition}

\begin{table}
	\centering
	\begin{tabular}{|c|c c c |}
		\hline
		Group & Type & Isomorphism class & Dimension\\
		\hline
		$\bigH$ & $E_6$ & Adjoint & $72$ \\
		$\bigG$ & $C_4$ & $\PSp_8$ & $36$ \\
		$\smallH$ & $F_4$ & Adjoint  & $52$  \\
		$\smallG$ & $C_3\times A_1$ & $\left(\Sp_6\times \SL_2\right)/\mu_2$ & $24$\\
		\hline
	\end{tabular}
	\caption{Properties of the semisimple groups.}
	\label{table 2}
\end{table}

\begin{proof}
    The properties of $\smallH$ follow from \cite[Lemma 3.1]{Reeder-torsion}.
	The isomorphism class of $\bigG$ and $\smallG$ over $\overbar{\Q}$ can be deduced from the analysis of the Ka\v{c} diagrams of the automorphisms $\bigtheta$, $\smalltheta$ given in \cite[\S7.1, Tables 2 and 6]{GrossLevyReederYu-GradingsPosRank}, using the results of \cite{Reeder-torsion}. (The notation $(\Sp_6\times \SL_2)/\mu_2$ means the quotient of $\Sp_6\times \SL_2$ by the diagonally embedded $\mu_2$ in the center.) These groups are split since $\smallT$ is a split torus of maximal rank. 
\end{proof}

The next proposition concerns the invariant theory of the pair $(\smallG,\smallV)$ and shows that regular semisimple orbits over algebraically closed fields are well understood. 
For a field $k/\Q$, we say $v\in \smallV(k)$ is \define{regular, nilpotent, semisimple} respectively if it is so when considered as an element of $\smallh(k)$.

\begin{proposition} \label{prop : graded chevalley}
	Let $k/\Q$ be a field. The following properties are satisfied:
	\begin{enumerate}
		\item $\smallV_k$ satisfies the Chevalley restriction theorem: if $\mathfrak{a} \subset \smallV_k$ is a Cartan subalgebra, then the map $N_{\smallG}(\mathfrak{a}) \rightarrow W_{\mathfrak{a}} \coloneqq N_{\smallH}(\mathfrak{a})/Z_{\smallH}(\mathfrak{a})$ is surjective, and the inclusions $\mathfrak{a} \subset \smallV_k\subset \smallh_k$ induce isomorphisms 
		$$\mathfrak{a}\GIT W_{\mathfrak{a}} \simeq \smallV_k\GIT \smallG \simeq \smallh_k \GIT \smallH .$$
		In particular, the quotient is isomorphic to affine space. 
		\item Suppose that $k$ is algebraically closed and let $x,y\in \smallV(k)$ be regular semisimple elements. Then $x$ is $\smallG(k)$-conjugate to $y$ if and only if $x,y$ have the same image in $\smallV\GIT \smallG$. 
	\end{enumerate}
\end{proposition}

\begin{proof}
	These are classical results in the invariant theory of graded Lie algebras due to Vinberg and Kostant--Rallis; we refer to \cite[\S2]{Thorne-thesis} for precise references. 
\end{proof}

We now give some alternative characterizations of regular semisimple elements in $\smallV$, after introducing some more notation. 
First recall that the \define{discriminant} of $\smallh$ is the image under the Chevalley isomorphism $\Q[\smallt]^{W(\smallH,\smallT)} \rightarrow \Q[\smallh]^{\smallH}$ of the product of all roots $\alpha \in \Phi(\smallH,\smallT)$, where $\smallt\coloneqq\Lie\smallT$. 
Write $\Delta\in \Q[\smallV]^{\smallG}$ for its restriction to $\smallV \subset \smallh$. 
Next we introduce weights of one-parameter subgroups. 
If $k/\Q$ is a field and $\lambda\colon \G_m \rightarrow \smallG_{k}$ a homomorphism, we may decompose $\smallV(k)$ as $\oplus_{i\in \Z} \smallV_i$ where $\smallV_i = \{v\in \smallV(k)\mid \lambda(t)\cdot v = t^iv  \}$. 
Every $v\in \smallV(k)$ can be written as $v = \sum v_i$ and we call integers $i$ with $v_i \neq 0$ the \define{weights of $v$ with respect to $\lambda$}. 

\begin{proposition}\label{proposition: equivalences regular semisimple}
	Let $k/\Q$ be field and $v\in \smallV(k)$. Then the following are equivalent:
	\begin{enumerate}
		\item $v$ is regular semisimple.
		\item $\Delta(v) \neq 0$.
		\item The $\smallG$-orbit of $v$ is closed in $\smallV$ and $Z_{\smallG}(v)$ is finite (i.e. $v$ is stable in the sense of geometric invariant theory). 
		\item For every nontrivial homomorphism $\lambda \colon \G_m \rightarrow \smallG_{\bar{k}}$, $v$ has a negative weight with respect to $\lambda$. 
	\end{enumerate}
\end{proposition}
\begin{proof}
	The equivalence between the first two properties is a well known property of the discriminant.  
	The first property implies the third by \cite[Proposition 2.8]{Thorne-thesis}, and the converse follows from \cite[Lemma 5.6]{GrossLevyReederYu-GradingsPosRank}.
	Finally, the equivalence between the last two properties is the content of the Hilbert-Mumford stability criterion \cite{Mumford-stabilityprojvarieties}.
\end{proof}


We write $\smallB \coloneqq \smallV\GIT \smallG = \Spec \Q[\smallV]^{\smallG}$, $\bigB \coloneqq \bigV\GIT \bigG = \Spec \Q[\bigV]^{\bigG}$ and $\smallpi \colon \smallV \rightarrow \smallB$, $\bigpi\colon \bigV \rightarrow \bigB$ for the natural quotient maps.
Scaling defines $\G_m$-actions on $\smallV$ and $\bigV$, and there are unique $\G_m$-actions on $\smallB$ and $\bigB$ such that the morphisms $\smallpi$ and $\bigpi$ are $\G_m$-equivariant. In \S\ref{subsection: a family of curves} we will describe the weights of $\smallB$ and $\bigB$.


\subsection{The distinguished orbit}\label{subsection: distinguished orbit}

We describe a section of the quotient map $\smallpi \colon \smallV \rightarrow \smallB$ whose construction is originally due to Kostant.
Let $E \coloneqq \sum_{\alpha \in S_{\bigH}} Y_{\alpha} \in \bigh$. Then $E$ is a regular nilpotent element of $\bigh$ which lies in $\smallh(1)$. 
Using \cite[Proposition 2.7]{Thorne-thesis}, there exists a unique normal $\liesl_2$-triple $(E,X,F)$ containing $E$. By definition, this means that $(E,X,F)$ satisfies the identities
\begin{equation*}
[X,E] = 2E , \quad [X,F] = -2F ,\quad [E,F] = H, 
\end{equation*}
with the additional property that $X\in \bigh(0)$ and $F \in \bigh(1)$. Since $(E,\zeta(X),\zeta(F))$ is also a normal $\liesl_2$-triple containing $E$, we see that $(E,\zeta(X),\zeta(F)) = (E,X,F)$ hence $X$ and $F$ lie in $\smallh$. 

We define affine linear subspaces $\bigkappa \coloneqq E +\mathfrak{z}_{\bigh}(F)  \subset \bigV$ and $\smallkappa \coloneqq \bigkappa^{\zeta} = E+\mathfrak{z}_{\smallh}(F) \subset \smallV$. 
\begin{proposition}\label{proposition: kostant section}
    \begin{enumerate}
        \item The composite maps $\smallkappa \hookrightarrow \smallV\rightarrow \smallB$ and $\bigkappa \hookrightarrow \bigV\rightarrow \bigB$ are isomorphisms. 
        \item $\smallkappa$ and $\bigkappa$ are contained in the open subscheme of regular elements of $\smallV$ and $\bigV$ respectively.
        \item The morphisms $\smallG \times \smallkappa \rightarrow \smallV, (g,v) \mapsto g\cdot v$ and $\bigG \times \bigkappa \rightarrow \bigV, (g,v) \mapsto g\cdot v$ are \'etale.
    \end{enumerate}
\end{proposition}
\begin{proof}
    Parts 1 and 2 are \cite[Lemma 3.5]{Thorne-thesis}; the last part is \cite[Proposition 3.4]{Thorne-thesis}. (These facts are stated only for simply laced groups in \cite{Thorne-thesis} but they remain valid in the $F_4$ case by the same proof.)
\end{proof}

Write $\smallsigma: \smallB \rightarrow \smallV$ for the inverse of $\smallpi|_{\smallkappa}$ and $\bigsigma: \bigB\rightarrow \bigV$ for the inverse of $\bigpi|_{\bigkappa}$. We call $\smallsigma$ the \define{Kostant section} for the pair $(\smallG,\smallV)$. It determines, for every field $k/\Q$ and $b\in \smallB(k)$, a distinguished orbit in $\smallV(k)$ with invariants $b$, playing an analogous role to reducible binary quartic forms as studied in \cite{BS-2selmerellcurves}. 
It will be used to organize the set of $\smallG(k)$-orbits of $\smallV(k)$.
\begin{definition}\label{definition: k-reducible orbit}
Let $k/\Q$ be a field and $v\in \smallV(k)$.
We say $v$ is \define{$k$-reducible} if $v$ is not regular semisimple or $v$ is $\smallG(k)$-conjugate to $\smallsigma(b)$ with $b = \smallpi(v)$. 
\end{definition}

If $k/\Q$ is algebraically closed, every element of $\smallV(k)$ is $k$-reducible by Proposition \ref{prop : graded chevalley}.

\subsection{The action of \texorpdfstring{$\zeta$}{zeta} on \texorpdfstring{$\bigB$}{B}}

The involution $\zeta: \bigV \rightarrow \bigV$ induces an involution $\bigB \rightarrow \bigB$, still denoted by $\zeta$.

\begin{proposition}\label{proposition: action zeta on B}
    \begin{enumerate}
        \item The inclusion $\smallV\subset \bigV$ induces a closed embedding $\smallB\hookrightarrow \bigB$ whose image is the subset of $\zeta$-fixed points of $\bigB$. 
        \item The involution $\zeta\colon \bigB \rightarrow \bigB$ coincides with the involution $(-1)\colon \bigB \rightarrow \bigB$ induced by the $\G_m$-action on $\bigB$.
        \item Let $k/\Q$ be a field and $v\in \smallV(k)$. Then $v$ is regular semisimple as an element of $\smallh(k)$ if and only if $v$ is regular semisimple as an element of $\bigh(k)$.
    \end{enumerate}
\end{proposition}
\begin{proof}
    Because the inclusion $\smallV\subset \bigV$ restricts to the inclusion $\smallkappa \subset \bigkappa$, the first claim follows from Part 1 of Proposition \ref{proposition: kostant section}.
    
    To prove the second claim, recall that $\bigT\subset \bigH$ denotes a split maximal torus. Write $\bigt\subset \bigh$ for its Lie algebra and $W_{\mathrm{E}}$ for its Weyl group. 
	By the classical Chevalley restriction theorem and Proposition \ref{prop : graded chevalley} respectively, the inclusions $\bigt \hookrightarrow \bigh$, $\bigV \hookrightarrow \bigh$ induce isomorphisms $\bigt\GIT W_{\mathrm{E}} \simeq \bigh\GIT \bigH$, $\bigB\simeq \bigh\GIT \bigH$, equivariant with respect to the actions of $\G_m$ and $\zeta$.
	So it suffices to prove that the action of $\zeta$ on $\bigt\GIT W_{\mathrm{E}}$ is given by $-1$. 
	Since $\zeta$ and $-1$ are not contained in $W_{\mathrm{E}}$ and this group has index $2$ in $N_{\bigG}(\bigt)$, the product $-\zeta$ lies in $W_{\mathrm{E}}$.
	Therefore $-\zeta$ acts trivially on $\bigt\GIT W_{\mathrm{E}}$, as desired.
	
	To prove the third claim, we may assume that $k$ is algebraically closed and after conjugating by $\smallH(k)$ that $v\in \smallt(k) = \bigt^{\zeta}(k)$.
	Then $v$ is regular semisimple as an element of $\smallh(k)$ if and only if $d\alpha(v)\neq 0$ for all $\alpha\in \Phi(\smallH,\smallT)$, and $v$ is regular semisimple as an element of $\bigh(k)$ if and only if $d\alpha(v)\neq 0$ for all $\alpha\in \Phi(\bigH,\bigT)$.
	These two statements are equivalent because the restriction map $\Phi(\bigH,\bigT)\rightarrow \Phi(\smallH,\smallT)$ is surjective.
\end{proof}

\subsection{An explicit description of \texorpdfstring{$\smallV$}{V}}\label{subsection: an explicit description of V}

In this section we give an explicit description of $\smallV$ which will be convenient for performing computations in \S\ref{subsection: the resolvent binary quartic}, \S\ref{subsection: a criterion for reducibility} and \S\ref{subsection: cutting off cusp}.
Recall from \S\ref{subsection: definition of the representations} that $\smallh$ is a Lie algebra of type $F_4$ and that there is a direct sum decomposition $\smallh = \smallg \oplus \smallV$ where $\smallg \simeq \liesp_6 \oplus \liesl_2$ and $\smallV$ is a $28$-dimensional representation of $\smallg$. 
The split maximal torus $\smallT \subset \smallH$ gives rise to three subsets of $X^*(\smallT)$: $\Phi(\smallH,\smallT)$, $\Phi(\smallG,\smallT)$ and $\Phi(\smallV,\smallT)$. They will be denoted by $\Phi_{\smallH}$, $\Phi_{\smallG}$ and $\Phi_{\smallV}$ and satisfy $\Phi_{\smallH} = \Phi_{\smallG}\sqcup\Phi_{\smallV}$.
Using the root basis $S_{\smallH}$ fixed in \S\ref{subsection: definition of the representations}, $\Phi_{\smallV}$ (resp. $\Phi_{\smallG}$) consists of those roots in $\Phi_{\smallH}$ which have odd root height (resp. even root height). 

Following Bourbaki \cite[Planche VIII]{Bourbaki-Liealgebras}, we denote the elements of $S_{\smallH} = \{\alpha_1,\alpha_2,\alpha_3,\alpha_4\}$ according to the following labeling of the nodes of the Dynkin diagram:
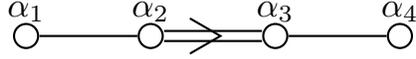
\begin{center}
	\begin{tikzpicture}[transform shape, scale=1.3]
	\node[root] (b) {};
	\node[root] (c) [right=of b] {};
	\node[root] (d) [right=of c] {};
	\node[root] (e) [right=of d] {};
	\node [above] at (b.north) {$\alpha_1$};
	\node [above] at (c.north) {$\alpha_2$};
	\node [above] at (d.north) {$\alpha_3$};
	\node [above] at (e.north) {$\alpha_4$};
	\draw[thick] (b) -- (c) ;
	\draw[thick] (d) -- (e);
	\draw[doublearrow] (c)--(d);
	\end{tikzpicture}
\end{center}

Define $\beta_1, \beta_2, \beta_3, \beta_4$ to be $\alpha_2+\alpha_3, \alpha_3+\alpha_4, \alpha_1+\alpha_2, \alpha_1+\alpha_2+2\alpha_3$ respectively. 
Then $S_{\smallG} \coloneqq \{\beta_1,\beta_2,\beta_3,\beta_4\}$ is a root basis of $\Phi_{\smallG}$, according to the following labelling of the Dynkin diagram of type $C_3\times A_1$:
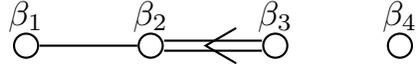
\begin{center}
	\begin{tikzpicture}[transform shape, scale=1.3]
	\node[root] (b) {};
	\node[root] (c) [right=of b] {};
	\node[root] (d) [right=of c] {};
	\node[root] (e) [right=of d] {};
	\node [above] at (b.north) {$\beta_1$};
	\node [above] at (c.north) {$\beta_2$};
	\node [above] at (d.north) {$\beta_3$};
	\node [above] at (e.north) {$\beta_4$};
	\draw[thick] (b) -- (c) ;
	\draw[doublearrow] (d)--(c);
	\end{tikzpicture}
\end{center}

With respect to this root basis the positive roots of $\Phi_{\smallG}$, denoted $\Phi_{\smallG}^+$, are given by 
\begin{displaymath}
\{ \beta_1, \beta_2, \beta_3,\beta_1+\beta_2,\beta_2+\beta_3,2\beta_2+\beta_3,\beta_1+\beta_2+\beta_3,\beta_1+2\beta_2+\beta_3,2\beta_1+2\beta_2+\beta_3  \} \cup  \{\beta_4\}.
\end{displaymath}
Another basis of $X^*(\smallT)\otimes 
{\Q}$ will be convenient for describing $\Phi_{\smallG}$ and $\Phi_{\smallV}$.
We define 
\begin{displaymath}
\begin{cases}
L_1 = (2\beta_1+2\beta_2+\beta_3)/2, \\
L_2 = (2\beta_2+\beta_3)/2, \\
L_3 = \beta_3/2,\\
L_4 = \beta_4/2.\\
\end{cases}
\end{displaymath}
Then $S_{\smallG} = \{L_1-L_2,L_2-L_3,2L_3,2L_4 \}$ and 
the elements of $\Phi_{\smallG}$ are given by 
\begin{equation*}
	\{\pm L_i \pm L_j \mid 1\leq i, j \leq 3 \} \cup \{\pm 2L_4 \}.
\end{equation*}

Using the above explicit description or a general recipe applied to the Ka\v{c} diagram of $\smalltheta$ (given in \cite[\S7.1,Table 6]{GrossLevyReederYu-GradingsPosRank}), we see that $\smallV$ is isomorphic to $\smallW\boxtimes (2)$ where $\smallW$ is a $14$-dimensional irreducible representation of $\liesp_6$ with highest weight $L_1+L_2+L_3$ (we will choose an explicit realization of this representation in a moment) and $(2)$ denotes the standard representation of $\liesl_2$. 
The elements of $\Phi_{\smallV}$ are of the form $x\pm L_4$, where $x$ is any element of the set 
\begin{displaymath}
\Phi_{\smallW} \coloneqq \{\pm L_i \mid i= 1,2,3  \} \cup \{\pm L_1 \pm L_2 \pm L_3 \}.
\end{displaymath}
Every element $\alpha \in X^*(\smallT) \otimes \Q$ has a unique expression of the form $\sum_i n_i(\alpha) \beta_i$ with $n_i(\alpha) \in \Q$. 
We define a partial ordering on $X^*(\smallT)\otimes \Q$ by declaring for $x,y\in X^*(\smallT)\otimes \Q$ that 
\begin{align}\label{equation: partial ordering}
	x\geq y \quad \text{if } \, n_i(x-y) \geq 0  \text{ for all }i= 1,\dots,4. 
\end{align}
This induces a partial ordering on $\Phi_{\smallV}$. 

We have tabulated the elements of $\Phi_{\smallV}$ in Table \ref{table 3}; the second column displays the coordinates of a weight in the basis $\{\beta_1/2,\beta_2/2,\beta_3/2,\beta_4/2\}$.
For example, the first entry is $\alpha_0 = L_1+L_2+L_3+L_4 = 2\alpha_1+3\alpha_2+4\alpha_3+2\alpha_4=(2\beta_1+4\beta_2+3\beta_3+\beta_4)/2 \in \Phi_{\smallV}$; it is the highest root of $\Phi_{\smallH}$ and the unique maximal element of $\Phi_{\smallV}$ with respect to the partial ordering. 

We now describe the $\Sp_6$-representation $\smallW$ explicitly following \cite[\S2.2]{IlievRanestad-GeometryGrassmanian}. 
Fix a vector space $\Q^6$ with standard basis $e_1, \dots, e_6$. 
We define $\Sp_6$ as the symplectic group stabilizing the $2$-form $\omega$ on $\Q^6$ given by the matrix 
$$
\begin{pmatrix}
	0 & I_3 \\
	-I_3 & 0
\end{pmatrix}.
$$
The form $\omega$ defines a $\Sp_6$-equivariant contraction map $\text{cont}_{\omega} \colon \bwedge{3}(\Q^6) \rightarrow \Q^6 , x_1 \wedge x_2 \wedge x_3 \mapsto \omega(x_2,x_3)-\omega(x_1,x_3)+\omega(x_1,x_2)$. 
Define $$\smallW \coloneqq \ker \text{cont}_{\omega} \subset \bwedge{3}(\Q^6).$$
We may organize an element $\sum c_{ijk} e_i \wedge e_j \wedge e_k \in \bwedge{3}\Q^6$ in the matrices:
\begin{displaymath}
	(u,X,Y,z) = \left(c_{123},\begin{pmatrix} c_{423} & c_{143} & c_{124} \\ c_{523} & c_{153} & c_{125} \\ c_{623} & c_{163} & c_{126} \end{pmatrix}, \begin{pmatrix} c_{156}  & c_{416} & c_{451} \\ c_{256} & c_{426} & c_{452} \\ c_{356} & c_{436} & c_{453} \end{pmatrix}, c_{456} \right).
\end{displaymath}
Then elements of $\smallW$ correspond to $4$-tuples $(u,X,Y,z)$ such that $X$ and $Y$ are \emph{symmetric} matrices.
An element of $\smallW$ will be usually thought of as such a $4$-tuple.

The ring of invariant polynomials $\Q[\smallW]^{\Sp_6}$ is freely generated by one degree-4 polynomial $F$ explicitly given by 
\begin{equation}\label{equation: invariant W}
	F(u,X,Y,z) \coloneqq \left(uz-\tr XY \right)^2+4u\det Y+4 z \det X - 4 \sum_{ij} \det(\hat{X}_{ij})\det(\hat{Y}_{ij}),
\end{equation}
where for a matrix $A$ we denote by $\hat{A}_{ij}$ the matrix obtained by crossing out the $i$th row and $j$th column. 

\begin{proposition}\label{proposition: orbits prehom W}
	Let $k/\Q$ be an algebraically closed field. 
	Then $\smallW(k)$ has finitely many $\Sp_6(k) \times k^{\times}$-orbits.
	Moreover:
	\begin{itemize}
		\item $\{w\in \smallW(k) \mid F(w)\neq 0\}$ is the unique open dense orbit.
		\item If $w\in \smallW(k)$ is nonzero with $F(w)=0$, then $w$ is $\Sp_6(k)\times k^{\times}$-conjugate to an element of the form 
	\begin{equation*}
	\left(1,\begin{pmatrix} * & 0 & 0 \\ 0 & * & 0 \\ 0 & 0 & * 	\end{pmatrix}, \begin{pmatrix} 0 & 0 & 0 \\ 0 & 0 & 0 \\ 0 & 0 & 0 	\end{pmatrix}, 0 \right).
	\end{equation*}
	\end{itemize}
\end{proposition}
\begin{proof}
	It is well-known that $\smallW(k)$ has finitely many $\Sp_6(k)\times k^{\times}$-orbits with $\{F\neq 0\}$ the unique open dense one; see \cite[\S2.3]{IlievRanestad-GeometryGrassmanian} for precise references.
	The description of the remaining orbits and Proposition 2.3.3 of loc. cit. implies the existence of the representatives above.
\end{proof}


We now fix the identifications of this subsection to remove any ambiguities.
There exists an isomorphism $\smallG \simeq (\Sp_6\times \SL_2)/\mu_2$ such that:
\begin{itemize}
	\item the weights $L_1, L_2, L_3$ correspond to the weights of $e_1, e_2, e_3$ in the defining representation of $\Sp_6$ (and $\SL_2$ acts trivially),
	\item the weight $L_4$ corresponds to the weight $\begin{pmatrix} t& 0 \\ 0 &t^{-1} \end{pmatrix} \mapsto t$ of $\SL_2$.
\end{itemize}
Then there exists a unique isomorphism $\smallV\simeq \smallW\boxtimes (2)$ of $\smallG$-representations which sends $X_{\alpha_1}\in \smallV_{\alpha_1}$ (part of the pinning of $\smallH$ fixed in \ref{subsection: definition of the representations}) to the element $(e_4\wedge e_2\wedge e_3,0)$.
This choice is somewhat arbitrary but what is important for us is that it preserves the `obvious' integral structures on both sides; this will be relevant in \S\ref{subsection: integral structures}.
We fix these isomorphisms for the remainder of the paper. 
It is therefore permitted, for every field $k/\Q$, to view an element $v\in \smallV(k)$ as a pair $(w_1,w_2)$ of elements of $\smallW(k)$, where $A\in \SL_2(k)$ acts on $(w_1,w_2)$ via $(w_1,w_2)\cdot A^{t}$.

\begin{table}
\centering
\begin{tabular}{|c|c c c c | l | l |}
\hline
$\#$ & \multicolumn{4}{c|}{Weights} & \multicolumn{1}{c|}{Basis $S_{\smallH}$} & \multicolumn{1}{c|}{Basis $\{L_i\}$  } \\
\hline
$1$ & $2$ & $4$ & $3$ & $1$ & $2 \alpha _1+3 \alpha _2+4 \alpha _3+2 \alpha _4$ & $L_1+L_2+L_3+L_4$ \\
$2$ & $2$ & $4$ & $1$ & $1$ & $\alpha _1+2 \alpha _2+4 \alpha _3+2 \alpha _4$ & $L_1+L_2-L_3+L_4$ \\
$3$ & $2$ & $2$ & $1$ & $1$ & $\alpha _1+2 \alpha _2+3 \alpha _3+\alpha _4$ & $L_1+L_4$ \\
$4$ & $2$ & $4$ & $3$ & $-1$ & $\alpha _1+2 \alpha _2+2 \alpha _3+2 \alpha _4$ & $L_1+L_2+L_3-L_4$ \\
$5$ & $0$ & $2$ & $1$ & $1$ & $\alpha _1+\alpha _2+2 \alpha _3+\alpha _4$ & $L_2+L_4$ \\
$6$ & $2$ & $0$ & $1$ & $1$ & $\alpha _1+2 \alpha _2+2 \alpha _3$ & $L_1-L_2+L_3+L_4$ \\
$7$ & $2$ & $4$ & $1$ & $-1$ & $\alpha _2+2 \alpha _3+2 \alpha _4$ & $L_1+L_2-L_3-L_4$ \\
$8$ & $2$ & $2$ & $1$ & $-1$ & $\alpha _2+\alpha _3+\alpha _4$ & $L_1-L_4$ \\
$9$ & $0$ & $0$ & $1$ & $1$ & $\alpha _1+\alpha _2+\alpha _3$ & $L_3+L_4$ \\
$10$ & $2$ & $0$ & $-1$ & $1$ & $\alpha _2+2 \alpha _3$ & $L_1-L_2-L_3+L_4$ \\
$11$ & $0$ & $2$ & $1$ & $-1$ & $\alpha _4$ & $L_2-L_4$ \\
$12$ & $0$ & $0$ & $-1$ & $1$ & $\alpha _3$ & $-L_3+L_4$ \\
$13$ & $2$ & $0$ & $1$ & $-1$ & $\alpha _2$ & $L_1-L_2+L_3-L_4$ \\
$14$ & $-2$ & $0$ & $1$ & $1$ & $\alpha _1$ & $-L_1+L_2+L_3+L_4$ \\
\hline
$15$ & $2$ & $0$ & $-1$ & $-1$ & $-\alpha _1$ & $L_1-L_2-L_3-L_4$ \\
$16$ & $-2$ & $0$ & $-1$ & $1$ & $-\alpha _2$ & $-L_1+L_2-L_3+L_4$ \\
$17$ & $0$ & $0$ & $1$ & $-1$ & $-\alpha _3$ & $L_3-L_4$ \\
$18$ & $0$ & $-2$ & $-1$ & $1$ & $-\alpha _4$ & $-L_2+L_4$ \\
$19$ & $-2$ & $0$ & $1$ & $-1$ & $-\alpha _2-2 \alpha _3$ & $-L_1+L_2+L_3-L_4$ \\
$20$ & $0$ & $0$ & $-1$ & $-1$ & $-\alpha _1-\alpha _2-\alpha _3$ & $-L_3-L_4$ \\
$21$ & $-2$ & $-2$ & $-1$ & $1$ & $-\alpha _2-\alpha _3-\alpha _4$ & $-L_1+L_4$ \\
$22$ & $-2$ & $-4$ & $-1$ & $1$ & $-\alpha _2-2 \alpha _3-2 \alpha _4$ & $-L_1-L_2+L_3+L_4$ \\
$23$ & $-2$ & $0$ & $-1$ & $-1$ & $-\alpha _1-2 \alpha _2-2 \alpha _3$ & $-L_1+L_2-L_3-L_4$ \\
$24$ & $0$ & $-2$ & $-1$ & $-1$ & $-\alpha _1-\alpha _2-2 \alpha _3-\alpha _4$ & $-L_2-L_4$ \\
$25$ & $-2$ & $-4$ & $-3$ & $1$ & $-\alpha _1-2 \alpha _2-2 \alpha _3-2 \alpha _4$ & $-L_1-L_2-L_3+L_4$ \\
$26$ & $-2$ & $-2$ & $-1$ & $-1$ & $-\alpha _1-2 \alpha _2-3 \alpha _3-\alpha _4$ & $-L_1-L_4$ \\
$27$ & $-2$ & $-4$ & $-1$ & $-1$ & $-\alpha _1-2 \alpha _2-4 \alpha _3-2 \alpha _4$ & $-L_1-L_2+L_3-L_4$ \\
$28$ & $-2$ & $-4$ & $-3$ & $-1$ & $-2 \alpha _1-3 \alpha _2-4 \alpha _3-2 \alpha _4$ & $-L_1-L_2-L_3-L_4$ \\
\hline
\end{tabular}
\caption{The elements of $\Phi_{\smallV}$.}
\label{table 3}
\end{table}

\subsection{The resolvent binary quartic}\label{subsection: the resolvent binary quartic}

In this section we define for every $v\in \smallV(k)$ a binary quartic form $Q_v$. 
At the end of \S\ref{subsection: an explicit description of V} we fixed an isomorphism $\smallG \simeq (\Sp_6\times \SL_2)/\mu_2$; let $p\colon \smallG \rightarrow \PGL_2$ be the corresponding projection map. 
Moreover we have fixed an isomorphism $\smallV \simeq \smallW\boxtimes (2)$, where $\smallW$ is the $14$-dimensional $\Sp_6$-representation described in \S\ref{subsection: an explicit description of V}.

\begin{definition}
	Let $k/\Q$ be a field and $v\in \smallV(k)$, giving rise to a pair of elements $(w_1,w_2)$ in $\smallW(k)$. We define the \define{resolvent binary quartic form} $Q_v$ by the formula
	\begin{align*}
		Q_v \coloneqq F(xw_1+yw_2) \in k[x,y]_{\deg = 4}.
	\end{align*}
\end{definition}

Note that $Q_{\lambda v} = \lambda^4 Q_v$ and $Q_{g\cdot v} = p(g)\cdot Q_v$, where an element $[A] \in \PGL_2(k)$ acts on a binary quartic form $Q(x,y)$ by $[A]\cdot Q(x,y)\coloneqq Q((x,y)\cdot A)/(\det A)^2$.

\begin{definition}
	Let $k/\Q$ be a field and $v\in \smallV(k)$. We say $v$ is \define{almost regular semisimple} if $Q_v$ has distinct roots in $\P^1(\bar{k})$. 
\end{definition}

\begin{lemma}\label{lemma: stable implies almost stable}
	Let $k/\Q$ be a field and $v\in \smallV(k)$. If $v$ is regular semisimple, then $v$ is almost regular semisimple.
\end{lemma}
\begin{proof}
	We may assume that $k$ is algebraically closed. 
	Assume for contradiction that $Q_v$ does not have distinct roots. 
	Then there exists an element $\gamma\in \PGL_2(k)$ so that the coefficients of $\gamma\cdot Q_v$ at $x^4$ and $x^3y$ vanish.
	Choosing a lift $g\in \smallG(k)$ of $\gamma$ and replacing $v$ by $g\cdot v$, we may assume that this holds for $Q_v$. 
	Therefore if $v = (w_1,w_2)\in \smallV(k)$ and $g(t) := F(w_1+tw_2)$ then $g(0)=g'(0)=0$. Since $F(w_1)=0$, Proposition \ref{proposition: orbits prehom W} shows that we may assume after conjugation by $\Sp_6(k)$ that $w_1$ is of the form 
	\begin{equation*}
	\left(1,\begin{pmatrix} * & 0 & 0 \\ 0 & * & 0 \\ 0 & 0 & * \end{pmatrix}, \begin{pmatrix} 0 & 0 & 0 \\ 0 & 0 & 0 \\ 0 & 0 & 0 \end{pmatrix}, 0 \right).
\end{equation*}
To derive a contradiction we will use the equivalence between Parts 1 and 4 of Proposition \ref{proposition: equivalences regular semisimple} repeatedly. 

We first claim that all the elements $*$ on the diagonal are nonzero.
If not, then we may assume that the one in the bottom right corner is zero. 
But then the one-parameter subgroup (in the explicit realizations of $\Sp_6$ described in \S\ref{subsection: an explicit description of V})
$$t\mapsto \text{diag}(1,1,t,1,1,t^{-1}) \times \text{diag}(t^{-1},t)$$
does not have a negative weight with respect to $v$, contradicting the assumption that $v$ is regular semisimple. 
Secondly, the condition $g'(0)=0$ translates into the condition that the coordinate of $w_2$ at $z$ in the decomposition $(u,X,Y,z)$ vanishes, by an explicit computation using Formula (\ref{equation: invariant W}). 
But then the one-parameter subgroup
$t \mapsto \text{diag}(t,t,t,t^{-1},t^{-1},t^{-1}) \times \text{diag}(t^{-1},t)$ again has no negative weight with respect to $v$, a contradiction.
\end{proof}


\begin{definition}\label{definition: almost $k$-reducible}
	For a field $k/\Q$ and an element $v \in \smallV(k)$, we say $v$ is \define{almost $k$-reducible} if it is not regular semisimple or the resolvent binary quartic form $Q_v$ has a $k$-rational linear factor. 
\end{definition}

\begin{lemma}\label{lemma: reducible implies almost reducible}
	Let $k/\Q$ be a field and $v\in \smallV(k)$. If $v$ is $k$-reducible (Definition \ref{definition: k-reducible orbit}), then $v$ is almost $k$-reducible. 
\end{lemma}
\begin{proof}
	We may assume that $v$ is regular semisimple and of the form $\sigma(b)$ for some $b\in \smallB(k)$. 
	A well-known result of Kostant determines the adjoint action of the $\liesl_2$-subalgebra generated by $(E,X,F)$ on $\smallh$ in terms of the exponents $2,6,8,12$ of $F_4$ \cite[Corollary 8.7]{Kostant-principalthreedimsubgroup}.
	It implies that $\sigma(b)\in \smallkappa(k)$ is supported on vectors whose weights, considered as elements of $\Phi_{\smallH}$, have root height $1,-1,-5,-7,-11$ with respect to $S_{\smallH}$. 
	Using Table \ref{table 3} it follows that if $\sigma(b) = (w_1,w_2)$ with $w_i \in \smallW(k)$ then $w_1$ is of the form
\begin{equation*}
	\left(0,\begin{pmatrix} * & 0 & 0 \\ 0 & 0 & 0 \\ 0 & 0 & 0 \end{pmatrix}, \begin{pmatrix} 0 & * & * \\ * & * & 0 \\ * & 0 & * \end{pmatrix}, * \right).
\end{equation*}
Formula (\ref{equation: invariant W}) shows that the polynomial $F$ vanishes on elements of such form, so $Q_v$ is divisible by $y$. 
\end{proof}

\begin{remark}\label{remark: explicit example element which is not almost reducible}
	Not every element $v\in \smallV(\Real)$ is almost $\Real$-reducible. 
	For a somewhat arbitrary example, let $v = (w_1,w_2)\in \smallV(\Real)$ be given by:
	\begin{align*}
	w_1  =\left(1,\begin{pmatrix} 1 & 2& 3\\ 4 & 1 &0 \\ 2 & 0 & 1\end{pmatrix},\begin{pmatrix} 1 & 2 & 3 \\ 1 & -2 & -1\\ 2 & 3 &1\end{pmatrix},0\right), \,
	w_2  =\left(-1, \begin{pmatrix} 3 & 2 & 5 \\ -1 & 0 & 0\\ 0 & 2 & 1\end{pmatrix} , \begin{pmatrix} 0 & -3 & 1\\ 2 & 2 & 1 \\  1 &1 &0\end{pmatrix},2\right).
	\end{align*}
	Then one computes that $Q_v = 376 x^4+507x^3y+1697x^2y^2+846 xy^3 +119 y^4$, which has no real roots nor repeated roots. 
	If $v$ is regular semisimple, we have obtained a valid example; if not, then we may replace $v$ by a small perturbation which is regular semisimple whose resolvent binary quartic form has no real roots either. This observation will be used in the proof of Lemma \ref{lemma: centralizer resolv bin quartic bigonal ell curve}.
\end{remark}

\subsection{A criterion for almost reducibility} \label{subsection: a criterion for reducibility}
Let $\Phi_{\smallH}^+$ denote the positive roots of $\Phi_{\smallH}$ with respect to $S_{\smallH}$ and write $\Phi_{\smallV}^+ \coloneqq \Phi_{\smallV} \cap \Phi_{\smallH}^+$. If $v\in \smallV$ we can decompose $v$ as $\sum_{{\alpha}\in \Phi_{\smallV}} v_{\alpha}$ with $v_{\alpha}$ in the weight space corresponding to ${\alpha}$. For any subset $M$ of $\Phi_{\smallV}$ we define the linear subspace 
\begin{displaymath}
\smallV(M) = \{v\in \smallV \mid v_{\alpha} = 0 \text{ for all }{\alpha} \in M  \}\subset \smallV.
\end{displaymath}

We state a lemma which describes sufficient conditions for an element $v\in \smallV$ to be almost $\Q$-reducible. This will (only) be useful when estimating the number of irreducible orbits in the cuspidal region in \S\ref{subsection: cutting off cusp}. 
Recall that we write $\alpha = \sum n_i(\alpha) \beta_i$.

\begin{lemma}\label{lemma: Q-reducibility conditions}
	Let $M$ be a subset of $\Phi_{\smallV}$, and suppose that one of the following three conditions is satisfied: 
	\begin{enumerate}
		\item There exist integers $b_1,\dots,b_4$ not all equal to zero such that 
		\begin{displaymath}
		\left\{{\alpha} \in \Phi_{\smallV} \mid \sum_{i=1}^4 b_i n_i({\alpha}) > 0 \right\} \subset M.
		\end{displaymath}
		\item For every $v = (w_1,w_2) \in \smallV(M)(\Q)$, we have $F(w_1)=0$. 
	\end{enumerate}
	Then every element of $\smallV(M)(\Q)$ is almost $\Q$-reducible. 
\end{lemma}
\begin{proof}
	In the first case, the integers $b_1,\dots,b_4$ determine a cocharacter of $\smallT$ with respect to which every element of $\smallV(M)(\Q)$ has only nonnegative weights. By the Hilbert-Mumford stability criterion (Proposition \ref{proposition: equivalences regular semisimple}), $\smallV(M)(\Q)$ then contains no regular semisimple elements so consists solely of almost $\Q$-reducible elements. 
	If the second condition is satisfied, then for every $v\in \smallV(\Q)$ the resolvent binary quartic form $Q_v$ has a $\Q$-rational linear factor, so $v$ is almost $\Q$-reducible too. 
\end{proof}

\begin{lemma}\label{lemma: Q-reducibility conditions explicit weights}
	Let $M$ be a subset of $\Phi_{\smallV}$, and suppose $M$ contains one of the following subsets, in the notation of Table \ref{table 3}: 
	\begin{align*}
		\{1,2,3,4,5,7,8,11\}, 
		\{1,2,3,4,6,7,8,10,13,15\},\\
		\{1,2,3,5,6,9,10\},
		\{1,2,3,5,6,9,14\}.
	\end{align*}
	Then every element of $\smallV(M)(\Q)$ is almost $\Q$-reducible. 
	
\end{lemma}
\begin{proof}
We show that $M$ satisfies one of the conditions of Lemma \ref{lemma: Q-reducibility conditions}.
If $M$ contains the first displayed subset, we may use Condition 1 with $(b_1,b_2,b_3,b_4)= (0,1,0,0)$.
If $M$ contains the second subset, we use the same condition with $(b_1,b_2,b_3,b_4)= (1,0,0,0)$. 
The last two cases follow from Condition 2: indeed for $v\in \smallV(M)(\Q)$ the vector $w_1$ is either of the form
	\begin{equation*}
	\left(0,\begin{pmatrix} 0 & 0 & 0 \\ 0 & 0 & 0 \\ 0 & 0 & 0 \end{pmatrix}, \begin{pmatrix} * & * & * \\ * & * & * \\ * & * & * \end{pmatrix}, * \right) \text{ or }
	\left(0,\begin{pmatrix} * & 0 & 0 \\ 0 & 0 & 0 \\ 0 & 0 & 0 \end{pmatrix}, \begin{pmatrix} 0 & * & * \\ * & * & * \\ * & * & * \end{pmatrix}, * \right).
\end{equation*}
In both cases we see using Formula (\ref{equation: invariant W}) that $F(w_1) = 0$.
\end{proof}

\section{Geometry}\label{section: geometry}

\subsection{A family of curves}\label{subsection: a family of curves}

In this section we relate the representation $(\smallG,\smallV)$ to our family of curves of interest, see Proposition \ref{proposition: bridge F4 rep and F4 curves}.
The proof involves a similar result for the representation $(\bigG,\bigV)$ and a study of the involution $\zeta$.
We first recall this result for $(\bigG,\bigV)$, after introducing some notation. 

Let $\bigV^{\rs}$ denote the open subscheme of regular semisimple elements of $\bigV\subset \bigh$, and let $\bigB^{\rs}$ be its image under $\bigpi\colon \bigV\rightarrow \bigB$. 
Define the $\bigB^{\rs}$-scheme $\bigA \coloneqq  Z_{\bigH}(\bigsigma|_{\bigB^{\rs}})$, the centralizer of the $\bigB^{\rs}$-point $\bigsigma|_{\bigB^{\rs}}$ of $\bigV$. It is a family of maximal tori in $\bigH$ parametrized by $\bigB^{\rs}$. 
We also define $\bigLambda \coloneqq \Hom(\bigA,\G_m)$ as the character group of $\bigA$. Then $\bigLambda$ is an \'etale sheaf of $E_6$ root lattices on $\bigB^{\rs}$.
By definition, this means that $\bigLambda$ is a locally constant \'etale sheaf of finite free $\Z$-modules, equipped with a pairing $(\cdot,\cdot)\colon \bigLambda \times \bigLambda \rightarrow \Z$ such that for every geometric point $\bar{x}$ of $\bigB^{\rs}$, the stalk of $\bigLambda$ at $\bar{x}$ is a root lattice of type $E_6$. 
This induces a pairing $\bigLambda/2\bigLambda\times \bigLambda/2\bigLambda \rightarrow \{\pm 1\} : (\lambda,\mu) \mapsto (-1)^{(\lambda,\mu)}$. 

\begin{proposition}\label{proposition: E6 bridge jacobians root lattices}
	We can choose polynomials $p_2, p_5, p_6, p_8, p_9, p_{12} \in \Q[\bigV]^{\bigG}$ with the following properties:
	\begin{enumerate}
		\item Each polynomial $p_i$ is homogeneous of degree $i$ and $\Q[\bigV]^{\bigG} \simeq \Q[p_2, p_5, p_6, p_8, p_9, p_{12}]$. Consequently, there is an isomorphism $\bigB\simeq \A^6_{\Q}$.
		\item Let $\bigprojcurve \rightarrow \bigB$ be the family of projective curves inside $\P^2_{\bigB}$ with affine model
		\begin{equation}\label{equation : E6 family middle of paper}
		y^4+x(p_2y^2+p_5y)+(p_6y^2+p_9y) = x^3+p_8x+p_{12}.
		\end{equation}
		If $k/\Q$ is a field and $b\in \bigB(k)$, then $(\bigprojcurve)_b$ is smooth if and only if $b\in \bigB^{\rs}(k)$. 
		
		\item Let $\bigJac \rightarrow \bigB^{\rs}$ be the Jacobian of its smooth part \cite[\S9.3; Theorem 1]{BLR-NeronModels}. 
		Then there is a unique isomorphism $\bigLambda/2\bigLambda \simeq \bigJac[2] $ of finite \'etale group schemes over $\bigB^{\rs}$ that sends the pairing on $\bigLambda/2\bigLambda$ to the Weil pairing $\bigJac[2] \times \bigJac[2] \rightarrow \{ \pm 1\}$. 
		\item There exists an isomorphism $Z_{\bigG}(\bigsigma|_{\bigB^{\rs}}) \simeq \bigJac[2]$ of finite \'etale group schemes over $\bigB^{\rs}$. 
		

	\end{enumerate}
\end{proposition}
\begin{proof}
    This is a combination of classical results and Thorne's thesis \cite{Thorne-thesis}.
	We refer to \cite[Proposition 2.5]{Laga-E6paper} for precise references, with the caveat that the role of the coordinates $x$ and $y$ is interchanged here.
	The only part that remains to be proven is the uniqueness of the isomorphism $\bigLambda/2\bigLambda \simeq \bigJac[2]$ that preserves the pairings on both sides. This follows from \cite[Proposition 2.6(4)]{Laga-E6paper}. 
\end{proof}

We now incorporate the involution $\zeta$ in the picture, and compare it to an involution defined on the level of curves. Recall that $\zeta\colon \bigG \rightarrow \bigG$ is an involution with fixed points $\smallG$. 
Since $\zeta$ commutes with $\bigsigma$, it defines an involution of the scheme $Z_{\bigH}(\bigsigma|_{\bigB^{\rs}})$ lifting the involution $\zeta\colon \bigB^{\rs} \rightarrow \bigB^{\rs} $. 
It induces an involution of $\bigLambda/2\bigLambda$, still denoted by $\zeta$.

On the other hand, if $\bigprojcurve\rightarrow \bigB$ denotes the family of Equation (\ref{equation : E6 family middle of paper}), then the map $(x,y) \mapsto (x,-y)$ defines an involution $\curvezeta \colon \bigprojcurve \rightarrow \bigprojcurve$ lifting the involution $(-1)\colon \bigB\rightarrow \bigB$.
It induces an involution of $\bigJac[2]$, denoted by $\curvezeta^*$.

\begin{lemma}\label{lemma: identify involutions}
	Under the isomorphism $\phi\colon \bigLambda/2\bigLambda\xrightarrow{\sim} \bigJac[2]$ from Proposition \ref{proposition: E6 bridge jacobians root lattices}, the involutions $\zeta$ and $\curvezeta^*$ are identified.
\end{lemma}
\begin{proof}
	Write $\phi' = \curvezeta^* \circ \phi \circ \zeta$. 
	We need to prove that $\phi' = \phi$. 
	By the second part of Proposition \ref{proposition: action zeta on B}, $\phi'\colon \bigLambda/2\bigLambda\rightarrow \bigJac[2]$ is an isomorphism of $\bigB^{\rs}$-schemes. 
	Moreover, $\zeta$ and $\curvezeta^*$ respect the pairings on $\bigLambda/2\bigLambda$ and $\bigJac[2]$ respectively. 
	The result follows from the uniqueness statement in Part 3 of Proposition \ref{proposition: E6 bridge jacobians root lattices}. 
\end{proof}

Proposition \ref{proposition: E6 bridge jacobians root lattices} and Lemma \ref{lemma: identify involutions} have the following important consequence, which connects the representation $(\smallG,\smallV)$ with the subfamily $\smallprojcurve\rightarrow \smallB$ of $\bigprojcurve\rightarrow \bigB$. 
Again let $\smallV^{\rs}$ denote the open subscheme of regular semisimple elements of $\smallV$ and let $\smallB^{\rs}$ be its image under $\smallpi\colon \smallV\rightarrow \smallB$. 

\begin{proposition}\label{proposition: bridge F4 rep and F4 curves}
	We can choose polynomials $p_2, p_6, p_8, p_{12} \in \Q[\smallV]^{\smallG}$ with the following properties:
	\begin{enumerate}
		\item Each polynomial $p_i$ is homogeneous of degree $i$ and $\Q[\smallV]^{\smallG} \simeq \Q[p_2, p_6, p_8, p_{12}]$. Consequently, there is an isomorphism $\smallB\simeq \A^4_{\Q}$.
		\item Let $\smallprojcurve \rightarrow \smallB$ be the family of projective curves inside $\P^2_{\smallB}$ with affine model
		\begin{equation}\label{equation : F4 family middle of paper}
		y^4+p_2xy^2+p_6y^2 = x^3+p_8x+p_{12}.
		\end{equation} 
		If $k/\Q$ is a field and $b\in \smallB(k)$, then $\smallprojcurve_b$ is smooth if and only if $b\in \smallB^{\rs}(k)$. 
		\item Let $\Jac \rightarrow \smallB^{\rs}$ be the Jacobian of the morphism $\smallprojcurve^{\rs}\rightarrow \smallB^{\rs}$. Let $\curvezeta: \smallprojcurve^{\rs} \rightarrow \smallprojcurve^{\rs}$ be the involution of $\smallB^{\rs}$-schemes sending $(x,y)$ to $(x,-y)$ and let $\curvezeta^*:\Jac \rightarrow \Jac$ be the induced morphism on $\Jac$. Then the isomorphism $Z_{\bigG}(\smallsigma|_{\smallB^{\rs}}) \simeq \Jac[2]$ obtained from Proposition \ref{proposition: E6 bridge jacobians root lattices} intertwines the involutions $\zeta$ and $\curvezeta^*$ and restricts to an isomorphism $Z_{\smallG}(\smallsigma|_{\smallB^{\rs}}) \simeq \Jac[2]^{\tau^*}$
	\end{enumerate}
\end{proposition}
\begin{proof}
	Let $p'_2, p'_5, p'_6, p'_8, p'_9, p'_{12} \in \Q[\bigV]^{\bigG}$ be a choice of polynomials satisfying the conclusion of Proposition \ref{proposition: E6 bridge jacobians root lattices}.
	Write $p_i\in \Q[\smallV]^{\smallG}$ for the restriction of $p'_i$ to $\smallV$. The first two parts of Proposition \ref{proposition: action zeta on B} imply that $p_5 = p_9 = 0$ and $\Q[\smallV]^{\smallG} = \Q[p_2,p_6,p_8,p_{12}]$.
	The family $\smallprojcurve\rightarrow \smallB$ is the pullback of the family $\bigprojcurve\rightarrow \bigB$ along $\smallB\rightarrow \bigB$.
	Moreover, Part 3 of Proposition \ref{proposition: action zeta on B} shows that $\smallB^{\rs}=\smallB\cap \bigB^{\rs}$. 
	The proposition now follows from Proposition \ref{proposition: E6 bridge jacobians root lattices} and Lemma \ref{lemma: identify involutions}.
\end{proof}

We henceforth fix $p_2,p_6,p_8,p_{12}\in \Q[\smallV]^{\smallG}$ satisfying the conclusions of Proposition \ref{proposition: bridge F4 rep and F4 curves}.
Recall that we have defined a $\G_m$-action on $\smallB$ which satisfies $\lambda\cdot p_i = \lambda^ip_i$. 
The assignment $\lambda\cdot(x,y) = (\lambda^4x,\lambda^3y)$ defines a $\G_m$-action on $\smallprojcurve$ such that the morphism $\smallprojcurve\rightarrow \smallB$ is $\G_m$-equivariant. 

\subsection{Monodromy of \texorpdfstring{$\Jac[2]$}{J[2]}}\label{subsection: monodromy of J[2]}

We give some additional properties of the finite \'etale group scheme $\Jac[2] \rightarrow \smallB^{\rs}$.  
Recall that $\smallT$ is a split maximal torus of $\smallH$; let $\liet$ be its Lie algebra and $W\coloneqq N_{\smallG}(\smallT)/\smallT$ its Weyl group . 
We define a map $f\colon \liet^{\rs} \rightarrow \smallB^{\rs}$ as follows. 
The inclusions $\liet \subset \lieh$ and $\smallV\subset \lieh$ induce isomorphisms $\liet\GIT W \simeq \lieh\GIT \smallH$ and $\smallB \simeq \lieh\GIT \smallH$ by the classical Chevalley isomorphism and Proposition \ref{prop : graded chevalley} respectively. 
Composing the first with the inverse of the second determines an isomorphism $\liet\GIT W \xrightarrow{\sim} \smallB $. 
Precomposing this isomorphism with the natural projection $\liet \rightarrow \liet \GIT W$ and restricting to regular semisimple elements defines a morphism $f\colon \liet^{\rs} \rightarrow \smallB^{\rs}$. Since $\liet^{\rs} \rightarrow \liet^{\rs}\GIT W$ is a torsor under $W$, $f$ is a $W$-torsor too. 

Let $W_{\mathrm{E}}$ be the Weyl group of the split maximal torus $\bigT$ of $\bigH$. 
It it known that the inclusion $\smallT \subset \bigT$ induces an isomorphism of $W$ onto $W_{\mathrm{E}}^{\zeta}$, the centralizer of $\zeta$ in $W_{\mathrm{E}}$ \cite[\S13.3.3]{Carter-SimpleGroupsLieType1972}.
We therefore obtain an action of $W$ on $\splitLambda \coloneqq X^*(\bigT)$, a root lattice of type $E_6$.

\begin{proposition}\label{proposition: monodromy J[2]}
	The finite \'etale group scheme $\Jac[2] \rightarrow \smallB^{\rs}$ becomes trivial after the base change $f\colon \liet^{\rs} \rightarrow \smallB^{\rs}$, where it is isomorphic to the constant group scheme $\splitLambda/2\splitLambda$. The monodromy action is given by the natural action of $W \simeq W_{\mathrm{E}}^{\zeta}$. 
\end{proposition}
\begin{proof}
	By \cite[Part 1 of Proposition 2.6]{Laga-E6paper}, the group scheme $\bigJac[2] \rightarrow \bigB^{\rs} $ becomes trivial after the base change $f_{\mathrm{E}}\colon \bigt^{\rs} \rightarrow \bigB^{\rs}$ where $f_{\mathrm{E}}$ is defined analogously as before. Moreover the monodromy action is given by the natural action of $W_{\mathrm{E}}$ on $\splitLambda/2\splitLambda$. 
	The proposition is thus implied by the following commutative diagram:
	\begin{center}
	\begin{tikzcd}
		\smallt^{\rs} \arrow[d,"f"]  \arrow[r] & \bigt^{\rs} \arrow[d, "f_{\mathrm{E}}" ] \\
		\smallB^{\rs} \arrow[r] & \bigB^{\rs}
	\end{tikzcd}
	\end{center}
	
\end{proof}

\begin{corollary}\label{corollary: subgroups Jac[2]}
	The finite \'etale $\smallB^{\rs}$-subgroup schemes of $\Jac[2]$ are 
	\begin{align}
		0 \subset (1+\curvezeta^*)\Jac[2] \subset \Jac[2]^{\curvezeta^*}\subset \Jac[2] 
	\end{align}
	of order $1,2^2,2^4,2^6$ respectively. Moreover the $\smallB^{\rs}$-group schemes $(1+\curvezeta^*)\Jac[2]$ and $\Jac[2]^{\curvezeta^*}/(1+\curvezeta^*)\Jac[2]$ are not isomorphic, even after base change to $k$ for any field extension $k/\Q$. 
\end{corollary} 
\begin{proof}
In light of Proposition \ref{proposition: monodromy J[2]}, the above claims are reduced to analyzing the action of $W_{\mathrm{E}}^{\zeta}$ on $\splitLambda/2\splitLambda$. For example for the first part it suffices to determine the $W_{\mathrm{E}}^{\zeta}$-invariant subgroups of $\splitLambda/2\splitLambda$ and for the second part, it suffices to find an element of $W_{\mathrm{E}}^{\zeta}$ which acts trivially on $(1+\zeta)\left(\splitLambda/2\splitLambda \right)$ but not so on $\left(\splitLambda/2\splitLambda\right)^{\zeta}/(1+\zeta)\left(\splitLambda/2\splitLambda \right)$. Both are direct computations in the $E_6$ root lattice, which we omit. 
\end{proof}

\subsection{A family of Prym varieties} \label{subsection: def prym variety}

In this section we introduce the family of Prym surfaces $\Prym \rightarrow \smallB^{\rs}$ and discuss some of its properties. We first discuss it in a more general set-up.

Let $k$ be a field of characteristic different from $2$ and $X/k$ a smooth projective genus-$3$ curve.
Let $\tau\colon X\rightarrow X$ be an involution with four fixed points. 
Suppose we are given a $k$-point $\infty\in X(k)$ fixed by $\tau$. 
Let $E\coloneqq X/\tau$ be the quotient of $X$ by $\tau$ and $f\colon X\rightarrow E$ be the associated double cover which is branched at four points. 
By the Riemann--Hurwitz formula, $E$ is an elliptic curve with origin $f(\infty)$. This defines an isomorphism $E\simeq J_E$ between $E$ and its Jacobian which sends $f(\infty)$ to the identity of $J_E$.

The Jacobian variety $J_X$ of $X$ is not simple. Indeed the map $f$ induces a surjective norm homomorphism $f_*\colon J_X \rightarrow J_E \simeq E$ which sends the equivalence class $[D]$ of a divisor to $[f(D)]$,
so $E$ is an isogeny factor of $J_X$.
To describe the remaining part of $J_X$ we use the following classical definition.

\begin{definition}
	We define the \define{Prym variety} $\Prym_{X,\tau}$ of the pair $(X,\tau)$ as the kernel of the norm map:
	\begin{align*}
		\Prym_{X,\tau} \coloneqq \ker\left( f_*: J_X \rightarrow E\right).
	\end{align*}
\end{definition}

Prym varieties have been studied by Mumford in a much more general set-up \cite{Mumford-prymvars}.
We warn the reader that many authors only consider fixed-point free involutions when defining Prym varieties or equivalently, unramified double covers.
In our case the algebraic group $\Prym_{X,\tau}$ satisfies the following properties:
\begin{enumerate}
	\item \label{enum: prym 1} Let $f^*\colon E \rightarrow \Jac_X$ be the pullback map on divisors. Then $f^*$ is injective, $f_*\circ f^*=[2]$ and $f^*\circ f_* = 1+\tau^*$. Hence
	\begin{align}\label{equation: def prym as kernel 1+tau}
	\Prym_{X,\tau} = \ker\left(1+\tau^*\colon J_X\rightarrow J_X \right).
	\end{align}
	\item \label{enum: prym 2} $\Prym_{X,\tau}$ is connected, hence an abelian surface. 
	\item \label{enum: prym 3} The restriction of $f^*: E \rightarrow J_X$ to $E[2]$ induces an isomorphism $E[2] \xrightarrow{\sim} \image(f^*) \cap \Prym_{X,\tau}$. Consequently there is an injective morphism $\psi: E[2] \hookrightarrow \Prym_{X,\tau}[2]$. 
	\item \label{enum: prym 4} The map $E\times \Prym_{X,\tau} \rightarrow \Jac_X$, determined by $f^*$ and the inclusion  $\Prym_{X,\tau} \hookrightarrow \Jac_X$, is surjective with kernel equal to the graph of $\psi$, given by $\{(x,\psi(x)) \mid x\in E[2] \}$. Consequently there is an isomorphism 
	$$J_X \simeq \left(E \times \Prym_{X,\tau} \right)/\{(x,\psi(x) ) \mid x\in E[2] \}.$$
	So $J_X$ is isogenous to $E \times \Prym_{X,\tau}$. 
	\item \label{enum: prym 5} The cokernel of $f^*$ is naturally identified with the dual abelian variety of $\Prym_{X,\tau}$, written $\Prym_{X,\tau}^{\vee}$, and the composite $\Prym_{X,\tau} \hookrightarrow J_X \twoheadrightarrow \Prym_{X,\tau}^{\vee} $, denoted $\rho$, is a polarization of type $(1,2)$. (This means that $\Prym_{X,\tau}[\rho]$ is isomorphic to $(\Z/2)^2$ over $\bar{k}$.) 
	The map $f^*\colon \ellcurve[2] \rightarrow \Prym_{X,\tau}[\rho]$ is an isomorphism.
\end{enumerate}

Indeed, to verify the above properties we may assume that $k$ is algebraically closed. 
Then Property \ref{enum: prym 1} follows from \cite[\S3; Lemma 1]{Mumford-prymvars} and the fact that $X\rightarrow E$ is ramified at four points.
The other properties follow from going through the correspondence described in \cite[\S2]{Mumford-prymvars}: in the notation of that paper, we start with Data I of the form $(X,Y,\phi) = (E,X,f^*)$ whose invariants are $(a,b,c) = (1,2,1)$.
Equation (2.1) of loc. cit. holds by the discussion in \S1 of op. cit.
We additionally record the following important fact.

\begin{lemma}\label{lemma: rho is self-dual}
The isogeny $\rho\colon \Prym_{X,\tau} \rightarrow \Prym_{X,\tau}^{\vee}$ is self-dual. 
\end{lemma}
\begin{proof}
	This follows from the fact that $f_*$ and $f^*$ are dual to each other when transported along the principal polarizations of $J_X$ and $E$, see \cite[End of \S1]{Mumford-prymvars}.
\end{proof}

We now specialize to our situation of interest.
Recall that $\smallprojcurve^{\rs}\rightarrow \smallB^{\rs}$ consists of the smooth members of the projective closure of the family of curves 
\begin{equation*}
	y^4+p_2xy^2+p_6y^2 = x^3+p_8x+p_{12}
\end{equation*} 
and that $\tau\colon \smallprojcurve \rightarrow \smallprojcurve, (x,y) \mapsto (x,-y)$ is the involution which defines, for every field $k/\Q$ and $b\in \smallB^{\rs}(k)$, an involution $\tau_b\colon \smallprojcurve_b\rightarrow \smallprojcurve_b$ with four fixed points fixing the point at infinity $\infty\in \smallprojcurve_b(k)$.

Define $\Cellcurve \rightarrow \smallB$ to be the projective completion of the family of plane curves given by 
\begin{equation}\label{equation: elliptic curve}
y^2+p_2xy+p_6y = x^3+p_8x+p_{12}.
\end{equation}
Define $\ellcurve \rightarrow \smallB^{\rs}$ to be its restriction to $\smallB^{\rs}$. 
Then there is a unique morphism of $\smallB$-schemes $f:\smallprojcurve\rightarrow \Cellcurve$ sending a point $(x,y)$ to $(x,y^2)$. 
This identifies, for each field $k/\Q$ and $b\in \smallB^{\rs}(k)$, $\ellcurve_b$ with the quotient of $\smallprojcurve_b$ by $\curvezeta_b$. 

The morphism $\curvezeta$ defines via pullback a morphism of abelian schemes $\curvezeta^*\colon \Jac \rightarrow \Jac$. Define 
\begin{align}\label{equation: def prym of our family}
\Prym \coloneqq \ker(1+\curvezeta^*\colon \Jac \rightarrow \Jac).
\end{align}
The morphism $\Prym \rightarrow \smallB^{\rs}$ is proper and by Equation (\ref{equation: def prym as kernel 1+tau}) its fibres are abelian surfaces enjoying the properties described above. 
The next useful lemma \cite[Proposition 3.5]{Edixhoven-Neronmodelstameramification} applied to the $\Z/2\Z$-action $-\tau^*$ on $\Jac \rightarrow \smallB^{\rs}$ shows that $\Prym \rightarrow \smallB^{\rs}$ is smooth, hence an abelian scheme.

\begin{lemma}\label{lemma: fixed points smooth morphism is smooth}
    Let $G$ be a finite group, acting equivariantly on a smooth morphism of schemes $X\rightarrow S$. If the order of $G$ is invertible on $S$, then the induced morphism on fixed points $X^G \rightarrow S^G$ is smooth. 
\end{lemma}

\begin{lemma}\label{lemma: filtration J[2] interpretation Prym}
The filtration 
$0 \subset (1+\curvezeta^{*})\Jac[2] \subset \Jac[2]^{\curvezeta^*} \subset \Jac[2]$ of Corollary \ref{corollary: subgroups Jac[2]}
is identified with the filtration 
$$
0 \subset \ellcurve[2] \subset \Prym[2] \subset \Jac[2].
$$
Here we see $\ellcurve[2]$ as a subgroup of $\Prym[2]$ using the pullback map $f^* \colon \ellcurve[2] \hookrightarrow \Jac[2]$.
\end{lemma}
\begin{proof}
    This follows from Corollary \ref{corollary: subgroups Jac[2]} and the fact that $\ellcurve[2]$ and $\Prym[2]$ have order $2^2$ and $2^4$ respectively.
\end{proof}

The map $1+\tau^*\colon \Jac[2]\rightarrow \Jac[2]$ has image $\ellcurve[2]$ and kernel $\Prym[2]$, so $\Jac[2]/\Prym[2]\simeq \ellcurve[2]$. 
The remaining graded piece $\Prym[2]/\ellcurve[2]$ of the filtration of Lemma \ref{lemma: filtration J[2] interpretation Prym} will be determined in Corollary \ref{corollary: subgroups Prym[2] using bigonal}.

Since $\Prym[2] = \Jac[2]^{\curvezeta^*}$, Proposition \ref{proposition: bridge F4 rep and F4 curves} immediately implies the following.

\begin{proposition}\label{proposition: iso centralizer in G and prym[2]}
	The isomorphism $Z_{\bigG}(\smallsigma|_{\smallB^{\rs}}) \simeq \Jac[2]$ of Proposition \ref{proposition: bridge F4 rep and F4 curves} restricts to an isomorphism $Z_{\smallG}(\smallsigma|_{\smallB^{\rs}}) \simeq \Prym[2]$ of finite \'etale group schemes over $\smallB^{\rs}$. 
\end{proposition}

The following diagram of smooth group schemes over $\smallB^{\rs}$ summarizes the situation.
\begin{center}
	\begin{tikzcd}
		&                                         & 1 \arrow[d]                             &                       &   \\
		&                                         & \ellcurve \arrow[d, "f^*"] \arrow[rd,"\times2"  ]       &                       &   \\
		1 \arrow[r] & \Prym \arrow[r] \arrow[rd, "\rho"'] & \Jac \arrow[r, "{f_{*}}"] \arrow[d] & \ellcurve \arrow[r] & 1 \\
		&                                         & \Prym^{\vee} \arrow[d]                &                       &   \\
		&                                         & 1                                       &                       &  
	\end{tikzcd}
\end{center}

\subsection{The bigonal construction}\label{subsection: the bigonal construction}

Let $k/\Q$ be a field and let $(X,\tau)$ be a pair where $X/k$ a smooth projective genus-3 curve and $\tau\colon X \rightarrow X$ an involution with four fixed points. 
Let $\infty\in X(k)$ be a $k$-point fixed by $\tau$. 
In \S\ref{subsection: def prym variety} we have associated to this data a Prym variety $\Prym_{X,\tau}$.
In this section we will, under the presence of additional assumptions, realize the dual $\Prym_{X,\tau}^{\vee}$ as the Prym variety of another such pair $(\hat{X},\hat{\tau})$. This is a special case of the bigonal construction going back to Pantazis \cite{Pantazis-Prymvarietiesgeodesicflow} but we present it in a way closer in spirit to Barth \cite{Bart-Abeliansurfacestype12} who analyzed the above situation in great detail. Only Theorem \ref{theorem: summary bigonal construction} will be used later.

Recall that we have defined a polarization $\rho\colon \Prym_{X,\tau} \rightarrow \Prym_{X,\tau}^{\vee}$  of type $(1,2)$; let $\rhodual\colon \Prym^{\vee}_{X,\tau}\rightarrow \Prym_{X,\tau}$ be the unique isogeny such that $\rhodual\circ \rho = [2]$. 
(Warning: $\rhodual$ is not the dual of $\rho$!)
Define $i\colon X\rightarrow \Prym_{X,\tau}^{\vee}$ as the composite of the Abel--Jacobi map $X\rightarrow \Jac_X$ with respect to $\infty$ with the projection $\Jac_X \rightarrow \Prym_{X,\tau}^{\vee}$.

\begin{proposition}[Barth]\label{proposition: barth results}
	\begin{enumerate}
		\item The morphism $i\colon X\hookrightarrow \Prym_{X,\tau}^{\vee}$ is a closed embedding. 
		\item The divisor $i(X)$ is ample and the induced polarization of $\Prym_{X,\tau}^{\vee}$ coincides with $\rhodual$. 
		\item If $A/k$ is an abelian surface and $j\colon X\hookrightarrow A$ is a closed embedding mapping $\infty$ to $0$ such that $[-1]$ restricts to $\tau$ on $X$, then there exists a unique isomorphism of abelian varieties $\Prym_{X,\tau}^{\vee}\rightarrow A$ sending $i$ to $j$.   
	\end{enumerate} 
\end{proposition}
\begin{proof}
	We may suppose that $k$ is algebraically closed. Part 1 follows from the proof of \cite[Proposition 1.8]{Bart-Abeliansurfacestype12}.
	For Part 2, note that by the adjunction formula we have $i(X)\cdot i(X) = 2p_a(X)-2 = 4$ and if a curve $Y$ on $A$ is not numerically equivalent to $i(X)$ we can translate $Y$ using $A$ so that it intersects $i(X)$ in a finite non-empty subscheme, implying that $Y\cdot i(X)>0$. 
	 Therefore $i(X)$ is ample by the Nakai--Moishezon criterion. Let $\lambda\colon \Prym_{X,\tau}^{\vee} \rightarrow \Prym_{X,\tau}$ be the corresponding polarization. 
	The fact that $\lambda = \rhodual$ follows from the equality $\ker \lambda = \ker \rhodual$ \cite[Lemma 1.11]{Bart-Abeliansurfacestype12}. 
	Part 3 is \cite[Proposition 1.10]{Bart-Abeliansurfacestype12}.
\end{proof}

We now describe the bigonal construction.
Suppose in addition to the above that $X$ is not hyperelliptic and we are given an effective divisor $\kappa$ on $X$ fixed by $\tau$ such that $2\kappa$ is canonical. 
Let $\Theta_{\kappa}\subset J_X$ be the corresponding theta divisor, namely the pullback of the image of the natural summing map $X\times X\rightarrow \Pic^2(X)$ along the translation-by-$\kappa$ map $J_X \rightarrow \Pic^2(X), D\mapsto D+\kappa$.
The divisor $\Theta_{\kappa}$ is symmetric and induces the principal polarization on $J_X$; we refer to \cite[Chapter 11, \S 2]{BirkenhakeLange-CAV} for these classical facts. 
Set $$\hat{X}  \coloneqq \Theta_{\kappa} \cap \Prym_{X,\tau}.$$
Then $\hat{X}$ induces the polarization $\rho\colon \Prym_{X,\tau} \rightarrow \Prym_{X,\tau}^{\vee}$, by construction of $\rho$. 
Let $\hat{\tau}$ be the restriction of $[-1] $ to $\hat{X}$, which coincides with the restriction of $\tau^*$ to $\hat{X}$. 
\begin{lemma}\label{lemma: bigonal dual is smooth proj genus 3}
	The curve $\hat{X}$ is smooth, geometrically connected and of genus $3$. The involution $\hat{\tau}\colon \hat{X} \rightarrow \hat{X}$ has $4$ fixed points over $\bar{k}$. 
\end{lemma}
\begin{proof}
	We may suppose that $k$ is algebraically closed. 
	Since $\hat{X}$ is an ample divisor on the smooth projective surface $\Prym_{X,\tau}$, it is connected by the Kodaira vanishing theorem. 
	Moreover because $\hat{X}$ defines a polarization of degree $4$, it has self-intersection $4$ so arithmetic genus $3$ by the adjunction formula. 
	Because we assumed that $X$ is not hyperelliptic and of genus $3$, $\Theta_{\kappa}$ is smooth by Riemann's singularity theorem \cite[Chapter 11, \S2.5]{BirkenhakeLange-CAV}.
	Therefore $\hat{X}$ is smooth by Lemma \ref{lemma: fixed points smooth morphism is smooth}, being the fixed points of the involution $[-1]\circ \tau\colon \Theta_{\kappa} \rightarrow \Theta_{\kappa}$. 
	
	It remains to calculate the number of fixed points of $\hat{\tau}$.  
	Let $f\colon X\rightarrow E$ be the quotient of $X$ by $\tau$ and let $g\colon E\rightarrow \P^1$ be the morphism induced by the degree-$2$ divisor $f_*(\kappa)$.
	Since $\Theta_{\kappa}$ is smooth and $X$ is not hyperelliptic, the summing map $\Sym^2X\rightarrow \Theta_{\kappa}, D\mapsto D-\kappa$ is an isomorphism; let $\widetilde{X}$ be the inverse image of $\hat{X}$ under this isomorphism. 
	An effective degree-$2$ divisor $D$ lies on $\widetilde{X}$ if and only if $D+\tau(D)\sim 2\kappa$. 
 	Since $f^*\colon \Pic(E) \rightarrow \Pic(X)$ is injective and $f^*\circ f_*=1+\tau^*$ (Property \ref{enum: prym 1} of \S\ref{subsection: def prym variety}), the latter holds if and only if $f_*(D)\sim f_*(\kappa)$.
 	
 	It suffices to prove that the involution $D\mapsto \tau(D)$ on $\widetilde{X}$ has $4$ fixed points. 
 	If $e_1,\dots,e_4$ are the ramification points of $g$ then $f^*(e_1),\dots,f^*(e_4)$ are fixed points; we claim that these are the only ones. 
 	Arguing by contradiction, suppose that $D=P_1+P_2\in \widetilde{X}$ is fixed by $\tau$ and not of this form. Then $\tau(P_i) = P_i$ for $i=1,2$ and $P_1\neq P_2$; write $P_3, P_4$ for the remaining fixed points of $\tau$ on $X$. 
 	We have equivalences of divisors $2P_1+2P_2 = D+\tau(D) \sim 2\kappa \sim P_1+P_2+P_3+P_4$ where last equivalence follows from the Riemann-Hurwitz formula applied to $f$. 
 	This implies that $P_1+P_2 \sim P_3+P_4$. 
 	Since $X$ is not hyperelliptic and $P_1,\dots,P_4$ are distinct, we obtain a contradiction.
\end{proof}

The effective degree-2 divisor $\kappa$ defines a point $\hat{\infty}\in\hat{X}(k)$ fixed by $\hat{\tau}$. 
We thus obtain a Prym variety $\Prym_{\hat{X},\hat{\tau}}$ and an embedding $\hat{i}\colon \hat{X} \hookrightarrow \Prym_{\hat{X},\hat{\tau}}^{\vee}$ as defined above. 
The inclusion $\hat{X} \hookrightarrow \Prym_{X,\tau}$ maps $\hat{\infty}$ to $0$ and extends to a homomorphism $J_{\hat{X}} \rightarrow \Prym_{X,\tau}$ from the Jacobian of $\hat{X}$. 

\begin{proposition}\label{proposition: bigonal construction iso}
	The homomorphism $J_{\hat{X}} \rightarrow \Prym_{X,\tau}$ factors through an isomorphism of abelian varieties $\Prym_{\hat{X},\hat{\tau}}^{\vee} \rightarrow \Prym_{X,\tau}$ which identifies the polarizations $\rhodual_{\hat{X}}$ and $\rho_{X}$. 
\end{proposition}
\begin{proof}
	The first claim follows from Part 3 of Proposition \ref{proposition: barth results} applied to the closed embedding $\hat{X} \hookrightarrow \Prym_{X,\tau}$. 
 	Since the polarizations of $\Prym_{X,\tau}$ and $\Prym^{\vee}_{\hat{X},\hat{\tau}}$ are defined by the embedded curve $\hat{X}$, the isomorphism identifies the polarizations.
\end{proof}

We apply the above generalities to the family of curves that concern us.
If $b=(p_2,p_6,p_8,p_{12}) \in \smallB^{\rs}(k)$ then $\smallprojcurve_b$ and $\ellcurve_b$ are of the form
\begin{align*}
	\smallprojcurve_b : y^4+p_2xy^2+p_6y^2 = x^3+p_8x+p_{12},\\
	\ellcurve_b : y^2+p_2xy+p_6y = x^3+p_8x+p_{12}.
\end{align*}
The point $\infty\in \smallprojcurve_b(k)$ is the unique point at infinity, $\tau_b\colon \smallprojcurve_b \rightarrow \smallprojcurve_b$ is the involution sending $(x,y)$ to $(x,-y)$ and $f_b\colon \smallprojcurve_b\rightarrow \ellcurve_b$ the quotient of $\smallprojcurve_b$ by $\tau_b$.  
The divisor $\kappa = 2\infty$ is a theta characteristic fixed by $\curvezeta_b$. 

The proof of Lemma \ref{lemma: bigonal dual is smooth proj genus 3} shows that $\hat{\smallprojcurve}_b$ is isomorphic to the closed subscheme of $\Sym^2 \smallprojcurve_b$ consisting of degree-$2$ divisors $D$ with the property that $f_b(D)\sim 2f_b(\infty)$.  
It follows that $\hat{\smallprojcurve}_b$ has an affine open given by the closed subscheme of $\A^4$ defined by the equations
\begin{displaymath}
\begin{cases}
	y^4+p_2xy^2+p_6y^2 = x^3+p_8x+p_{12},\\
	y'^4+p_2x'y'^2+p_6y'^2 = x'^3+p_8x'+p_{12},\\
	x = x', \\
	y^2+y'^2+p_2x+p_6 = 0,

\end{cases}
\end{displaymath}

quotiented by the involution $(x,y,x',y') \mapsto (x',y',x,y)$. 
This quotient can be realized by introducing the variables $y+y'$ and $yy'$;
a computation then shows that $\hat{\smallprojcurve}_b$ and its quotient by $\hat{\tau}$ are given by (the projective closure of) the equations
\begin{align}
	\hat{\smallprojcurve}_b : (y^2+p_2x+p_6)^2 = -4(x^3+p_8x+p_{12}),\\
	\hat{\ellcurve}_b : (y+p_2x+p_6)^2 = -4(x^3+p_8x+p_{12}). \label{equation: bigonal ell curve}
\end{align}
This construction motivates us to define a $\G_m$-equivariant morphism $\chi \colon \smallB\rightarrow \smallB$ by sending $(p_2,p_6,p_8,p_{12})$ to
\begin{equation}\label{equation: formula bigonal construction}
	 3\cdot \left(-2 p_2,8 p_6-\frac{2 p_2^3}{3},16 p_8 -\frac{p_2^4}{3}+8 p_2p_6,-64p_{12}-\frac{2p_2^6}{27}+\frac{8p_2^3p_6}{3}-16p_6^2+\frac{16p_2^2p_8}{3}  \right) .
	\end{equation}
(We include the factor $3$ in front so that $\chi$ has integer coefficients.)
We have defined $\chi$ so that for all $b\in \smallB^{\rs}(k)$, $C_{\chi(b)}$ is isomorphic to $\hat{C}_b$. We also write $\hat{b}$ for $\chi(b)$, thinking of it as the \emph{`bigonal dual'} of $b$.

\begin{theorem}[The bigonal construction]\label{theorem: summary bigonal construction}
	For any field $k/\Q$ and $b = (p_2,p_6,p_8,p_{12})\in \smallB^{\rs}(k)$:
	\begin{itemize}
		\item The projective curve $\smallprojcurve_{\chi(b)}$ is isomorphic to the projective closure of the curve 
		\begin{equation}\label{equation: bigonal genus 3 curve}
	\left(y^2+p_2x+p_6\right)^2 = -4(x^3+p_8x+p_{12}),
		\end{equation}
		and $\tau\colon \smallprojcurve_{\chi(b)}\rightarrow \smallprojcurve_{\chi(b)}$ maps $(x,y)$ to $(x,-y)$.
		\item There exists an isomorphism $\Prym_{\chi(b)} \simeq \Prym_b^{\vee}$ of $(1,2)$-polarized abelian varieties. 
		\item We have $\chi(\chi(b)) = 18\cdot b$ for all $b\in \smallB$. 
	\end{itemize}
\end{theorem}
\begin{proof}
	Only the last part is not yet established, which follows from an explicit computation. 
\end{proof}

For any $\smallB$-scheme $U$ we define the $\smallB$-scheme $\hat{U}$ as the pullback of $U\rightarrow \smallB$ along $\chi\colon \smallB\rightarrow \smallB$.
In particular, we obtain the $\smallB$-schemes $\hat{\smallprojcurve}, \hat{\ellcurve}, \hat{\Jac}, \hat{\Prym}$.
In this notation, one can prove that there exists an isomorphism $\Prym^{\vee} \simeq \hat{\Prym}$ of polarized abelian schemes over $\smallB^{\rs}$. 
However, we will not need this fact in what follows.

\begin{corollary}\label{corollary: subgroups Prym[2] using bigonal}
 There exists an exact sequence of finite \'etale group schemes over $\smallB^{\rs}$: 
	\begin{align*}
		0 \rightarrow \ellcurve[2] \rightarrow \Prym[2]\rightarrow \hat{E}[2] \rightarrow 0
	\end{align*}
	isomorphic to the exact sequence 
	$$
	0 \rightarrow \Prym[\rho] \rightarrow \Prym[2] \xrightarrow{\rho} \Prym^{\vee}[\rhodual]\rightarrow 0.
	$$
Moreover, the $\smallB^{\rs}$-groups $\ellcurve[2]$ and $\hat{E}[2]$ are not isomorphic, even after base change to $k$ for any field extension $k/\Q$.
	
\end{corollary}
\begin{proof}
	Since $\smallB^{\rs}$ is normal, it suffices to prove the corollary over the generic point.
	The second exact sequence of the corollary follows from the identity $\rhodual \circ \rho = [2]$. 
	We have seen in \S\ref{subsection: def prym variety} (Property \ref{enum: prym 5}) that the kernel $\Prym[\rho]$ of $\rho$ is identified with $\ellcurve[2]$. 
	Since $\Prym^{\vee}[\rhodual] \simeq \hat{\Prym}[\rho]$ by Theorem \ref{theorem: summary bigonal construction}, we see that $\Prym^{\vee}[\rhodual]\simeq \hat{\ellcurve}[2]$.
	The last claim follows from the last claim of Corollary \ref{corollary: subgroups Jac[2]}.

\end{proof}

\subsection{The compactified Prym variety}\label{subsection: compactifications}

In Section \ref{subsection: def prym variety} we have constructed a family of abelian varieties $\Prym \rightarrow \smallB^{\rs}$. 
In this section we construct a projective scheme $\CPrym \rightarrow \smallB$ containing $\Prym$ as a dense open subscheme. 
The properties of $\CPrym$ (which are summarized in Proposition \ref{proposition: good compactifications exist}) will be the crucial geometric input for the construction of integral orbit representatives for $(\smallG,\smallV)$ in \S\ref{subsection: integral reps: the $2$-Selmer case}.

Recall that $\bigprojcurve\rightarrow \bigB$ is the family of projective curves given by Equation (\ref{equation : E6 family middle of paper}).
Let $\bigJac\rightarrow \bigB^{\rs}$ be the relative Jacobian of its smooth part, a smooth and proper morphism. 
In \cite[\S4.3]{Laga-E6paper}, a proper morphism $\bigCJac\rightarrow \bigB$ is constructed which parametrizes torsion-free rank-1 sheaves on the fibres of $\bigprojcurve\rightarrow \bigB$. (In fact, in that paper the compactified Jacobian $\intCJac_{E}$ was constructed over $\Spec \Z[1/N]$ for some $N\geq 1$; we define $\bigCJac$ as the $\Q$-fibre of $\intCJac_{E}$.)
We state some of its properties here, referring to  \cite[Corollary 4.13]{Laga-E6paper} for proofs and references.
\begin{proposition}\label{proposition: properties E6 comp jacobian}
	\begin{enumerate}
		\item For any $\bigB$-scheme $T$, the $T$-points of $\bigCJac$ are in  natural bijection with the set of isomorphism classes of locally finitely presented $\O_{\bigprojcurve\times T}$-modules $\sh{F}$, flat over $T$, with the property that $\sh{F}_t$ is torsion-free rank $1$ of degree zero for every geometric point $t$ of $T$, and that there exists an isomorphism of $\O_T$-modules $\infty_T^*\sh{F}\simeq \O_T$, where $\infty\colon \bigB \rightarrow \bigprojcurve$ denotes the section at infinity.
		\item The morphism $\bigCJac\rightarrow \bigB$ is flat, projective and its restriction to $\bigB^{\rs}\subset \bigB$ is isomorphic to $\bigJac\rightarrow \bigB^{\rs}$. 
		\item The variety $\bigCJac\rightarrow \Spec \Q$ is smooth. 
	\end{enumerate}
\end{proposition}

Recall from \S\ref{subsection: a family of curves} that the involution $\curvezeta: (x,y) \mapsto (x,-y)$ of $\bigprojcurve$ lifts the involution $(-1) \colon \bigB\rightarrow \bigB$. It induces an involution $\curvezeta^*$ of $\bigCJac$, sending a rank $1$ torsion-free sheaf $\sh{F}$ to its pullback $\curvezeta^*(\sh{F})$. 

On the other hand, we may construct a different involution of $\bigCJac$ extending $[-1]: \bigJac \rightarrow \bigJac$, as follows.
If $\sh{F}$ is a coherent sheaf on a scheme $X$, we define $\sh{F}^{\vee} \coloneqq \sh{H} \kern -.5pt om(\sh{F},\O_X)$. 

\begin{lemma}
    Let $T$ be a $\bigB$-scheme and $\sh{F}$ an $\O_{\bigprojcurve\times T}$-module, corresponding to a $T$-point of $\bigCJac$. 
    Then the $\O_{\bigprojcurve\times T}$-module $\sh{F}^{\vee}$ corresponds to a $T$-point of $\bigCJac$, and the corresponding morphism $\bigCJac\rightarrow \bigCJac,\, \sh{F}\mapsto \sh{F}^{\vee}$ is an involution.
\end{lemma}
\begin{proof}
    Since the fibres of $\bigprojcurve \rightarrow \bigB$ are all Gorenstein curves (being complete intersections), \cite[Lemma 1.1(a)]{Hartshorne-generalizeddivisorsGorensteincurves} shows that $\sh{E} \kern -.5pt xt_{\O_{\smallprojcurve_{\mathrm{E},t}}}^1(\sh{F}_t,\O_{\smallprojcurve_{\mathrm{E},t}})=0$ for all geometric points $t$ of $T$.
    Therefore \cite[Theorem 1.10(ii)]{AltmanKleiman-CompactifyingThePicardScheme} implies that $\sh{F}^{\vee}$ is locally finitely presented and flat over $T$, and that $(\sh{F}^{\vee})_S\simeq (\sh{F}_S)^{\vee}$ for every morphism $S\rightarrow T$.
    It follows that $(\sh{F}^{\vee})_t=\sh{F}_t^{\vee}$ is torsion-free rank $1$ of degree zero since the same is true for $\sh{F}_t$.
    Moreover $\infty_T^*(\sh{F}^{\vee})\simeq (\infty_T^*\sh{F})^{\vee} \simeq \O_T^{\vee}\simeq \O_T$. 
    Therefore by Proposition \ref{proposition: properties E6 comp jacobian}, $\sh{F}^{\vee}$ corresponds to a $T$-point of $\bigCJac$.
    
    It remains to prove that the natural map $\sh{F} \rightarrow \sh{F}^{\vee\vee}$ is an isomorphism. 
    Since the formation of $\sh{F}^{\vee\vee}$ commutes with base change, we may assume that $T$ is the spectrum of an algebraically closed field. 
    In this case the claim follows from \cite[Lemma 1.1(b)]{Hartshorne-generalizeddivisorsGorensteincurves}.
\end{proof}

Write $\mu$ for the composite of the commuting involutions $\tau^*$ and $
\sh{F}\mapsto \sh{F}^{\vee}$.
Write $\CJac \rightarrow \smallB$ for the restriction of $\bigCJac$ to $\smallB\hookrightarrow \bigB$. 

\begin{definition}
	We define the \define{compactified Prym variety} $\CPrym\rightarrow \smallB$ as the $\mu$-fixed points of the morphism $\bigCJac \rightarrow \bigB$. 
\end{definition} 
The scheme $\CPrym$ is a closed subscheme of $\bigCJac$ so the morphism $\CPrym \rightarrow \smallB$ is projective.
By definition of $\Prym$ (cf. Equation (\ref{equation: def prym of our family})) the restriction of $\CPrym$ to $\smallB^{\rs}\subset \smallB$ is isomorphic to $\Prym$.

\begin{lemma}\label{lemma: compactified prym smooth}
    The scheme $\CPrym$ is smooth over $\Q$.
\end{lemma}
\begin{proof}
	Apply Lemma \ref{lemma: fixed points smooth morphism is smooth} to the $\Z/2\Z$-action of $\mu$ on the smooth morphism $\bigCJac \rightarrow \Spec \Q$. 
\end{proof}

We now analyze the irreducible components of $\CPrym$. The following lemma contains the key calculation of the central fibre of $\CPrym \rightarrow \smallB$.

\begin{proposition}\label{proposition: compactprym central fibre}
The fibre of the morphism $\CPrym \rightarrow \smallB$ above $0$ is geometrically irreducible of dimension $2$. 
\end{proposition}
\begin{proof}

In the course of the proof we may and will assume that all schemes are base changed to $\C$ and by abuse of notation will identify them with their set of complex points.  
Write $\Jac_0$ for the identity component of the Picard scheme of the projective curve $\smallprojcurve_0$ given by the equation $(y^4=x^3)$, an open subscheme of $\CJac_0$ stable under $\mu$. 
To prove the proposition it suffices to prove that $\Prym_0\coloneqq \Jac_0^{\mu}$ is irreducible of dimension $2$ and $\Prym_0 \subset \CPrym_0$ is a dense open subscheme.
Because the normalization $\pi \colon \tilde{\smallprojcurve}_0\rightarrow \smallprojcurve_0$ is rational and $\smallprojcurve_0$ is Gorenstein, we may appeal to the results of \cite{Beauville-rationalcurvesK3} (originally due to Rego \cite{Rego-CompactifiedJacobian}) to describe $\CJac_0$ explicitly.
 
Define the local rings $\tilde{\O} = \C[[t]]$, $\O = \C[[t^3,t^4]]\subset \tilde{\O}$ and the truncated versions $\tilde{A} = \tilde{\O}/t^6$ and $A = \image( \O \rightarrow \tilde{\O}/t^6)\subset \tilde{A}$. Then $\O$ is the completed local ring of $\smallprojcurve_0$ at the origin and $\tilde{\O}$ its normalization. 
Let $Gr(3,\tilde{A})$ be the Grassmannian parametrizing $3$-dimensional subspaces of $\tilde{A}$. 
Let $\mathcal{M} \subset Gr(3,\tilde{A})$ be the reduced closed subscheme parametrizing those subspaces which are stable under the action of $A$.
The map $M \mapsto M\otimes_{\O}A$ establishes a bijection between the $\O$-submodules $M$ of $\tilde{\O}$ with $\dim_{\C} \tilde{\O}/M=3$ and $\mathcal{M}$ (by \cite[Lemma 1.1(iv)]{GreuelPfister-Modulispaces} and the fact that $\dim_{\C} \tilde{\O}/\O=3$), whose inverse we denote by $M\mapsto M^{\O}$. 
We have a natural morphism of $\O_{\smallprojcurve_0}$-modules $\pi_*\O_{\tilde{C}_0}\rightarrow \tilde{A}$, where $\tilde{A}$ is considered as the structure sheaf of the degree-$6$ divisor supported at the preimage under $\pi$ of the singular point. 
The assignment $$M \mapsto \sh{F}_M\coloneqq \ker(\pi_*\O_{\tilde{\smallprojcurve}_0} \rightarrow \tilde{A}/M)$$
defines a morphism $e\colon \mathcal{M}\rightarrow \CJac_0$ which is bijective and proper (\cite[Proposition 3.7]{Beauville-rationalcurvesK3}), hence a homeomorphism in the Zariski topology. 
The sheaf $\sh{F}_M$ is invertible if and only if $M^{\O}$ is a cyclic $\O$-module; the locus of such $M$ define an open subscheme $\mathcal{M}^{\circ}\subset \mathcal{M}$. 
Let $\tau\colon \tilde{A} \rightarrow \tilde{A}$ be the $\C$-algebra homomorphism sending $t$ to $-t$. 
For $M\in \mathcal{M}(\C)$ define $M^{\vee}  = \{ x\in \tilde{A} \mid x\cdot M \subset A\}$ and $\tau^*M  = \{ \tau(m) \mid m\in M\}$; they define involutions $(-)^{\vee}$ and $\tau^*$ of $\mathcal{M}$ with composite $\mu$. 
Since $\sh{F}_{\mu(M)} \simeq \mu(\sh{F}_M)$, it suffices to prove that $\mathcal{M}_P^{\circ} \coloneqq \left(\mathcal{M}^{\circ}\right)^{\mu}$ is irreducible, two-dimensional and dense in $\mathcal{M}_P\coloneqq \mathcal{M}^{\mu}$. 

The map $a\mapsto A\cdot a$ defines a bijection $\tilde{A}^{\times}/A^{\times}\rightarrow \mathcal{M}^{\circ} $. 
We have a group isomorphism $\G_a^3\rightarrow  \tilde{A}^{\times}/A^{\times}$ given by sending $(a_1,a_2,a_5)$ to the coset of $$\text{exp}(a_1t+a_2t^2+a_5t^5) = 1+(a_1t+a_2t^2+a_5t^5)+(a_1t+a_2t^2+a_5t^5)^2/2+\dots+(a_1t+a_2t^2+a_5t^5)^5/5!.$$ 
The composite $\G_a^3\rightarrow \mathcal{M}^{\circ}$ is an isomorphism of varieties and gives $\mathcal{M}^{\circ}$ the structure of an algebraic group which acts on $\mathcal{M}$. 
The restriction of $\mu$ to $\mathcal{M}^{\circ}$ corresponds to the involution $(a_1,a_2,a_5) \mapsto (a_1,-a_2,a_5)$ under the above isomorphism.
Therefore $\mathcal{M}_P^{\circ}$ is isomorphic to $\G_a^2$, hence irreducible and two-dimensional; it remains to prove that it is dense in $\mathcal{M}_{\Prym}$.
The orbits of the action of $\mathcal{M}^{\circ}$ stratifies $\mathcal{M}$ into affine cells which are described in \cite[\S4]{Cook-SimpleSingularitiesCompJac}. They correspond to isomorphism classes of torsion-free rank $1$ $\O$-modules and their properties are described in Table \ref{table 4}. 
The second column gives an $\mathcal{O}$-module representative $M^{\mathcal{O}}$ for some $M\in X_i$; the third column depicts the powers of $t$ generating $M$. 

Since $\mu$ preserves $\mathcal{M}^{\circ}$ it permutes the strata.
By dimension reasons it can only permute $X_2$ and $X_3$. 
Since the dual of $tA$ is $t^2A+t^3A$ and $\tau$ fixes $tA$ we see that $\mu(X_2)=X_3$. 
Therefore $\mathcal{M}_{\Prym} = \mathcal{M}^{\circ}_P \sqcup X_4^{\mu} \sqcup X_5^{\mu}$. 
So it will be enough to show that the closure of $\mathcal{M}^{\circ}_P$ contains $X_4^{\mu} \sqcup X_5^{\mu}$.

\begin{table}
	\centering
	\begin{tabular}{|c|c c c c|}
		\hline
		Stratum & Module & Type & Dimension & Image under $\mu$ \\
		\hline
		$X_1$ & $\O$ & \{0,3,4\} & $3$ & $X_1$\\
		$X_2$ & $t\O+t^6\O$ & \{1,4,5\} & $2$ & $X_3$ \\
		$X_3$ & $t^2\O+t^3\O$ & \{2,3,5\} & $2$ & $X_2$\\
		$X_4$ & $t^2\O+t^4\O$ & \{2,4,5\} & $1$ & $X_4$ \\
		$X_5$ & $t^3\O+t^4\O+t^5\O$ & \{3,4,5\} & $0$ & $X_5$ \\
		\hline
	\end{tabular}
	\caption{Stratification of $M$}
	\label{table 4}
\end{table}

Using the description of $\mathcal{M}^{\circ}$ given above and the exponential map, every element of $\mathcal{M}_P^{\circ}$ is an $A$-module generated by
$$
\left(1+at+\frac{a^2}{2}t^2+\frac{a^3}{6}t^3+\frac{a^4}{24}t^4+\left(\frac{a^5}{120}+b\right)t^5 \right)
$$
for some $a,b \in \C$.
Using the Plucker coordinates $\{t^i\wedge t^j \wedge t^k \}$ in $Gr(3,\tilde{A})$, one can compute that the closure of $\mathcal{M}_P^{\circ}$ contains $\lambda (t^2\wedge t^4\wedge t^5)+\mu (t^3\wedge t^4\wedge t^5)$ for all $\lambda, \mu \in \C$.
Since every element of $X_4$ or $X_5$ is of this form, this proves the proposition. 
\end{proof}

Recall that we have defined a $\G_m$-action on $\smallprojcurve\rightarrow \smallB$ in \S\ref{subsection: a family of curves} after Proposition \ref{proposition: bridge F4 rep and F4 curves}.
By functoriality this induces a $\G_m$-action on $\CPrym$ such that the morphism $\CPrym \rightarrow \smallB$ is $\G_m$-equivariant. 
The following fact will be used in the next three lemmas: if $Z\subset \smallB$ is a closed, nonempty and $\G_m$-invariant subscheme, then it contains the central point $0$.

\begin{lemma}\label{lemma: compactified prym irreducible}
	The scheme $\CPrym$ is geometrically irreducible.
\end{lemma}
\begin{proof}
	Since $\CPrym$ is smooth (Lemma \ref{lemma: compactified prym smooth}), the irreducible components of $\CPrym_{\overbar{\Q}}$ coincide with its connected components so in particular are disjoint.
	The image of each connected component under $\CPrym \rightarrow \smallB$ is closed (using properness) and $\G_m$-invariant, hence contains the central point. 
	But $\CPrym_{0,\overbar{\Q}}$ is irreducible by Proposition \ref{proposition: compactprym central fibre} so there exists at most one such connected component, as required.
\end{proof}

\begin{lemma}\label{lemma: compactified prym flat over B}
	The morphism $\CPrym \rightarrow \smallB$ is flat. 
\end{lemma}
\begin{proof}
	We first claim that all the fibres of $\CPrym \rightarrow \smallB$ are $2$-dimensional. Since $\CPrym \rightarrow \smallB$ is proper, the fibre dimension of this morphism is upper semicontinuous on $\smallB$ \cite[Corollaire 13.1.5]{EGAIV-3}. The general fibre is $2$-dimensional; let $Z\subset \smallB$ be the closed subset where the fibre has larger dimension. The $\G_m$-action on $\CPrym\rightarrow \smallB$ shows that this locus is invariant under $\G_m$ hence it must contain the central point $0$, if it is non-empty. But Proposition \ref{proposition: compactprym central fibre} shows that $0\not\in Z$, proving the claim.
	
	The lemma now follows from the smoothness and irreducibility of $\CPrym$ (Lemmas \ref{lemma: compactified prym smooth} and \ref{lemma: compactified prym irreducible}) and $\smallB$ and Miracle Flatness \cite[Theorem 23.1]{Matsumura-CommutativeRingTheory}.
\end{proof}

\begin{lemma}\label{lemma: compactified prym integral fibres}
	The fibres of the morphism $\CPrym\rightarrow \smallB$ are geometrically integral.
\end{lemma}
\begin{proof}
	We first claim that $\CPrym_0$ is geometrically reduced. 
	Proposition \ref{proposition: compactprym central fibre} shows that $\CPrym_0$ contains a smooth open dense subscheme $\Prym_0$. 
	Therefore $\CPrym_0$ is generically reduced and it suffices to prove that it is geometrically Cohen--Macaulay. But since $\CPrym \rightarrow \smallB$ is flat (Lemma \ref{lemma: compactified prym flat over B}) and $0\hookrightarrow \smallB$ is a complete intersection, the pullback $\CPrym_0 \hookrightarrow \CPrym$ is a complete intersection. Since $\CPrym$ is smooth (Lemma \ref{lemma: compactified prym smooth}), $\CPrym_{0,\overbar{\Q}}$ is a local complete intersection hence Cohen--Macaulay.
	We conclude that $\CPrym_0$ is geometrically reduced hence by Proposition \ref{proposition: compactprym central fibre} geometrically integral.
	
	The proposition now follows from the contracting $\G_m$-action. Indeed, the locus $Z$ of elements of $\smallB$ above which the fibre fails to be geometrically integral is closed \cite[Théorème 12.2.1(x)]{EGAIV-3} and $\G_m$-invariant.
	Since we have just shown that $Z$ does not contain the central point, it must be empty.
\end{proof}

We summarize the properties of $\CPrym$ for later reference in the following proposition.

\begin{proposition}\label{proposition: good compactifications exist}
	The morphism $\CPrym \rightarrow \smallB$ constructed above is flat, projective and its restriction to $\smallB^{\rs}$ is isomorphic to $\Prym$. Moreover $\CPrym$ is smooth and geometrically integral. The locus of $\CPrym$ where the morphism $\CPrym \rightarrow \smallB$ is smooth is an open subset whose complement has codimension at least two. 
\end{proposition}
\begin{proof}
    The only thing that remains to be proven is the statement about the smooth locus of $\CPrym \rightarrow \smallB$; denote this morphism by $\phi$. 
    Let $Z\subset \CPrym$ be the (reduced) closed subscheme where $\phi$ fails to be smooth. 
    The smoothness of $\Prym \rightarrow \smallB^{\rs}$ shows that $Z_b$ is empty if $b\in \smallB^{\rs}$. 
    Moreover since the fibres of $\phi$ are geometrically integral by Lemma \ref{lemma: compactified prym integral fibres}, the smooth locus of $\CPrym_b$ is nonempty and $Z_b\subset \CPrym_b$ is a proper closed subset of smaller dimension for every $b\in B$. 
    Combining the last two sentences proves the statement. 
\end{proof}

The discussion of this section has another geometric consequence, which will be useful in \S \ref{subsection: neron component groups pryms}.

\begin{proposition}\label{proposition: generalized prym variety has integral fibres}
Let $\Pic^0_{\smallprojcurve/\smallB} \rightarrow \smallB$ be the identity component of the relative Picard scheme of $\smallprojcurve\rightarrow \smallB$ \cite[\S9.3, Theorem 1]{BLR-NeronModels}.
Then the fibres of $\Pic^0[1+\curvezeta^*]\rightarrow \smallB$ are geometrically integral.
\end{proposition}
\begin{proof}
	By construction of $\CPrym$, there exists a morphism of $\smallB$-schemes $\Pic^0[1+\curvezeta^*]\rightarrow \CPrym$ which is an open immersion.
	Therefore for every $b\in \smallB$, $\Pic^0[1+\curvezeta^*]_b$ is a non-empty open subset of $\CPrym_b$. 
	Since the fibres of $\CPrym \rightarrow \smallB$ are geometrically integral (Lemma \ref{lemma: compactified prym integral fibres}), the proposition follows.
\end{proof}

\subsection{The discriminant polynomial}\label{subsection : discriminant polynomial}

We give an explicit description of the discriminant polynomial $\smalldisc \in \Q[\smallV]^{\smallG} = \Q[\smallB]$ introduced in \S\ref{subsection: definition of the representations} before Proposition \ref{proposition: equivalences regular semisimple}.
Recall that we have fixed an isomorphism $\Q[\smallV]^{\smallG} \simeq \Q[p_2,p_6,p_8,p_{12}]$ from Proposition \ref{proposition: bridge F4 rep and F4 curves}, so we consider $\Delta$ as a polynomial in $p_2,p_6,p_8,p_{12}$. 
Since the $F_4$ root system has $48$ roots, $\Delta$ is homogeneous of degree $48$ with respect to the $\G_m$-action on $\smallB$. 
 
Set $\Delta_{\hat{\ellcurve}} \coloneqq 4p_8^3+27p_{12}^2$ and $\Delta_{\ellcurve} \coloneqq \Delta_{\hat{\ellcurve}}\circ \chi$, where $\chi$ is defined by Formula (\ref{equation: formula bigonal construction}), both elements of $\Q[\smallB]$. 
Then $\Delta_{\ellcurve}$ and $\Delta_{\hat{\ellcurve}}$ are up to elements of $\Q^{\times}$ the discriminants of the curves ${\Cellcurve}\rightarrow \smallB$ and $\hat{\Cellcurve}\rightarrow \smallB$.

\begin{lemma}\label{lemma: discriminant polynomial F4}
	The polynomial $\Delta \in \Q[p_2,p_6,p_8,p_{12}]$ equals, up to an element of $\Q^{\times}$, the polynomial $\Delta_{\ellcurve}\cdot \Delta_{\hat{\ellcurve}}$. 
	In other words, there exists a constant $A_0\in\Q^{\times}$ such that
	\begin{align*}
		\Delta(b) = A_0\left(4p_8(\hat{b})^3+27p_{12}(\hat{b}) \right) \left(4p_8(b)^3+27p_{12}(b) \right).
	\end{align*}
\end{lemma}
\begin{proof}
    It suffices to prove the claim when base changed to an algebraically closed field $k/\Q$. 
    The polynomials $\Delta$ and $\Delta_{\ellcurve}\cdot \Delta_{\hat{\ellcurve}}$ both have degree $48$.
    Moreover $\Delta_{\ellcurve}$ and $\Delta_{\hat{\ellcurve}}$ are irreducible coprime polynomials in $k[\smallB]$. 
	So to prove the claim it suffices to prove that the vanishing loci of $\Delta$ and $\Delta_{\ellcurve}\cdot \Delta_{\hat{\ellcurve}}$ agree. 
	By Proposition \ref{proposition: bridge F4 rep and F4 curves}, if $b\in \smallB(k)$ then $\Delta(b)\neq 0$ if and only if the curve $(y^4+p_2xy^2+p_6y^2 = x^3+p_8x+p_{12})$ is smooth. 
	By the Jacobian criterion for smoothness, this happens if and only if the curve $y^2+p_2xy+p_6y = x^3+p_8x+p_{12}$ is smooth and the polynomial $x^3+p_8x+p_{12}$ has no multiple roots. The lemma then follows from the explicit descriptions of $\Cellcurve\rightarrow \smallB$ and $\hat{\Cellcurve} \rightarrow \smallB$ given by Equations (\ref{equation: elliptic curve}) and (\ref{equation: bigonal ell curve}) respectively.
\end{proof}

\begin{remark}
    The factorization of $\Delta$ into a product of two degree-$24$ polynomials of Lemma \ref{lemma: discriminant polynomial F4} can be interpreted Lie-theoretically. 
    It corresponds to the fact that the Weyl group $W(\smallH,\smallT)$ has two orbits on $\Phi(\smallH,\smallT)$, namely an orbit consisting of the $24$ short roots and one consisting of the $24$ long roots.
    It is true (although we do not prove this) that $\Delta_{\hat{\ellcurve}}$ corresponds the short root orbit and $\Delta_{\ellcurve}$ to the long root orbit.
\end{remark}

\section{Orbit parametrization}\label{section: orbit parametrization}

In this section we construct, for each $b\in \smallB^{\rs}(\Q)$, an embedding of $\Sel_2 \Prym_b$ inside the set of $\smallG(\Q)$-orbits of $\smallV(\Q)$ with invariants $b$.
Moreover we introduce a different representation $(\rhoG,\rhoV)$ and similarly prove that $\Sel_{\rhodual} \Prym_b^{\vee}$ embeds in its rational orbits.

We first recall a well-known lemma which gives a cohomological description of orbits.

\begin{lemma}\label{lemma: AIT full generality}
Let $G\rightarrow S$ be a smooth affine group scheme.
Suppose that $G$ acts on the $S$-scheme $X$ and let $e\in X(S)$.
Suppose that the action map $m\colon G\rightarrow X, g\mapsto g\cdot e$ is smooth and surjective.
Then the assignment $x\mapsto m^{-1}(x)$ induces a bijection between the set of $G(S)$-orbits on $X(S)$ and the kernel of the map of pointed sets $\HH^1(S,Z_{G}(e)) \rightarrow \HH^1(S,G)$.  
\end{lemma}
\begin{proof}
	This is \cite[Exercise 2.4.11]{Conrad-reductivegroupschemes}:
	the conditions imply that $X\simeq G/Z_{G}(e)$ and since $G$ and $Z_G(e)$ (the fibre above $e$ of a smooth map) are $S$-smooth we can replace fppf cohomology by \'etale cohomology. 
\end{proof}

Lemma \ref{lemma: AIT full generality} has the following concrete consequence. Let $k$ be a field and $G/k$ a smooth algebraic group which acts on a $k$-scheme $X$. Suppose that $e\in X(k)$ has smooth stabilizer $Z_G(e)$ and the action of $G(k^s)$ on $X(k^s)$ is transitive, where $k^s$ denotes a separable closure of $k$. 
Then the $G(k)$-orbits of $X(k)$ are in bijection with $\ker(\HH^1(k,Z_G(e))\rightarrow \HH^1(k,G))$. 
This fact allows us to make the connection with Galois cohomology and lies at the basis of our orbit parametrizations in \S\ref{subsection: twisting and embedding the selmer group} and \S\ref{subsection: embedding the rho selmer group}.

\subsection{Embedding the 2-Selmer group}\label{subsection: twisting and embedding the selmer group}

The purpose of this section is to prove Theorem \ref{theorem: inject 2-descent into orbits} and its consequence, Corollary \ref{corollary: Sel2 embeds}.
The essential input is a similar orbit parametrization obtained in the $E_6$ case in \cite{Laga-E6paper}.
For any morphism $b\colon S\rightarrow \smallB$ we write $\smallV_b$ for the fibre of $\smallpi \colon \smallV\rightarrow \smallB$ under $b$ and similarly for $\bigV$.

\begin{lemma}\label{lemma: AIT}
	Let $R$ be a $\Q$-algebra and $b\in \smallB^{\rs}(R)$.
	Then there are canonical bijections of sets
	\begin{enumerate}
		\item $ \smallG(R) \backslash \smallV_b(R) \simeq  \ker\left(\HH^1(R,\Prym_b[2]) \rightarrow \HH^1(R,\smallG)\right).$
		\item $ \bigG(R)\backslash \smallV_{\mathrm{E},b}(R) \simeq \ker\left(\HH^1(R,\Jac_b[2]) \rightarrow \HH^1(R,\bigG)\right).$
	\end{enumerate}
	The reducible orbits $\smallG(R)\cdot \smallsigma(b)$ and $\bigG(R)\cdot \bigsigma(b)$ correspond to the trivial element in $\HH^1(R,\Prym_b[2])$ and $\HH^1(R,\Jac_b[2])$ respectively. Moreover the following diagram is commutative:
	\begin{center}
	\begin{tikzcd}
			\smallG(R) \backslash \smallV_b(R) \arrow[d] \arrow[r] & \bigG(R) \backslash \smallV_{\mathrm{E},b}(R) \arrow[d] \\
			\HH^1(R,\Prym_b[2]) \arrow[r] & \HH^1(R,\Jac_b[2]) 
	\end{tikzcd}
\end{center}
	Here the horizontal maps are induced by the natural inclusions and the vertical maps are the injections induced by the above bijections.
\end{lemma}
\begin{proof}
	We consider the case of $(\smallG,\smallV)$, the case of $(\bigG,\bigV)$ being analogous. 
	The bijection then follows from Lemma \ref{lemma: AIT full generality} applied to the action of $\smallG_{\smallB^{\rs}}$ on $\smallV^{\rs}$. 
	Indeed, the action map $\smallG \times \smallB^{\rs} \rightarrow \smallV^{\rs}, (g,b) \mapsto g\cdot \smallsigma(b)$ is \'etale (Proposition \ref{proposition: kostant section}) and it is surjective by Part 2 of Proposition \ref{prop : graded chevalley}. 
	Moreover we have an isomorphism $Z_{\smallG}(\sigma|_{\smallB^{\rs}})\simeq \Prym[2]$ by Proposition \ref{proposition: iso centralizer in G and prym[2]}.
	Pulling back along $b\colon \Spec R\rightarrow \smallB^{\rs}$ gives the desired bijection. 
	
	The claim about $\smallG(R)\cdot \sigma(b)$ follows from the explicit description of the bijection of Lemma \ref{lemma: AIT full generality}. 
	The commutative diagram follows from the definition of the pushout of torsors and the compatibility between the isomorphisms $Z_{\bigG}(\smallsigma|_{\smallB^{\rs}}) \simeq \Jac[2]$ and $Z_{\smallG}(\smallsigma|_{\smallB^{\rs}}) \simeq \Prym[2]$. 
\end{proof}

\begin{lemma}\label{lemma: H1 simply connected form trivial}
	Let $\smallG^{sc}\rightarrow \smallG$ be the simply connected cover of $\smallG$.
	Let $R$ be a $\Q$-algebra such that every locally free $R$-module of constant rank is free. Then the pointed set $\HH^1(R,\smallG^{sc})$ is trivial. 
\end{lemma}
\begin{proof}
	We have $\smallG^{sc} \simeq \SL_2 \times \Sp_6$ (Proposition \ref{proposition: table properties algebraic groups}). The result now follows from the triviality of $\HH^1(R,\SL_2)$ (by Hilbert's theorem 90) and $\HH^1(R,\Sp_6)$ \cite[Lemma 3.12]{Laga-E6paper}. 
\end{proof}

\begin{lemma}\label{lemma: relation F4 E6 orbits}
Let $R$ be a $\Q$-algebra such that every locally free $R$-module of constant rank is free. Then the natural map of pointed sets $\HH^1(R,\smallG) \rightarrow \HH^1(R,\bigG)$ has trivial kernel.
\end{lemma}
\begin{proof}
Let $\bigG^{sc}\rightarrow \bigG$ be the simply connected cover of $\bigG$.
We have a commutative diagram with exact rows over $R$:

\begin{center}
	\begin{tikzcd}
			1 \arrow[r] & \mu_2 \arrow[r]                 & \bigG^{sc} \arrow[r]    & \bigG \arrow[r]                        & 1 \\
			1 \arrow[r] & \mu_2 \arrow[r] \arrow[u, "="'] & \smallG^{sc} \arrow[r] \arrow[u] & \smallG \arrow[r] \arrow[u] & 1
	\end{tikzcd}
\end{center}
Here the maps are the natural ones and we omit the subscript $R$ from the notation. Considering the long exact sequence in cohomology we obtain a commutative diagram with exact rows of pointed sets:

\begin{center}
	\begin{tikzcd}
			 \HH^1(R,\bigG^{sc}) \arrow[r]                 & \HH^1(R,\bigG) \arrow[r]    & \HH^2(R,\mu_2)                      \\
			 \HH^1(R,\smallG^{sc}) \arrow[r] \arrow[u] & \HH^1(R,\smallG) \arrow[r] \arrow[u] & \HH^2(R,\mu_2)\arrow[u, "="']
	\end{tikzcd}
\end{center}	
Lemma \ref{lemma: Haar measure iwasawa decomposition} implies that $\HH^1(R,\smallG^{sc})$ is trivial. The exactness of the rows and the commutativity of the diagram imply that the kernel of the map $\HH^1(R,\smallG) \rightarrow \HH^1(R,\bigG) $ is trivial, as desired.

\end{proof}

\begin{theorem}\label{theorem: inject 2-descent into orbits}
	Let $R$ be a $\Q$-algebra such that every locally free $R$-module is free and $b\in \smallB^{\rs}(R)$. Then there is a canonical injection $\eta_b : \Prym_b(R)/2\Prym_b(R) \hookrightarrow \smallG(R)\backslash \smallV_b(R)$ compatible with base change. 
	Moreover the map $\eta_b$ sends the identity element to the orbit of $\smallsigma(b)$.
\end{theorem}
\begin{proof}
If $A\in \Prym_b(R)$, define $\eta_b(A)\in \HH^1(R,\Prym_b[2])$ as the image of $A$ under the $2$-descent map, namely the isomorphism class of the $\Prym_b[2]$-torsor $[2]^{-1}(A)$. 
It suffices to prove, under the identification of Lemma \ref{lemma: AIT}, that the class $\eta_b(A)$ is killed under the map $\HH^1(R,\Prym_b[2]) \rightarrow \HH^1(R,\smallG)$.
By Lemma \ref{lemma: relation F4 E6 orbits} it suffices to prove that this class is trivial in $\HH^1(R,\bigG)$.
By the parametrization of orbits of the representation $\bigV$ \cite[Theorem 3.13]{Laga-E6paper}, the composite $\Jac_b(R)/2\Jac_b(R) \hookrightarrow \HH^1(R,\Jac_b[2]) \rightarrow \HH^1(R,\bigG)$ is trivial.
The commutative diagram
\begin{center}
	\begin{tikzcd}
			\Prym_b(R)/2\Prym_b(R) \arrow[d] \arrow[r] & \Jac_b(R)/2\Jac_b(R) \arrow[d] \\
			\HH^1(R,\Prym_b[2]) \arrow[d]\arrow[r] & \HH^1(R,\Jac_b[2]) \arrow[d]   \\
			\HH^1(R,\smallG) \arrow[r] & \HH^1(R,\bigG)                 
	\end{tikzcd}
\end{center}
then implies the theorem.
\end{proof}

We obtain the following concrete corollary of the parametrization of $2$-Selmer elements.

\begin{corollary}\label{corollary: Sel2 embeds}
	Let $b\in \smallB^{\rs}(\Q)$ and write $\Sel_2 \Prym_b$ for the $2$-Selmer group of $\Prym_b$.
	Then the injection $\eta_b\colon \Prym_b(\Q)/2\Prym_b(\Q) \hookrightarrow \smallG(\Q)\backslash \smallV_b(\Q)$ of Theorem \ref{theorem: inject 2-descent into orbits} extends to an injection
	$$\Sel_2 \Prym_b \hookrightarrow \smallG(\Q)\backslash \smallV_b(\Q).$$
\end{corollary}
\begin{proof}
	To prove the corollary it suffices to prove that $2$-Selmer elements in $\HH^1(\Q,\Prym_b[2])$ are killed under the natural map $\HH^1(\Q,\Prym_b[2]) \rightarrow \HH^1(\Q,\smallG)$. 
	By definition, an element of $\Sel_2 \Prym_b$ consists of a class in $\HH^1(\Q,\Prym_b[2])$ whose restriction to $\HH^1(\Q_v,\Prym_b[2])$ lies in the image of the $2$-descent map for every place $v$. 
	By Theorem \ref{theorem: inject 2-descent into orbits} the image of such an element in $\HH^1(\Q_v,\smallG)$ is trivial for every $v$. Since the restriction map $\HH^2(\Q,\mu_2) \rightarrow \prod_{v} \HH^2(\Q_v,\mu_2)$ has trivial kernel by the Hasse principle for the Brauer group, the kernel of $\HH^1(\Q,G) \rightarrow \prod_{v} \HH^1(\Q_v,G)$ is trivial too.
\end{proof}

\subsection{The representation \texorpdfstring{$(\rhoG,\rhoV)$}{(G*,V*)}}\label{subsection: the representation rhoV}

We define a representation $(\rhoG, \rhoV)$ and study its relation to $(\smallG,\smallV)$ using the binary quartic resolvent map from \S\ref{subsection: the resolvent binary quartic}.

\begin{definition}
	Define the $\Q$-group $\rhoG\coloneqq \PGL_2$. 
	Define the $\rhoG$-representation $\rhoV \coloneqq\Q\oplus \Q\oplus \Sym^4(2)$, where $\Q$ denotes a copy of the trivial representation and $\Sym^4(2)$ denotes the space of binary quartic forms $$\{q\mid q(x,y) = ax^4+bx^3y+cx^2y^2+dxy^3+ey^4\}.$$ 
	An element $[A]\in \PGL_2(\Q)$ acts on $q$ via $[A]\cdot q(x,y)=q((x,y)\cdot A)/(\det A)^2$.
	Define $\rhoB \coloneqq \rhoV \GIT \rhoG$. 
\end{definition}
We will typically write an element of $\rhoV$ as a triple $(b_2,b_6,q)$. 
We define a $\G_m$-action on $\rhoV$ by $\lambda\cdot (b_2,b_6,q) = (\lambda^2b_2,\lambda^6b_6,\lambda^4 q)$. 
Write $\bigresolv\colon \smallV\rightarrow \rhoV$ for the morphism $v\mapsto (p_2(v),p_6(v), Q_v)$, where $Q_v$ denotes the resolvent binary quartic from \S\ref{subsection: the resolvent binary quartic} and $p_2,p_6$ denote the invariant polynomials fixed in Proposition \ref{proposition: bridge F4 rep and F4 curves}.
The odd choice of $\G_m$-action on $\rhoV$ is explained by the fact it makes $\bigresolv$ equivariant with respect to the $\G_m$-actions on $\smallV$ and $\rhoV$.
Similarly to \S\ref{subsection: the resolvent binary quartic} write $p\colon \smallG\rightarrow \rhoG$ for the projection associated to the identification $\smallG \simeq (\Sp_6\times \SL_2)/\mu_2$ chosen in \S\ref{subsection: an explicit description of V}. 
There exists a unique $\G_m$-action on $\rhoB$ such that the quotient morphism $\rhopi \colon \rhoV \rightarrow \rhoB$ is $\G_m$-equivariant.

If $q(x,y) = ax^4+bx^3y+cx^2y^2+dxy^3+ey^4 $, we define 
\begin{align}
	I(q) &\coloneqq -3(12ae-3bd+c^2), \label{equation: invariant I} \\
	J(q) &\coloneqq (72ace+9bcd-27ad^2-27eb^2-2c^3). \label{equation: invariant J}
\end{align}
Then $I, J$ generate the ring of invariants of a binary quartic form. (Our $I(q)$ is $-3$ times the degree-$2$ invariant defined in \cite[\S2, Eq. (4)]{BS-2selmerellcurves}.) 
We obtain an isomorphism of graded $\Q$-algebras $\Q[\rhoB]\simeq \Q[b_2,b_6,I,J]$ where $b_2,b_6,I,J$ have degree $2,6,8,12$ respectively. 
Moreover a binary quartic form $q$ with coefficients in a field extension $k/\Q$ has distinct roots in $\P^1(\bar{k})$ if and only if $4I(q)^3+27J(q)^2 \neq 0$.

We describe centralizers of elements of $\rhoV$ in two ways. 
First we recall their classical relation to $2$-torsion of elliptic curves. 
If $k$ is a field and $I,J\in k$ write $E^{I,J}$ for the elliptic curve over $k$ given by the Weierstrass equation $y^2 = x^3+Ix+J$. 

\begin{lemma}\label{lemma: centralizer bin quartic classical ell curve}
	Let $k/\Q$ be a field and $v\in \rhoV(k)$ have invariants $(I,J) \coloneqq (I(v),J(v)) \in k^2$ such that $4I^3+27J^2 \neq 0$. 
	Then there is an isomorphism of finite \'etale group schemes over $k$:
	\begin{align*}
		Z_{\rhoG}(v) \simeq E^{I,J}[2].
	\end{align*}
\end{lemma}
\begin{proof}
	Up to scaling the invariants and changing an elliptic curve by a quadratic twist which doesn't affect the $2$-torsion group scheme, this is contained in \cite[Theorem 3.2]{BS-2selmerellcurves}.
\end{proof}

Next we give an alternative interpretation of centralizers in $\rhoV$ using the results of \S\ref{section: geometry}.
Recall from Corollary \ref{corollary: subgroups Prym[2] using bigonal} that we have an exact sequence of finite \'etale group schemes over $\smallB^{\rs}$:
$$
0 \rightarrow \ellcurve[2] \rightarrow \Prym[2] \rightarrow \hat{\ellcurve}[2]\rightarrow 0.
$$

\begin{lemma}\label{lemma: centralizer resolv bin quartic bigonal ell curve}
	The following two morphisms are canonically identified:
	\begin{itemize}
		\item The morphism $p\colon Z_{\smallG}(\sigma|_{\smallB^{\rs}}) \rightarrow Z_{\rhoG}(\bigresolv\circ \sigma|_{\smallB^{\rs}})$.
		\item The morphism $\Prym[2] \rightarrow \hat{\ellcurve}[2]$.
	\end{itemize}
	In particular for every field $k/\Q$ and $b\in \smallB^{\rs}(k)$, we have an isomorphism of $k$-group schemes $$Z_{\rhoG}(\bigresolv(\smallsigma(b))) \simeq E^{p_8(b),p_{12}(b)}[2].$$ 
\end{lemma}
\begin{proof}
	The last sentence follows from the first claim and the fact that $\hat{\ellcurve}_b$ and $E^{p_8(b),p_{12}(b)}$ are quadratic twists so have isomorphic $2$-torsion group scheme.  
	To prove the first claim it suffices to prove that the map $Z_{\smallG}(\sigma|_{\smallB^{\rs}}) \rightarrow Z_{\rhoG}(\bigresolv\circ \sigma|_{\smallB^{\rs}})$ is a nonconstant morphism of finite \'etale group schemes and $Z_{\rhoG}(\bigresolv\circ \sigma|_{\smallB^{\rs}})$ has order $4$; its kernel must then correspond, under the isomorphism $Z_{\smallG}(\sigma|_{\smallB^{\rs}})\simeq \Prym[2]$ of Proposition \ref{proposition: iso centralizer in G and prym[2]}, to the unique finite \'etale subgroup scheme of $\Prym[2]$ of order $4$ by Corollary \ref{corollary: subgroups Jac[2]}. 
	Lemma \ref{lemma: stable implies almost stable} implies that $Z_{\rhoG}(\bigresolv\circ \sigma|_{\smallB^{\rs}})$ is finite \'etale and Lemma \ref{lemma: centralizer bin quartic classical ell curve} implies that it is of order $4$. 
	Assume for contradiction that $p$ is constant. 
	Then by Lemma \ref{lemma: AIT full generality} we obtain a commutative diagram for every field $k/\Q$ and $b\in \smallB^{\rs}(k)$:
	\begin{center}
	\begin{tikzcd}
			\smallG(k) \backslash \smallV_b(k) \arrow[d] \arrow[r,"\bigresolv"] & \rhoG(k) \backslash \rhoV_{\bigresolv(b)}(k) \arrow[d] \\
			\HH^1(k,Z_{\smallG}(\sigma(b))) \arrow[r] & \HH^1(k,Z_{\rhoG}(\bigresolv(\sigma(b))) ) 
	\end{tikzcd}
\end{center}
	where the bottom map is constant.
	This implies that for every field $k/\Q$ and for every two $v_1, v_2\in \smallV_b(k)$, the binary quartic forms $Q_{v_1}$ and $Q_{v_2}$ are $\PGL_2(k)$-equivalent. 
	In particular by taking $v_1 = \sigma(b)$, Lemma \ref{lemma: reducible implies almost reducible} shows that $Q_v$ has a $k$-rational linear factor for every $v\in \smallV^{\rs}(k)$. 
	It is now simple to exhibit an explicit $v\in \smallV^{\rs}(k)$ for which this fails; an example with $k=\Real$ is given in Remark \ref{remark: explicit example element which is not almost reducible}.
\end{proof}

The morphism $\bigresolv \colon \smallV\rightarrow \rhoV$ induces a morphism $\smallB\rightarrow \rhoB$, still denoted by $\bigresolv$.
We write $\bigresolv_2, \bigresolv_6 , \bigresolv_I , \bigresolv_J \in \Q[\smallB]$ for the components of $\bigresolv$ using the coordinates $b_2,b_6, I,J$. 
Evidently, we have $\bigresolv_2 = p_2$ and $\bigresolv_6 =p_6$.
The next lemma determines $\bigresolv_I$ and $\bigresolv_J$ up to a constant.

\begin{proposition}\label{proposition: comparison invariants ell curve V}
There exists $\lambda\in \Q^{\times}$ such that 
\begin{align*}
	\bigresolv_I = \lambda^2 p_8, \quad \bigresolv_J = \lambda^3 p_{12}.
\end{align*}
\end{proposition}
\begin{proof}
	Since $\bigresolv$ is $\G_m$-equivariant, the elements $\bigresolv_I$ and $\bigresolv_J$ of $\Q[\smallB]$ are homogeneous of degree $8, 12$ respectively. 
	Lemma \ref{lemma: stable implies almost stable} implies that $\bigresolv$ maps $\smallB^{\rs}$ in the locus of $\rhoB$ where $4\bigresolv_I^3+27\bigresolv_J^2$ does not vanish. 
	In other words, we have a divisibility of polynomials in $\Q[\smallB]$:
	\begin{align}\label{equation: proof divisibility invariants}
	4\bigresolv_I(b)^3+27\bigresolv_J(b)^2 \mid (4p_8(\hat{b})^3+27p_{12}(\hat{b})^2)(4p_8(b)^3+27p_{12}(b)^2).
	\end{align}
	Here we have replaced $\Delta\in \Q[\smallB]$ by its explicit description afforded by Lemma \ref{lemma: discriminant polynomial F4}.
	By degree considerations and the fact that the right hand side of (\ref{equation: proof divisibility invariants}) is a product of two irreducible polynomials, we know that up to a nonzero constant $4\bigresolv_I(b)^3+27\bigresolv_J(b)^2$ equals either $4p_8(b)^3+27p_{12}(b)^2$ or $4p_8(\hat{b})^3+27p_{12}(\hat{b})^2$.
	In the first case, an explicit computation (using that $\bigresolv_I$ is a $\Q$-linear combination of elements of the form $p_8, p_2p_6, p_2^3$ and analogously for $\bigresolv_J$) one see that we must have $(\bigresolv_I(b), \bigresolv_J(b)) = (\lambda^2 p_8(b), \lambda^3 p_{12}(b))$ for some $\lambda\in \Q^{\times}$. 
	In the second case, we must have $(\bigresolv_I(b), \bigresolv_J(b)) = (\lambda^2 p_8(\hat{b}), \lambda^3 p_{12}(\hat{b}))$ for some $\lambda\in \Q^{\times}$ since $b\mapsto \hat{b}$ is an isomorphism (Theorem \ref{theorem: summary bigonal construction}).
	
	We argue by contradiction to exclude the second case, so suppose that it holds.
	Let $k/\Q$ be an algebraically closed field extension and $\mu\in k^{\times}$ a fourth root of $\lambda$. 
	Then $(\bigresolv_I(b), \bigresolv_J(b)) = (p_8(\mu\cdot \hat{b}), p_{12}(\mu \cdot \hat{b}))$. 
	Let $\eta\colon \Spec k(\eta) \rightarrow \smallB_k$ be the generic point of $\smallB_k$ and for ease of notation write $v^{\star} = \bigresolv(\sigma(\eta))$. 
	By Lemma \ref{lemma: centralizer bin quartic classical ell curve}  we have an isomorphism $Z_{\rhoG}(v^{\star}) \simeq E^{\bigresolv_I(\eta),\bigresolv_J(\eta)}[2]$. 
	On the other hand by Lemma \ref{lemma: centralizer resolv bin quartic bigonal ell curve} we have $Z_{\rhoG}(v^{\star}) \simeq E^{p_8(\eta),p_{12}(\eta)}[2]$.
	Using the assumption $(\bigresolv_I(b), \bigresolv_J(b)) = (p_8(\mu\cdot \hat{b}),p_{12}(\mu\cdot \hat{b}))$ and the fact that $\hat{E}_{b}[2] \simeq E^{p_8(b),p_{12}(b)}[2]$, we obtain a chain of isomorphisms
	\begin{align*}
		\hat{E}_{\eta}[2] \simeq E^{p_8(\eta),p_{12}(\eta)}[2] \simeq Z_{\rhoG}(v^{\star}) \simeq  E^{\bigresolv_I(\eta),\bigresolv_J(\eta)}[2] = E^{p_8(\mu\cdot \hat{\eta}),p_{12}(\mu\cdot \hat{\eta})}[2] \simeq E_{\eta}[2].
	\end{align*}
	But by Corollary \ref{corollary: subgroups Prym[2] using bigonal}, $E[2]$ and $\hat{E}[2]$ are not isomorphic as finite \'etale group schemes over $\smallB^{\rs}$.
	By \cite[Tag \href{https://stacks.math.columbia.edu/tag/0BQM}{0BQM}]{stacksproject} and the fact that $\smallB$ is normal, the $k(\eta)$-groups $E_{\eta}[2]$ and $\hat{E}_{\eta}[2]$ are not isomorphic either. 
	This is a contradiction, proving the proposition. 
\end{proof}

\begin{corollary}\label{corollary: resolv iso on GIT}
The map $\bigresolv \colon \smallB\rightarrow \rhoB$ is a $\G_m$-equivariant isomorphism. 
\end{corollary}
\begin{proof}
	In the coordinates $\smallB \simeq \A^4_{(p_2,p_6,p_8,p_{12})}$ and $\rhoB \simeq \A^4_{(b_2,b_6,I,J)}$, $\bigresolv$ takes the form $(p_2,p_6,p_8,p_{12}) \mapsto (p_2,p_6,\lambda^2 p_8, \lambda^3 p_{12})$ for some $\lambda\in \Q^{\times}$ by Proposition \ref{proposition: comparison invariants ell curve V}.
\end{proof}

\subsection{Embedding the \texorpdfstring{$\rhodual$}{rhohat}-Selmer group}\label{subsection: embedding the rho selmer group}

The following proposition follows from Lemmas \ref{lemma: AIT full generality} and \ref{lemma: centralizer resolv bin quartic bigonal ell curve} by the same proof as Lemma \ref{lemma: AIT}.
If $b^{\star}\colon S\rightarrow \rhoB$ is an $S$-valued point we write $\rhoV_{b^{\star}}$ for the fibre of $\rhopi \colon \rhoV \rightarrow \rhoB$ under $b^{\star}$. 

\begin{proposition}\label{proposition: AIT rho selmer}
	Let $R$ be a $\Q$-algebra. Let $b\in \smallB^{\rs}(R)$ with $b^{\star} \coloneqq \bigresolv(b)$. 
	Then there are canonical bijections of sets:
	\begin{enumerate}
		\item $ \smallG(R) \backslash \smallV_b(R) \simeq  \ker\left(\HH^1(R,\Prym_b[2]) \rightarrow \HH^1(R,\smallG)\right).$
		\item $ \rhoG(R) \backslash \rhoV_{b^{\star}}(R) \simeq  \ker\left(\HH^1(R,\hat{\ellcurve}_b[2]) \rightarrow \HH^1(R,\rhoG)\right).$
	\end{enumerate}
	The reducible orbits $\smallG(R)\cdot \sigma(b)$ and $\rhoG(R)\cdot \bigresolv(\sigma(b))$ correspond to the trivial element in $\HH^1(R,\Prym_b[2])$ and $\HH^1(R,\hat{\ellcurve}_b[2])$ respectively. 
	Moreover the following diagram is commutative.
	\begin{center}
	\begin{tikzcd}
			\smallG(R) \backslash \smallV_b(R) \arrow[d] \arrow[r] & \rhoG(R) \backslash \rhoV_{b^{\star}}(R) \arrow[d] \\
			\HH^1(R,\Prym_b[2]) \arrow[r] & \HH^1(R,\hat{\ellcurve}_b[2]) 
	\end{tikzcd}
\end{center}
	Here the horizontal maps are induced by $\bigresolv \colon \smallV\rightarrow \rhoV$ and the projection $\Prym_b[2] \rightarrow \hat{\ellcurve}_b[2]$ respectively and the vertical maps are the injections induced by the above identifications.
\end{proposition}

The following corollary will be useful later and connects the notion of almost reducibility to a more arithmetic one.
It follows from the commutative diagram of Proposition \ref{proposition: AIT rho selmer}.

\begin{corollary}\label{corollary: almost reducible equivalences}
	Let $k/\Q$ be a field and $b \in \smallB^{\rs}(k)$. Then the following are equivalent for $v \in \smallV_b(k)$:
	\begin{itemize}
		\item $v$ is almost $k$-reducible (Definition \ref{definition: almost $k$-reducible}).
		\item The class of $\smallG(k)\cdot v$ in $\HH^1(k,\Prym_b[2])$ under the bijection of Proposition \ref{proposition: AIT rho selmer} lies in the kernel of the map $\HH^1(k,\Prym_b[2]) \rightarrow \HH^1(k,\hat{\ellcurve}_b[2])$.
	\end{itemize}
\end{corollary}

\begin{theorem}\label{theorem: inject rho-descent orbits}
	Let $R$ be a $\Q$-algebra such that every locally free $R$-module of constant rank is free and let $b\in \smallB^{\rs}(R)$ with $b^{\star} \coloneqq \bigresolv(b)$. 
	Then there exists a natural embedding $\eta^{\star}_b\colon \Prym_b(R)/\rhodual(\Prym_b^{\vee}(R)) \hookrightarrow \rhoG(R)\backslash \rhoV_{b^{\star}}(R) $ compatible with base change on $R$. Moreover it sends the identity element to the orbit $\rhoG(R)\cdot \bigresolv(\smallsigma(b))$. 
\end{theorem}
\begin{proof}
	Recall from Corollary \ref{corollary: subgroups Prym[2] using bigonal} that we have an isomorphism $\Prym_b^{\vee}[\rhodual] \simeq \hat{\ellcurve}_b[2]$. For $A \in \Prym_b(R)$, write $\eta^{\star}_b(A)\in \HH^1(R,\hat{\ellcurve}_b[2])$ for the image of $A$ under the $\rhodual$-descent map transported along the isomorphism $\HH^1(R,\Prym_b^{\vee}[\rhodual]) \simeq \HH^1(R,\hat{\ellcurve}_b[2])$.
	Using the identification of Proposition \ref{proposition: AIT rho selmer} it suffices to prove that $\eta^{\star}_b(A)$ is killed under the map $\HH^1(R,\hat{\ellcurve}_b[2]) \rightarrow \HH^1(R,\rhoG)$. 
The commutative diagram
\begin{center}
	\begin{tikzcd}
		\Prym_b(R)/2\Prym_b(R) \arrow[d] \arrow[r, two heads] & \Prym_b(R)/\rhodual(\Prym_b^{\vee}(R)) \arrow[d] \\
		\HH^1(R,\Prym_b[2]) \arrow[r]                      & \HH^1(R,\Prym_b[\rhodual])                          
	\end{tikzcd}
\end{center}
shows that $\eta^{\star}_b(A)$ lifts to a class in $\HH^1(R,\Prym_b[2])$ lying in the image of the $2$-descent map.  By the proof of Theorem \ref{theorem: inject 2-descent into orbits}, the image of this class in $\HH^1(R,\smallG)$ is trivial. 
Therefore the image of $\eta_b^{\star}(A)$ in $\HH^1(R,\rhoG)$ is trivial too. 
\end{proof}

We obtain the following consequence for the $\rhodual$-Selmer group, whose proof is identical to the proof of Corollary \ref{corollary: Sel2 embeds} and uses the fact that $\HH^1(\Q,\rhoG ) \rightarrow \prod_v \HH^1(\Q_v,\rhoG)$ has trivial kernel.

\begin{corollary}\label{corollary: Selrho embeds}
	Let $b\in \smallB^{\rs}(\Q)$ with $b^{\star} \coloneqq \bigresolv(b)$ and write $\Sel_{\rhodual} \Prym_b^{\vee}$ for the $\rhodual$-Selmer group of $\Prym_b^{\vee}$. Then the injection $\eta^{\star}_b\colon \Prym_b(\Q)/\rhodual(\Prym^{\vee}_b(\Q)) \hookrightarrow \rhoG(\Q)\backslash \rhoV_{b^{\star}}(\Q)$ of Theorem \ref{theorem: inject rho-descent orbits} extends to an injection 
	$$
	\Sel_{\rhodual}\Prym^{\vee}_{b} \hookrightarrow \rhoG(\Q)\backslash \rhoV_{b^{\star} }(\Q).
	$$
\end{corollary}

\section{Integral orbit representatives}\label{section: integral orbit representatives}

\subsection{Integral structures}\label{subsection: integral structures}

So far we have considered properties of the pair $(\smallG,\smallV)$ and $(\rhoG,\rhoV)$ over $\Q$. In this subsection we define these objects over $\Z$ and observe that the above results and constructions are still valid over $\Z[1/N]$ for an appropriate choice of integer $N\geq 1$. 

Indeed, our choice of pinning of $\smallH$ in \S\ref{subsection: definition of the representations} determines a Chevalley basis of $\smallh$, hence a $\Z$-form $\intsmallh$ of $\smallh$ (in the sense of \cite[\S 1]{Borel-propertieschevalley}) with adjoint group $\intsmallH$, a split reductive group of type $F_4$ over $\Z$.
The $\Z$-lattice $\intsmallV = \smallV\cap \intsmallh$ is admissible \cite[Definition 2.2]{Borel-propertieschevalley}; define $\intsmallG$ as the Zariski closure of $\smallG$ in $\GL(\intsmallV)$. 
The $\Z$-group scheme $\intsmallG$ has generic fibre $\smallG$ and acts faithfully on the free $\Z$-module $\intsmallV$ of rank $28$. 
The automorphism $\smalltheta\colon \smallH \rightarrow \smallH$ extends by the same formula to an automorphism $\intsmallH\rightarrow \intsmallH$, still denoted by $\smalltheta$. 
We may similarly define $\intbigH, \intbigG$ and $\intbigV$ and extend $\bigtheta$, $\zeta$ to involutions $\intbigH\rightarrow \intbigH$. 


\begin{lemma}
\begin{enumerate}
    \item $\intsmallG$ is a split reductive group over $\Z$ of type $C_3\times A_1$.
    \item The equality $\smallH^{\smalltheta}=\smallG$ extends to an isomorphism $\intsmallH^{\smalltheta}_{\Z[1/2]}\simeq \intsmallG_{\Z[1/2]}$.
    \item The equality $\bigH^{\zeta}=\smallH$ extends to an isomorphism $\intsmallH^{\zeta}_{\mathrm{E},\Z[1/2]}\simeq \intsmallH_{\Z[1/2]}$.
\end{enumerate}
\end{lemma}
\begin{proof}
    For the first claim, it suffices to prove that $\intsmallG \rightarrow \Spec \Z$ is smooth and affine and that its geometric fibres are connected reductive groups. 
    But $\intsmallG$ is $\Z$-flat and affine by construction, and its fibres are reductive by \cite[\S4.3]{Borel-propertieschevalley}. 
   The second claim follows from the fact that $\intsmallH^{\smalltheta}_{\Z[1/2]}$ is a reductive group scheme of the same type as $\intsmallG_{\Z[1/2]}$, which follows from \cite[Remark 3.1.5]{Conrad-reductivegroupschemes}. 
    The third claim follows from the fact that $\intsmallH^{\zeta}_{\mathrm{E},\Z[1/2]}$ is $\Z[1/2]$-smooth by Lemma \ref{lemma: fixed points smooth morphism is smooth}, and that its geometric fibres are adjoint semisimple of type $F_4$ (by the same reasoning as \cite[\S3.1]{Reeder-torsion}). 
\end{proof}

We define the smooth $\Z$-group $\intrhoG \coloneqq \PGL_2$ and $\intrhoG$-representation $\intrhoV \coloneqq \Z\oplus \Z\oplus \Sym^4(2)$, where $\Sym^4(2)$ denotes the space of binary quartic forms $ax^4+bx^3y+cx^2y^2+dxy^3+ey^4$ with $a,\dots,e\in \Z$. 
The $\Z$-module $\intrhoV$ is free of rank $7$.

Recall that in \S\ref{subsection: a family of curves} we have fixed polynomials $p_2,p_6,p_8,p_{12}\in \Q[\smallV]^{\smallG}$ satisfying the conclusions of Proposition \ref{proposition: bridge F4 rep and F4 curves}. 
Note that those conclusions are invariant under the $\G_m$-action on $\smallB$. 
By rescaling the polynomials $p_2,\dots,p_{12}$ using this $\G_m$-action, we can assume they lie in $\Z[\intsmallV]^{\intsmallG}$. 
We may additionally assume that the discriminant $\Delta$ from \S\ref{subsection : discriminant polynomial} lies in $\Z[\intsmallV]^{\intsmallG}$.
Define $\intsmallB \coloneqq \Spec \Z[p_2,p_6,p_8,p_{12}]$ and $\intsmallB^{\rs} \coloneqq \Spec \Z[p_2,p_6,p_8,p_{12}][\Delta^{-1}]$. 
We have an invariant map $\pi \colon \intsmallV \rightarrow \intsmallB $.


Recall from \S\ref{subsection: the representation rhoV} that we have defined a morphism $\bigresolv \colon \smallV\rightarrow \rhoV$ using the binary quartic resolvent from \S\ref{subsection: the resolvent binary quartic}, which extends by the same formula to a morphism $\bigresolv \colon \intsmallV \rightarrow \intrhoV$ (This follows from Formula (\ref{equation: invariant W}) and our choice of isomorphism $\smallV \simeq \smallW \boxtimes (2)$ made at the end of \S\ref{subsection: an explicit description of V}).
Define $\intrhoB = \Spec \Z[b_2,b_6,I,J]$ and write $\rhopi \colon \intrhoV\rightarrow \intrhoB$ for the invariant map.

Extend the morphism $\chi$ from \S\ref{subsection: the bigonal construction} to the morphism $\chi\colon \intsmallB\rightarrow \intsmallB$ given by the same Formula (\ref{equation: formula bigonal construction}).
Following \S\ref{subsection : discriminant polynomial} we define $\Delta_{\hat{\ellcurve}} \coloneqq 4p_8^3+27p_{12}^2$ and $\Delta_{\ellcurve} \coloneqq \Delta_{\ellcurve} \circ \chi$, both elements of $\Z[\intsmallB]$.

We extend the family of curves given by Equation (\ref{equation : F4 family middle of paper}) to the family $\intsmallcurve\rightarrow \intsmallB$ given by that same equation. Similarly we define $\intCellcurve\rightarrow \intsmallB$ by the family of curves given by Equation (\ref{equation: elliptic curve}).
They are defined by the projective closures of the equations 
\begin{align}
\intsmallcurve\colon y^4+p_2xy^2+p_6y^2 = x^3+p_8x+p_{12}, \label{equation: F4 family integral reps} \\
\intCellcurve\colon y^2+p_2xy+p_6y = x^3+p_8x+p_{12}. \label{equation: ellcurve family integral reps}
\end{align}
As before if $\mathcal{X}$ is a $\intsmallB$-scheme we write $\hat{\mathcal{X}}$ for the pullback of $\mathcal{X}$ along $\chi \colon \intsmallB \rightarrow \intsmallB$.

We can find an integer $N$ with the following properties (set $S = \Z[1/N]$):
\begin{enumerate}
	\item \label{enum: int 1} The integer $N$ is good in the sense of \cite[Proposition 4.1]{Laga-E6paper}. In particular, $2,3$ and $5$ are invertible in $S$ and $\intsmallcurve_S \rightarrow \intsmallB_S$ is flat and proper with geometrically integral fibres and smooth exactly above $\intsmallB_S^{\rs}$. 
	\item \label{enum: int 2} The morphism $\bigresolv \colon \smallB \rightarrow \rhoB$ of \S \ref{subsection: the representation rhoV} extends to an isomorphism $\bigresolv \colon \intsmallB_S\rightarrow \intrhoB_S$, and there exists $\lambda\in S^{\times}$ such that $(\bigresolv_I,\bigresolv_J) = (\lambda^2p_8,\lambda^{3}p_{12})$. (Proposition \ref{proposition: comparison invariants ell curve V}.) 
	\item \label{enum: int 3} The discriminant locus $\{ \Delta =0 \}_S \rightarrow \intsmallB_S$ has geometrically reduced fibres. Moreover $\Delta$ agrees with $\Delta_{\ellcurve}\Delta_{\hat{\ellcurve}}$ up to a unit in $\Z[1/N]$. (Proposition \ref{lemma: discriminant polynomial F4}.)
	\item  \label{enum: int 4} There exists open subschemes $\intsmallV^{\rs} \subset \intsmallV^{\reg} \subset \intsmallV_S$ such that if $S\rightarrow k$ is a map to a field and $v\in \intsmallV(k)$ then $v$ is regular if and only if $v \in \intsmallV^{\reg}(k)$ and $v$ is regular semisimple if and only if $v\in \intsmallV^{\rs}(k)$. 
	Moreover, $\intsmallV^{\rs}$ is the open subscheme defined by the nonvanishing of the discriminant polynomial $\Delta\in S[\intsmallV]$.
	\item  \label{enum: int 5} $S[\intsmallV]^{\intsmallG} = S[p_2,p_6,p_8,p_{12}]$. The Kostant section of \S\ref{subsection: distinguished orbit} extends to a section $\sigma\colon \intsmallB_S \rightarrow \intsmallV^{\reg}$ of $\smallpi$ satisfying the following property: for any $b\in \intsmallB(\Z)\subset \intsmallB_S(S)$ we have $\smallsigma(N\cdot b) \in \intsmallV(\Z)$. 
	\item  \label{enum: int 6} Define $\intJac \rightarrow \intsmallB_S^{\rs}$ to be the Jacobian of the family of smooth curves $\intsmallcurve^{\rs}_S \rightarrow \intsmallB_S^{\rs}$ \cite[\S9.3, Theorem 1]{BLR-NeronModels}.
	Let $\intellcurve\rightarrow \intsmallB_S^{\rs}$ denote the restriction of $\intCellcurve_S$ to $\intsmallB_S^{\rs}$. Let $\intPrym\rightarrow \intsmallB_S$ be the Prym variety of the cover $\intsmallcurve^{\rs}_S\rightarrow \intellcurve$ as defined in \S\ref{subsection: def prym variety}. Then the isomorphism from Proposition \ref{proposition: iso centralizer in G and prym[2]} extends to an isomorphism $\intPrym[2] \simeq Z_{\intsmallG_S}(\smallsigma|_{\intsmallB_S^{\rs}})) $ of finite \'etale group schemes over $\intsmallB_S^{\rs}$.
	\item  \label{enum: int 7} The action map $\intsmallG_S \times \intsmallB_S \rightarrow \intsmallV^{\reg},\, (g,b) \mapsto g\cdot \sigma(b) $ is \'etale and its image contains $\intsmallV^{\rs}$. (Proposition \ref{proposition: kostant section}.)
	\item \label{enum: int 8} The $\smallB$-scheme $\CPrym$ constructed in \S\ref{subsection: compactifications} extends to a $\intsmallB_S$-scheme $\intCPrym\rightarrow \intsmallB_S$ which is flat, projective, with geometrically integral fibres and whose restriction to $\intsmallB_S^{\rs}$ is isomorphic to $\intPrym$. 
	Moreover, $\intCPrym\rightarrow S$ is smooth with geometrically integral fibres, and the smooth locus of the morphism $\intCPrym\rightarrow \intsmallB_S$ is an open subscheme of $\intCPrym$ whose complement is $S$-fibrewise of codimension at least two. (Proposition \ref{proposition: good compactifications exist}.)
	\item \label{enum: int 9} Let $\Pic_{\intsmallcurve_S/\intsmallB_S}^0$ denote the identity component of the relative Picard scheme of $\intsmallcurve_S\rightarrow \intsmallB_S$. 
	Then the fibres of $\Pic_{\intsmallcurve_S/\intsmallB_S}^0[1+\curvezeta^*]\rightarrow \intsmallB_S$ are geometrically integral. (Proposition \ref{proposition: generalized prym variety has integral fibres}.)
	\item \label{enum: int 10} For every field $k$ of characteristic not dividing $N$ and $b\in \intsmallB^{\rs}(k)$, there exists an isomorphism $\hat{\intPrym}_b\simeq \intPrym^{\vee}_b$ of $(1,2)$-polarized abelian varieties. (Theorem \ref{theorem: summary bigonal construction}.)
\end{enumerate}

The existence of such an $N$ follows from the principle of spreading out. (See \cite[Proposition 4.1]{Laga-E6paper} for more details.)
We fix such an integer for the remainder of the paper. 

Using these properties, we can extend our previous results to $S$-algebras rather than $\Q$-algebras. 
We mention in particular:

\begin{proposition}[Analogue of Lemma \ref{lemma: AIT} and Proposition \ref{proposition: AIT rho selmer}]\label{proposition: spread out orbit parametrization galois}
Let $R$ be an $S$-algebra and $b\in \intsmallB^{\rs}(R)$ with $b^{\star} \coloneqq \bigresolv(b)$. Then we have natural bijections of pointed sets:
	\begin{enumerate}
		\item $ \intsmallG(R) \backslash \intsmallV_b(R) \simeq  \ker\left(\HH^1(R,\intPrym_b[2]) \rightarrow \HH^1(R,\intsmallG)\right).$
		\item $\intbigG(R)\backslash \intsmallV_{\mathrm{E},b}(R) \simeq \ker\left(\HH^1(R,\intJac_b[2]) \rightarrow \HH^1(R,\intbigG)\right).$
		\item $ \intrhoG(R) \backslash \intrhoV_{b^{\star}}(R) \simeq  \ker\left(\HH^1(R,\hat{\intellcurve}_b[2]) \rightarrow \HH^1(R,\intrhoG)\right).$
	\end{enumerate}
	
\end{proposition}

\begin{proposition}[Analogue of Theorems \ref{theorem: inject 2-descent into orbits} and \ref{theorem: inject rho-descent orbits}]\label{proposition: spread out selmer group embedding}
Let $R$ be an $S$-algebra and $b\in \intsmallB^{\rs}(R)$ with $b^{\star}\coloneqq \bigresolv(b)$. Suppose that every locally free $R$-module of constant rank is free. Then there is a commutative diagram
	\begin{center}
	\begin{tikzcd}
	\intPrym_b(R)/2\intPrym_b(R) \arrow[r,"\eta_b"] \arrow[d] & \intsmallG(R)\backslash \intsmallV_b(R) \arrow[d,"\bigresolv"] \\
	\intPrym_b(R)/\rhodual(\intPrym^{\vee}_b(R)) \arrow[r,"\eta^{\star}_b"] & \intrhoG(R)\backslash \intrhoV_{b^{\star}}(R)
	\end{tikzcd}	
	\end{center}
	Here the horizontal arrows $\eta_b$ and $\eta^{\star}_b$ are injections and send the identity to the orbit of $\sigma(b)$ and $\bigresolv(\sigma(b))$ respectively. 	
\end{proposition}

The main aim of \S\ref{section: integral orbit representatives} is to prove the following two theorems concerning integral orbit representatives. Both have consequences for orbits over $\Z$, see Corollaries \ref{corollary: weak global integral representatives 2 case} and \ref{corollary: weak global integral representatives rho case}.

\begin{theorem}\label{theorem: F4 integral representatives exist}
	Let $p$ be a prime not dividing $N$. Let $b\in \intsmallB(\Z_p)$ with $\Delta(b)\neq 0$. 
	Then every orbit in the image of the map 
	$$\eta_b \colon\Prym_b(\Q_p)/2\Prym_b(\Q_p) \rightarrow \smallG(\Q_p) \backslash \smallV_b(\Q_p)$$
	of Theorem \ref{theorem: inject 2-descent into orbits} has a representative in $\intsmallV(\Z_p)$. 
\end{theorem}

\begin{theorem}\label{theorem: rho integral representatives exist}
	Let $p$ be a prime not dividing $N$. Let $b\in \intsmallB(\Z_p)$ with $\Delta(b)\neq 0$ and write $b^{\star} \coloneqq \bigresolv(b)$.  
	Then every orbit in the image of the map 
	$$\eta^{\star}_b\colon \Prym_b(\Q_p)/\rhodual(\Prym^{\vee}_b(\Q_p)) \rightarrow \rhoG(\Q_p) \backslash \rhoV_{b^{\star}}(\Q_p)$$
	of Theorem \ref{theorem: inject rho-descent orbits}  has a representative in $\intrhoV(\Z_p)$. 
\end{theorem}

Theorem \ref{theorem: rho integral representatives exist} will follow easily from Theorem \ref{theorem: F4 integral representatives exist}, so we spend most of \S\ref{section: integral orbit representatives} proving Theorem \ref{theorem: F4 integral representatives exist}.
We follow the general strategy of \cite[\S4]{Laga-E6paper}, the main difference being that the role of the compactified Jacobian is played here by the compactified Prym variety introduced in \S\ref{subsection: compactifications}.

\subsection{Some groupoids}\label{subsection: some groupoids}

In this section we define some groupoids which will be a convenient way to think about orbits in our representations and a crucial ingredient for the proof of Theorem \ref{theorem: F4 integral representatives exist}.
It is closely modelled on the corresponding section \cite[\S4.3]{Thorne-Romano-E8}; the reader may also consult \cite[\S4.2]{Laga-E6paper}.
Throughout this section we fix a scheme $X$ over $S = \Z[1/N]$.

We define the groupoid $\GrLie_X$ whose objects are pairs $(H',\theta')$ where 
\begin{itemize}
	\item $H'$ is a reductive group scheme over $X$ whose geometric fibres are simple of Dynkin type $F_4$. (See \cite[Definition 3.1.1]{Conrad-reductivegroupschemes} for the definition of a reductive group scheme over a general base.)
	\item $\theta': H' \rightarrow H'$ is an involution of reductive $X$-group schemes such that for each geometric point $\bar{x}$ of $X$ there exists a maximal torus $A_{\bar{x}}$ of $H'_{\bar{x}}$ such that $\theta'$ acts as $-1$ on $X^*(A_{\bar{x}})$.  
\end{itemize}
A morphism $(H',\theta') \rightarrow (H'',\theta'')$ in $\GrLie_X$ is given by an isomorphism $\phi: H'\rightarrow H''$ such that $\phi \circ \theta' = \theta'' \circ \phi$. 
There is a natural notion of base change and the groupoids $\GrLie_X$ form a stack over the category of schemes over $S$ in the \'etale topology. 
Recall that in \S\ref{subsection: integral structures} we have defined a pair $(\intsmallH_S,\smalltheta_S)$ which by \cite[Corollary 14]{GrossLevyReederYu-GradingsPosRank} defines an object of $\GrLie_{S}$.

\begin{proposition}\label{proposition: G-torsors in terms of groupoids}
	Let $X$ be an $S$-scheme. 
	The assignment $(H',\theta')\mapsto \Isom((\intsmallH_X,\smalltheta_X),(H',\theta'))$ defines a bijection between: 
	\begin{itemize}
		\item The isomorphism classes of objects in $\GrLie_X$. 
		\item The set $\HH^1(X,\intsmallG)$.
	\end{itemize}
\end{proposition}
\begin{proof}
Since $\GrLie$ is a stack in the \'etale topology of $S$-schemes and $\Aut((\intsmallH_X,\smalltheta_X))=\smallG_X$, it suffices to prove that any two objects $(H,\theta), (H',\theta')$ of $\GrLie_X$ are \'etale locally isomorphic.
The proof of this fact is very similar to the proof of \cite[Lemma 2.3]{Thorne-Romano-E8} and we omit it. (See also \cite[Proposition 4.4]{Laga-E6paper}.)
%
%
\end{proof}

We define the groupoid $\GrLieE_X$ whose objects are triples $(H',\theta',\gamma')$ where $(H',\theta')$ is an object of $\GrLie_X$ and $\gamma'\in \lieh'$ (the Lie algebra of $H'$) satisfying $\theta'(\gamma') = -\gamma'$. 
A morphism $(H',\theta',\gamma')\rightarrow (H'',\theta'',\gamma'')$ in $\GrLieE_X$ is given by a morphism $\phi: H' \rightarrow H''$ in $\GrLie_X$ mapping $\gamma'$ to $\gamma''$. 

We define a functor $\GrLieE_X \rightarrow \intsmallB(X)$ (where $\intsmallB(X)$ is seen as a discrete category) as follows. For an object $(H',\theta',\gamma')$ in $\GrLieE_X$, choose a faithfully flat extension $X'\rightarrow X$ such that there exists an isomorphism $\phi: (H',\theta')_{X'} \rightarrow (\intsmallH_S,\smalltheta)_{X'}$ in $\GrLie_X$. We define the image of the object $(H',\theta',\gamma')$ under the map $\GrLieE_X \rightarrow \intsmallB(X)$ by $\smallpi(\phi(\gamma'))$. This procedure is independent of the choice of $\phi$ and $X'$ and by descent defines an element of $\intsmallB(X)$. 
For $b\in \intsmallB(X)$ we write $\GrLieE_{X,b}$ for the full subcategory of elements of $\GrLieE_{X,b}$ mapping to $b$ under this map.
In \S\ref{subsection: integral structures} we have defined an object $(\intsmallH_{\intsmallB_S}, \smalltheta_{\intsmallB_S}, \sigma)$ of $\GrLieE_{\intsmallB_S}$.

Recall that for $b\in \intsmallB(X)$, $\intsmallV_b$ denotes the fibre of $b$ of the map $\smallpi: \intsmallV \rightarrow \intsmallB$. 

\begin{proposition}\label{proposition: H1 of stabilizer and GrLieE}
	Let $X$ be an $S$-scheme and let $b \in \intsmallB^{\rs}(X)$. The assignment $$(H',\theta',\gamma')\mapsto \Isom((\intsmallH_X,\smalltheta_X,\sigma(b)),(H',\theta',\gamma')) $$ defines a bijection between 
	\begin{itemize}
		\item Isomorphism classes of objects in $\GrLieE_{X,b}$.
		\item  The set $\HH^1(X,Z_{\intsmallG}(\sigma(b)))$.
	\end{itemize}
\end{proposition}
\begin{proof}
	The object $(\intsmallH_X,\smalltheta_X,\sigma(b))$ of $\GrLieE_{X,b}$ has automorphism group $Z_{\intsmallG}(\sigma(b))$.
	By descent, it suffices to prove that every object $(H',\theta', \gamma')$ in $\GrLieE_{X,b}$ is \'etale locally isomorphic to $(\intsmallH_X,\smalltheta_X,\sigma(b))$. 
	By Proposition \ref{proposition: G-torsors in terms of groupoids}, we may assume that $(H',\theta') = (\intsmallH_X,\smalltheta_X)$. 
	By Property \ref{enum: int 7} of \S\ref{subsection: integral structures} (which is a spreading out of Proposition \ref{proposition: kostant section} over $S$), the action map $\intsmallG_S \times \intsmallB^{\rs}_S \rightarrow \intsmallV^{\rs}$ is \'etale and surjective. 
	Therefore it has sections \'etale locally, hence $\gamma'$ is \'etale locally $\intsmallG$-conjugate to $\sigma(b)$. 
\end{proof}

The following important proposition gives an interpretation of the (not necessarily regular semisimple) $\intsmallG(X)$-orbits of $\intsmallV(X)$ in terms of the groupoids $\GrLie_X$ and $\GrLieE_X$. 

\begin{proposition}\label{proposition: G-orbits in terms of groupoids}
	Let $X$ be an $S$-scheme and let $b\in \intsmallB(X)$. 
	The following sets are in canonical bijection:
	\begin{itemize}
		\item The set of $\intsmallG(X)$-orbits on $\intsmallV_b(X)$. 
		\item Isomorphism classes of objects $(H',\theta',\gamma')$ in $\GrLieE_{X,b}$ such that $(H',\theta') \simeq (\intsmallH_S,\smalltheta)_X$ in $\GrLie_X$. 
	\end{itemize}
\end{proposition}
\begin{proof}
	If $v\in \intsmallV_b(X)$ we define the object $\mathcal{A}_v = (\intsmallH_X,\smalltheta_X,v)$ of $\GrLieE_{X,b}$. 
	The assignment $v\mapsto \mathcal{A}_v$ establishes a well-defined bijection between the two sets of the proposition; we omit the formal verification. 
	(See \cite[Proposition 4.6]{Laga-E6paper} for the proof of a similar statement.)
\end{proof}

The following lemma is analogous to \cite[Lemma 5.6]{Thorne-Romano-E8} and will be useful in \S\ref{subsection: integral reps: the $2$-Selmer case} to extend orbits over a base of dimension $2$.

\begin{lemma}\label{lemma: extend GrLieE complement codim 2}
Let $X$ be an integral regular scheme of dimension $2$.
Let $U\subset X$ be an open subscheme whose complement has dimension $0$. 
If $b\in \intsmallB_S(X)$, then restriction induces an equivalence of categories $\GrLieE_{X,b}\rightarrow \GrLieE_{U,b|_U}$. 	 
\end{lemma}
\begin{proof}
	We will use the following fact \cite[Lemme 2.1(iii)]{ColliotTheleneSansuc-Fibresquadratiques} repeatedly: if $Y$ is an affine $X$-scheme of finite type, then restriction of sections $Y(X)\rightarrow Y(U)$ is bijective.
	To prove essential surjectivity, let $(H',\theta',\gamma')$ be an object of $\GrLieE_{U,b|_U}$.
	By \cite[Th\'eoreme 6.13]{ColliotTheleneSansuc-Fibresquadratiques} and Proposition \ref{proposition: G-torsors in terms of groupoids}, $(H',\theta')$ extends to an object $(H'',\theta'')$ of $\GrLie_{X}$.
	If $Y$ is the closed subscheme of $\lieh''$ of elements $\gamma$ satisfying $\theta''(\gamma)=-\gamma$ and $\gamma$ maps to $b$ in $\intsmallB(X)$, then $Y$ is affine and of finite type over $X$. 
	It follows by the fact above that $\gamma'$ lifts to an element $\gamma''\in \lieh''(X)$ and that $(H'',\theta'',\gamma'')$ defines an object of $\GrLieE_{X,b}$. 
	Since the scheme of isomorphisms $\Isom_{\GrLieE}(\mathcal{A},\mathcal{A}')$ is $X$-affine, fully faithfulness follows again from the above fact. 
\end{proof}

\subsection{N\'eron Component groups of Prym varieties}\label{subsection: neron component groups pryms}

In this subsection we perform some calculations with component groups of N\'eron models of Prym varieties. 
For the purposes of constructing integral representatives (Theorems \ref{theorem: F4 integral representatives exist} and \ref{theorem: rho integral representatives exist}), only Proposition \ref{proposition: special fibres neron models} will be needed.
However, the finer analysis succeeding it is necessary to obtain a lower bound in Theorem \ref{theorem: intro rho selmer}; of this analysis only Corollary \ref{corollary: squarefree implies unramified selmer condition rhovee descent} will be used later.

\textbf{Notation.} For the remainder of \S\ref{subsection: neron component groups pryms}, let $R$ be a discrete valuation ring with residue field $k$ and fraction field $K$. We suppose that $N$ is invertible in $R$.

Recall that we have defined in \S\ref{subsection: integral structures} abelian schemes $\intJac, \intPrym$ and $\intellcurve$ over $\intsmallB_S^{\rs}$. 
We will use a minor abuse of notation and for any $b\in \intsmallB(R)$ with $\Delta(b)\neq 0$, we write $\intJac_b$ (which is a priori only defined when $\Delta(b)\in R^{\times}$) for the $K$-scheme $\intJac_{b_K}$, and similarly for $\intPrym_b$ and $\intellcurve_b$. 
For such $b$ we write $\NeronJac_b,\NeronPrym_b, \NeronPrym^{\vee}_b, \Neronellcurve_b$ for the N\'eron models of $\intJac_b, \intPrym_b, \intPrym_b^{\vee}, \intellcurve_b$ respectively. 
The involution $\curvezeta_b^*$ of $\intJac_b$ uniquely extends to an involution of $\NeronJac_b$, again denoted by $\curvezeta^*_b$. 

\begin{lemma}\label{lemma: neron model prym has same definition}
Let $b\in \intsmallB(R)$ with $\Delta(b)\neq 0$. 
Then the equality $\intPrym_b = \ker(1+\curvezeta_b^* \colon \intJac_b \rightarrow \intJac_b)$ from (\ref{equation: def prym as kernel 1+tau}) extends to an isomorphism $\NeronPrym_b \simeq \ker(1+\curvezeta^*_b\colon \NeronJac_b \rightarrow \NeronJac_b)$.
\end{lemma}
\begin{proof} 
	It suffices to prove that $\ker(1+\curvezeta^*_b\colon \NeronJac_b \rightarrow \NeronJac_b)$ is smooth over $R$ and satisfies the N\'eron mapping property for $\intPrym_b$. The smoothness follows from Lemma \ref{lemma: fixed points smooth morphism is smooth} applied to $-\curvezeta_b^*$. The N\'eron mapping property follows from that of $\NeronJac_b$.	
\end{proof}

\begin{lemma}\label{lemma: curve regular implies component groups vanish}
Let $b\in \intsmallB(R)$ with $\Delta(b)\neq 0$. 
Suppose that the curve $\intsmallcurve_b/R$ is regular. Then $\NeronJac_b$ and $\NeronPrym_b$ have connected fibres.
\end{lemma}
\begin{proof}
Since $\intsmallcurve_b\rightarrow \Spec R$ has geometrically integral fibres, \cite[\S9.5, Theorem 1]{BLR-NeronModels} shows that $\NeronJac_b$ is isomorphic to $\Pic^0_{\intsmallcurve_b/R}$, the identity component of the Picard scheme of $\intsmallcurve_b\rightarrow \Spec R$. 
Since $\Pic^0_{\intsmallcurve_b/R}$ has connected fibres by definition, the same holds for $\NeronJac_b$.

It remains to consider $\NeronPrym_b$. Lemma \ref{lemma: neron model prym has same definition} and the previous paragraph shows that $\NeronPrym_b$ is isomorphic to $\Pic^0_{\intsmallcurve_b/R}[1+\curvezeta_b^*]$.  
It therefore suffices to prove that $\Pic^0_{\intsmallcurve_{b}/R}[1+\curvezeta_b^*]$ has connected fibres. 
By Proposition \ref{proposition: generalized prym variety has integral fibres} (and its analogue over $\intsmallB_S$: Property \ref{enum: int 9} of \S \ref{subsection: integral structures}), $\Pic^0_{\intsmallcurve_S/\intsmallB_S}[1+\curvezeta^*]\rightarrow \intsmallB_S$ has connected fibres. 
Since $\Pic^0_{\intsmallcurve_{b}/R}[1+\curvezeta_b^*]$ is the pullback of $\Pic^0_{\intsmallcurve_S/\intsmallB_S}[1+\curvezeta^*]$ along the $R$-point $b$, the lemma follows.  
\end{proof}

\begin{proposition}\label{proposition: special fibres neron models}
	Let $R$ be a discrete valuation ring in which $N$ is a unit. Let $K = \Frac R$ and let $\ord_K: K^{\times} \twoheadrightarrow \Z$ be the normalized discrete valuation. Let $b\in \intsmallB(R)$ and suppose that $\ord_K \Delta(b)\leq 1 $. Let $\NeronJac_b, \NeronPrym_b, \NeronPrym^{\vee}_b$ and $\Neronellcurve_b$ be the N\'eron models of $\intJac_b$, $\intPrym_b, \intPrym_b^{\vee}$ and $\intellcurve_b$ respectively. Then:
	\begin{enumerate}
		\item The special fibres of the $R$-groups $\NeronJac_b, \NeronPrym_b, \NeronPrym_b^{\vee}$ and $\Neronellcurve_b$ are connected.
		\item If $\ord_K \Delta(b)=1 $, the special fibre of the quasi-finite \'etale $R$-group scheme $\NeronPrym_b[2]$ has order $2^3$.  
	\end{enumerate}
\end{proposition}
\begin{proof}
	If $\ord_K \Delta(b)= 0 $, all abelian varieties in question have good reduction so the proposition holds. Thus for the remainder of the proof we assume that $\ord_K \Delta(b)= 1 $.
	The choice of $N$ implies that $\Delta$ equals $\Delta_E \cdot  \Delta_{\hat{E}}$ up to a unit in $\Z[1/N]$ (Property \ref{enum: int 3} of \S\ref{subsection: integral structures}).
	This allows us to consider two separate cases, the first case being $(\ord_K \Delta_{\ellcurve}(b),\ord_K \Delta_{\hat{\ellcurve}}(b)) = (1,0)$ and the second case being $(\ord_K \Delta_{\ellcurve}(b),\ord_K \Delta_{\hat{\ellcurve}}(b)) = (0,1)$.
	Let $\mathcal{R}$ and $\mathcal{B}$ be the closed subschemes of $\intsmallcurve_b$ and $\intCellcurve_b$ given by intersecting them with the line $\{y=0\} \subset \P^2_R$ using Equations (\ref{equation: F4 family integral reps}) and (\ref{equation: ellcurve family integral reps}) respectively. 
	Then the morphism $\intsmallcurve_b \rightarrow \intCellcurve_b$ restricts to a finite \'etale morphism $\intsmallcurve_b - \mathcal{R} \rightarrow \intCellcurve_b - \mathcal{B}$. 
	We recall that $\Delta_{\hat{\ellcurve}}(b) = 4p_8(b)^3+27p_{12}(b)^2$.
	
	\begin{itemize}
	\item[Case 1.] In this case the discriminant of $\intCellcurve_b$ has valuation $1$. By Tate's algorithm (see \cite[Lemma IV.9.5(a)]{Silverman-advancedtopicsarithmeticellcurves}), this implies that $\intCellcurve_b$ is regular and its special fibre has a unique singularity, which is a node. 
	Since $\ord_K \Delta_{\hat{\ellcurve}}(b) =0$, the singular point of $\intCellcurve_{b,k}$ is not contained in $\mathcal{B}$ hence it lifts to two distinct singular points of $\intsmallcurve_{b,k}$ which are also nodes. 
	Since $\intsmallcurve_b \rightarrow \intCellcurve_b$ is \'etale outside $\mathcal{B}$, $\intsmallcurve_b$ is regular and its special fibre has two nodal singular points which are swapped by the involution $\curvezeta_b\colon \intsmallcurve_b \rightarrow \intsmallcurve_b$. 
	
	\item[Case 2.] Now $\intCellcurve_b$ is smooth over $R$ so the singular points of $\intsmallcurve_b$ are contained in $\mathcal{R}$. 
	Since $\ord_K \Delta_{\hat{\ellcurve}}(b) =1$, the discriminant of the polynomial $x^3+p_8x+p_{12}$ has valuation $1$. 
	Since $2,3\in R^{\times}$, the unique multiple root of the reduction of this polynomial lies in $k$; let $\alpha\in R$ be a lift of this root. 
	Using the substitution $x\mapsto x-\alpha$, the curve $\intsmallcurve_b$ is given by the equation 
	\begin{equation}\label{equation: neron model conn fibre general equation}
	y^4+a_2xy^2+a_6y^2=x^3+a_4x^2+a_8x+a_{12},
	\end{equation}
	for some $a_i\in R$ with $\ord_Ka_8\geq 1$ and $\ord_Ka_{12}\geq 1$. 
	Since the discriminant of the cubic polynomial on the right hand side of (\ref{equation: neron model conn fibre general equation}) has valuation $1$, the formula for such a discriminant shows that $a_{12}$ is a uniformizer and $a_4\in R^{\times}$. 
	Since $\intCellcurve_b$ is smooth over $R$, we have $a_6\in R^{\times}$. 
	Therefore $\intsmallcurve_b$ is regular and its special fibre contains a unique nodal singularity. (For this last claim, see \cite[Exercise 7.5.7(b)]{Liu-AlgebraicGeometryArithmeticCurves}.)
	\end{itemize}
	
	We conclude that in both cases $\intCellcurve_b$ and $\intsmallcurve_b$ are regular.
	Since $\Neronellcurve_b$ can be identified with the smooth locus of $\intCellcurve_b$ \cite[Theorem IV.9.1]{Silverman-advancedtopicsarithmeticellcurves}, the special fibre of $\Neronellcurve_b$ is connected. 
	By Lemma \ref{lemma: curve regular implies component groups vanish}, the special fibres of $\NeronJac_b$ and $\NeronPrym_b$ are connected.
	Because of the isomorphism $\NeronPrym_{b}^{\vee} \simeq \NeronPrym_{\hat{b}}$ (Theorem \ref{theorem: summary bigonal construction} and its spreading out: Property \ref{enum: int 10} of \S\ref{subsection: integral structures}) and the fact that $\Delta(b)$ and $\Delta(\hat{b})$ are equal up to a unit in $R$, the connectedness of the special fibre of $\NeronPrym_b^{\vee}$ follows from that of $\NeronPrym_b$.
	
	For the second part of the lemma, it suffices to prove that the special fibre of $\NeronPrym_b$ is an extension of an elliptic curve by a rank $1$ torus. This follows from the fact that the special fibres of $\NeronJac_b$ and $\Neronellcurve_b$ are semiabelian varieties of toric rank $2$ and $1$ respectively in the first case and of toric rank $1$ and $0$ in the second case.
	 
\end{proof}

We proceed with a finer analysis of N\'eron models of Prym varieties, only necessary to obtain a lower bound in Theorem \ref{theorem: intro rho selmer}.
If $A/K$ is an abelian variety with N\'eron model $\mathscr{A}/R$ we write $\mathscr{A}^{\circ}$ for the \define{identity component} of $\mathscr{A}$, obtained by removing the connected components of $\mathscr{A}_k$ not containing the identity section. 
Recall from Lemma \ref{lemma: neron model prym has same definition} that we may view $\NeronPrym_b$ as a closed subgroup scheme of $\NeronJac_b$.

\begin{definition}
Let $b\in \intsmallB(R)$ with $\Delta(b)\neq 0$. 
We say $b$ is \define{admissible} if $\NeronJac_b^{\circ} \cap \NeronPrym_b=\NeronPrym_b^{\circ}$ or equivalently, $\NeronJac_b^{\circ} \cap \NeronPrym_b$ has connected fibres.
\end{definition}

The reason for introducing admissibility is Lemma \ref{lemma: admissible b induces exact sequence}. 
It seems unlikely that every $b\in \intsmallB(R)$ with $\Delta(b)\neq 0$ is admissible, but we have not found a counterexample. 
Our first goal is establishing a sufficient condition for admissibility, Proposition \ref{proposition: DeltaE squarefree admissible}. 
This we achieve with the help of the following two lemmas.

\begin{lemma}\label{lemma: admissible if regular model has good tau props}
Let $b\in \intsmallB(R)$ with $\Delta(b)\neq 0$. 
Let $\tilde{\mathcal{C}}$ be a regular model of $\smallprojcurve_b$, i.e. a regular, proper, flat $R$-scheme whose generic fibre is isomorphic to $\smallprojcurve_b$. 
Suppose that the involution $\tau_b$ of $\smallprojcurve_b$ extends to an involution $\tau_b$ of $\tilde{\mathcal{C}}$. 
Let $\Pic^0_{\tilde{\mathcal{C}}/R}\rightarrow \Spec R$ be the identity component of the Picard scheme of $\tilde{\mathcal{C}}/R$. 
Then $b$ is admissible if (and only if) the special fibre of $\Pic^0_{\tilde{\mathcal{C}}/R}[1+\tau^*_b]$ is connected.
\end{lemma}
\begin{proof}
Since $\smallprojcurve_b$ has a $K$-rational point $\infty$, the special fibre of $\tilde{\mathcal{C}}$ has an irreducible component of degree $1$.
Therefore by a theorem of Raynaud \cite[\S9.5, Theorem 4(b)]{BLR-NeronModels}, we have an isomorphism $\NeronJac_b^{\circ}\simeq \Pic^0_{\tilde{\mathcal{C}}/R}$. 
This isomorphism intertwines the involutions $\tau_b^*$ on both sides, because these involutions are the unique extensions of their restriction to the generic fibre.
By Lemma \ref{lemma: neron model prym has same definition}, we see that $\NeronJac_b^{\circ}\cap \NeronPrym_b\simeq \Pic^0_{\tilde{\mathcal{C}}/R}[1+\tau_b^*]$. 
Since the generic fibre of $\NeronJac_b^{\circ}\cap \NeronPrym_b$ equals $\intPrym_b$ which is connected, the equivalence of the lemma follows from the definition of admissibility. 
\end{proof}

For the statement of the next lemma, recall \cite[Expose $\text{VI}_{\text{A}}$, Theoreme 5.4.2]{SGA3-TomeI} that the category of finite type commutative group schemes over a field is abelian. 

\begin{lemma}\label{lemma: exact sequence involution}
Let 
\begin{equation*}
	1 \rightarrow A \rightarrow B
	 \rightarrow C \rightarrow 1
\end{equation*}	
be a short exact sequence of finite type commutative group schemes over $k$. 
Let $\tau$ be an involution of $A$, $B$ and $C$ whose action is compatible with the above sequence.
Suppose that either (1) the quotient $A^{\tau}/(1+\tau)A$ is trivial, or (2) $A^{\tau}/(1+\tau)A$ is finite \'etale and $C[1+\tau]$ is connected. 
Then the following sequence is short exact:
\begin{equation*}
	1 \rightarrow A[1+\tau] \rightarrow B[1+\tau] \rightarrow C[1+\tau] \rightarrow 1.
\end{equation*}
\end{lemma}
\begin{proof}
    We consider $A, B$ and $C$ as sheaves on the big fppf site of $\Spec k$.
	The long exact sequence in $\Z/2\Z$-group cohomology of sheaves applied to the $\Z/2\Z$-action $-\tau$ shows that the following sequence is exact:
	\begin{equation*}
	1 \rightarrow A[1+\tau] \rightarrow B[1+\tau] \rightarrow C[1+\tau] \xrightarrow{\delta} A^{\tau}/(1+\tau)A.
	\end{equation*}
	(Alternatively, the exactness of this sequence is a statement that can be formulated in any abelian category. Since it is true in the category of $R$-modules for any ring $R$, it remains true in our setting.)
	If $A^{\tau}/(1+\tau)A$ is trivial, the lemma is proven. 
	If $A^{\tau}/(1+\tau)A$ is finite \'etale and $C[1+\tau]$ is connected, then $\delta=0$ since there are no nonconstant maps from a geometrically connected scheme to a finite \'etale $k$-scheme. 
\end{proof}

Recall that (up to a unit in $R$) the discriminant $\Delta$ factors as $\Delta_E\cdot \Delta_{\hat{E}}$.
The proof of Proposition \ref{proposition: special fibres neron models} shows that every $b\in \intsmallB(R)$ with $\ord_K\Delta(b)\leq 1$ is admissible. We will need the stronger: 

\begin{proposition}\label{proposition: DeltaE squarefree admissible}
Let $b\in \intsmallB(R)$ with $\Delta(b)\neq 0$ and $\ord_K\Delta_E(b)\leq 1$. 
Then $b$ is admissible. 
\end{proposition}
\begin{proof}
Since N\'eron models commute with the formation of strict henselization and completion \cite[\S7.2, Theorem 1(b)]{BLR-NeronModels}, we may assume that $R$ is complete and its residue field $k$ is separably closed.
We distinguish cases according to the value of $\ord_K \Delta_E(b)$.

\underline{Case $\ord_K \Delta_E(b)=0$.}

If $\intsmallcurve_b$ is regular, $b$ is admissible by Lemma \ref{lemma: curve regular implies component groups vanish}. 
We may therefore assume that $\intsmallcurve_b$ has a non-regular point $P\in \intsmallcurve_b$. 
Since $\intCellcurve_b$ is $R$-smooth, $P$ is the unique non-regular point and lies in the ramification locus of the morphism $\intsmallcurve_b\rightarrow \intCellcurve_b$.
By the same reasoning as Case 2 in the proof of Proposition \ref{proposition: special fibres neron models}, we may assume after changing variables $x\mapsto x-\alpha$ that $\intsmallcurve_b$ is given by the equation
\begin{equation}\label{equation: neron proof deltaE=0}
y^4+a_2xy^2+a_6y^2=x^3+a_4x^2+a_8x+a_{12}
\end{equation}
for some $a_i\in R$ with $\ord_K a_{8} \geq 1$ and $\ord_K a_{12} \geq 1$, and $P$ corresponds to the origin in the special fibre. 
Again by the smoothness of $\intCellcurve_b/R$, we see that $a_6\in R^{\times}$.
Therefore the completed local ring of $P_k$ in $\intsmallcurve_{b,k}$ is isomorphic to $k[[x,y]]/(y^2-(x^3+a_4x^2))$.
So $\intsmallcurve_{b,k}$ has one singular point which is a node or a cusp.

Consider the sequence of proper birational morphisms $\dots \rightarrow \mathcal{C}_2 \rightarrow \mathcal{C}_1 \rightarrow \mathcal{C}_0\coloneqq \intsmallcurve_b$, where for $i\geq 0$ we inductively define $\mathcal{C}_{i+1}\rightarrow \mathcal{C}_i$ as the composition of the blowup $\mathcal{C}_i' \rightarrow \mathcal{C}_i$ of the non-regular locus and the normalization $\mathcal{C}_{i+1}\rightarrow \mathcal{C}_i'$.
By a result of Lipman \cite[Theorem 8.3.44]{Liu-AlgebraicGeometryArithmeticCurves}, there exists an $n\geq 1$ such that the scheme $\mathcal{C}_n$ is regular; we denote this scheme by $\tilde{\mathcal{C}}$.
The morphism $\tilde{\mathcal{C}}\rightarrow \intsmallcurve_b$ does not depend on $n$ and we call it the \define{canonical desingularization} of $\intsmallcurve_b$.
Since this process is canonical, the involution $\tau_b$ of $\intsmallcurve_b$ lifts to an involution of $\tilde{\mathcal{C}}$.

Let $X$ be the closure of $\intsmallcurve_{b,k}\setminus \{P\}$ in $\tilde{\mathcal{C}}_k$ and let $Y\rightarrow X$ be its normalization. 
The composite $Y\rightarrow \intsmallcurve_{b,k}$ is also the normalization of $\intsmallcurve_{b,k}$. 
Since $\intsmallcurve_{b,k}$ has arithmetic genus $3$ and resolving a cusp or node decreases the genus by $1$, the smooth curve $Y$ has genus $2$. 
The involution $\tau_b$ of $\intsmallcurve_{b,k}$ uniquely lifts to an involution $\tau_b$ of $Y$. 
Moreover the composite morphism $Y\rightarrow X\rightarrow \intsmallcurve_{b,k}\rightarrow \intCellcurve_{b,k}$ is a double cover of a smooth genus-$1$ curve by a smooth genus-$2$ curve hence has two branch points.
Therefore the Prym variety $\Pic^0_{Y/k}[1+\tau_b^*]$ of this cover is connected and one-dimensional: the connectedness follows from \cite[\S2, Property (vi)]{Mumford-prymvars} (in particular the description of `$\ker \psi$' there) combined with \cite[\S3, Lemma 1]{Mumford-prymvars}.

We have an exact sequence \cite[\S9.2, Corollary 11]{BLR-NeronModels}
$$1\rightarrow G \rightarrow \Pic^0_{\tilde{\mathcal{C}}_k/k} \rightarrow  \Pic^0_{Y/k} \rightarrow 1,$$
where $G$ is a smooth commutative algebraic group of dimension $1$ which is an extension of an abelian variety by a connected linear algebraic group, hence connected.
Since $\Pic^0_{\tilde{\mathcal{C}}_{k}/k}[1+\tau_b^*]$ is two-dimensional and $\Pic^0_{Y/k}[1+\tau_b^*]$ is one-dimensional, $G[1+\tau_b^*]$ must be one-dimensional hence equal to $G$ itself.
Therefore $\tau^*_b|_{G}=-\Id_G$ and so $G^{\tau^*_b}/(1+\tau^*_b)G=G[2]$, which is finite \'etale (note that $2$ is invertible in $k$).
Since $\Pic^0_{Y/k}[1+\tau_b^*]$ is connected, Lemma \ref{lemma: exact sequence involution}(2) shows that the following sequence is exact:
$$1\rightarrow G \rightarrow \Pic^0_{\tilde{\mathcal{C}}_{k}/k}[1+\tau_b^*] \rightarrow  \Pic^0_{Y/k}[1+\tau_b^*] \rightarrow 1.$$
Since the outer terms of the sequence are connected, the same is true for the middle term. Therefore $b$ is admissible by Lemma \ref{lemma: admissible if regular model has good tau props}.

\underline{Case $\ord_K \Delta_E(b)=1$.}

By assumption, the curve $\intCellcurve_b/R$ is regular and its special fibre has a unique singularity, which is a node \cite[Lemma IV.9.5(a)]{Silverman-advancedtopicsarithmeticellcurves}; let $Q\in \intCellcurve_b$ be this point.
Recall that the morphism $f\colon \intsmallcurve_b\rightarrow \intCellcurve_b$ is branched over the closed subscheme $(y=0)$ and \'etale over the complement.
We distinguish two cases.  

\begin{itemize}
	\item Suppose that $Q$ is contained in the branch locus of $f$. 
Then $Q$ uniquely lifts to $P\in \intsmallcurve_b$ and $\intsmallcurve_b/R$ is smooth outside $P$. 
We may then assume after changing variables $x\mapsto x-\alpha$ that $\intsmallcurve_b$ is given by the equation
$$y^4+a_2xy^2+a_6y^2=x^3+a_4x^2+a_8x+a_{12},$$
where $\ord_K a_8\geq 1$ and $\ord_K a_{12} \geq 1$. 
Since $\intCellcurve_b/R$ is not smooth at $Q$, we have $\ord_K a_6\geq 1$. 
On the other hand, $\intCellcurve_b$ is regular at $Q$ so $a_{12}$ is a uniformizer for $R$. 
Therefore $\intsmallcurve_b$ is also regular at $P$.
So $\intsmallcurve_b$ is regular everywhere, hence $b$ is admissible by Lemma \ref{lemma: curve regular implies component groups vanish}.
	\item Suppose that $Q$ is not contained in the branch locus of $f$. 
Since $\intsmallcurve_b \rightarrow \intCellcurve_b$ is \'etale above $Q$, this point lifts to two regular nodal points $P_1,P_2$ of $\intsmallcurve_b$. 
Therefore the curve $\intsmallcurve_b$ is regular outside $(y=0)\subset \intsmallcurve_{b,k}$.
Let $\tilde{\intsmallcurve} \rightarrow \intsmallcurve_b$ be the canonical desingularization, described in the second paragraph of the first case of the proof. 
The involution $\tau_b$ of $\intsmallcurve_b$ lifts to an involution of $\tilde{\intsmallcurve}$.
Let $X\rightarrow \tilde{\intsmallcurve}_k$ be the partial normalization of the special fibre, given by normalizing the nodes corresponding to $P_1$ and $P_2$. 
Then $\tau_b$ also lifts to an involution of $X$. 
We have a $\tau_b^*$-equivariant exact sequence 
$$1 \rightarrow \G_m\times \G_m \rightarrow \Pic^0_{\tilde{\intsmallcurve}_k/k}\rightarrow\Pic^0_{X/k}\rightarrow 1,$$
where $\tau^*$ acts on $\G_m\times \G_m$ by interchanging the two factors.
Since $\Pic^0_{\tilde{\intsmallcurve}_k/k}[1+\tau^*]$ is two-dimensional and $(\G_m\times \G_m)[1+\tau_b^*]$ is one-dimensional, $\Pic^0_{X/k}[1+\tau^*]=\Pic^0_{X/k}$ by dimension reasons. 
By Lemma \ref{lemma: exact sequence involution}(1) we obtain an exact sequence 
$$1\rightarrow \G_m \rightarrow \Pic^0_{\tilde{\intsmallcurve}_k/k}[1+\tau^*]\rightarrow\Pic^0_{X/k}\rightarrow 1. $$
Since the outer terms are connected, the same is true for the middle term hence $b$ is admissible by Lemma \ref{lemma: admissible if regular model has good tau props}.
\end{itemize}
\end{proof}

%
%

The reason for introducing admissibility is the following lemma, which is a key ingredient for Proposition \ref{proposition: admissible implies rho exact}.

\begin{lemma}\label{lemma: admissible b induces exact sequence}
Let $b\in \intsmallB(R)$ be admissible. 
Then the following commutative diagram has exact rows:
\begin{center}
\begin{tikzcd}
0 \arrow[r]& \NeronPrym^{\circ}_b	\arrow[r] \arrow[d] & \NeronJac^{\circ}_b \arrow[r] \arrow[d] & \Neronellcurve^{\circ}_b \arrow[r]\arrow[d] & 0 \\
0 \arrow[r] & \NeronPrym_b	\arrow[r] & \NeronJac_b \arrow[r] & \Neronellcurve_b & 
\end{tikzcd}	
\end{center}

\end{lemma}
\begin{proof}
The exactness of the bottom row follows from the exact sequence $0\rightarrow \intPrym_b\rightarrow \intJac_b\rightarrow \intellcurve_b\rightarrow 0$ and \cite[\S7.5, Proposition 3(a)]{BLR-NeronModels}, noting that there exists an injection $\intellcurve_b\hookrightarrow \intJac_b$ such that the composite $\intellcurve_b\rightarrow \intJac_b \rightarrow \intellcurve_b$ is multiplication by $2$ (Property \ref{enum: prym 1} of \S\ref{subsection: def prym variety}). 
To verify exactness of the top row, note that again by \cite[\S7.5, Proposition 3(a)]{BLR-NeronModels} the image of $\NeronJac_b\rightarrow \Neronellcurve_b$ contains $\Neronellcurve_b^{\circ}$. 
Hence the top row is exact at $\Neronellcurve_b^{\circ}$. 
Since $b$ is admissible, it is also exact at $\NeronJac_b^{\circ}$. 
Finally because $\NeronPrym_b\rightarrow \NeronJac_b$ is a closed immersion, the same holds for $\NeronPrym^{\circ}_b\rightarrow \NeronJac^{\circ}_b$ so the top row is exact at $\NeronPrym_b^{\circ}$ too. 
\end{proof}

\begin{remark}
If $b$ is not admissible, the top row of the above commutative diagram fails to be exact at $\NeronJac_b^{\circ}$. 
\end{remark}

If $A/K$ is an abelian variety with N\'eron model $\mathscr{A}/R$, define the \define{component group} of $\mathscr{A}$ as $\Phi_A \coloneqq \mathscr{A}_k/\mathscr{A}_k^{\circ}$, a finite \'etale group scheme over $k$.

\begin{proposition}\label{proposition: admissible implies rho exact}
	Let $b\in \intsmallB(R)$ with $\Delta(b)\neq 0$. 
	Suppose that $b$ is admissible and that $\Neronellcurve_b=\Neronellcurve_b^{\circ}$. 
	Then the morphism $\rho\colon \NeronPrym_b \rightarrow \NeronPrym_b^{\vee}$ induces an isomorphism of component groups $\Phi_{\intPrym_b}\xrightarrow{\sim} \Phi_{\intPrym_b^{\vee}}$.
\end{proposition}
\begin{proof}
Since $\rhodual \circ \rho = [2]$ and $\rho \circ \rhodual = [2]$, it will suffice to prove that the restriction to $2$-primary parts $\Phi_{\intPrym_b}[2^{\infty}]\rightarrow \Phi_{\intPrym_b^{\vee}}[2^{\infty}]$ is an isomorphism of finite \'etale group schemes.
By definition, $\rho$ is given by the composite of $\NeronPrym_b\rightarrow \NeronJac_b$ with $\NeronJac_b\rightarrow \NeronPrym_b^{\vee}$.
So it will suffice to prove that the morphisms $\Phi_{\intPrym_b}[2^{\infty}]\rightarrow \Phi_{\intJac_b}[2^{\infty}]$ and $\Phi_{\intJac_b}[2^{\infty}]\rightarrow \Phi_{\intPrym^{\vee}_b}[2^{\infty}]$ are isomorphisms.

By Lemma \ref{lemma: admissible b induces exact sequence} and the snake lemma, we obtain an exact sequence $0\rightarrow \Phi_{\intPrym_b}\rightarrow \Phi_{\intJac_b} \rightarrow \Phi_{\intellcurve_b}$. 
Since $\Phi_{\intellcurve_b}$ is trivial by assumption, $\Phi_{\intPrym_b}\rightarrow \Phi_{\intJac_b}$ is an isomorphism of finite \'etale group schemes. 
Since Grothendieck's pairing on component groups is perfect on $l$-primary parts when $l$ is invertible in $k$ \cite[Theorem 7]{Bertapelle-perfectnessgrothendiecklparts}, the finite \'etale group schemes $\Phi_{\intPrym_b}[2^{\infty}], \Phi_{\intJac_b}[2^{\infty}]$ and $\Phi_{\intPrym^{\vee}_b}[2^{\infty}] $ have the same order. 

By the same reasoning as the proof of Lemma \ref{lemma: admissible b induces exact sequence} using the fact that $\NeronJac_b^{\circ}\cap \Neronellcurve_b=\Neronellcurve_b^{\circ} = \Neronellcurve_b$, the exact sequence $0\rightarrow \intellcurve_b\rightarrow \intJac_b\rightarrow \intPrym_b^{\vee}\rightarrow 0$ induces a commutative diagram with exact rows: 
\begin{center}
\begin{tikzcd}
0 \arrow[r]& \Neronellcurve^{\circ}_b	\arrow[r] \arrow[d] & \NeronJac^{\circ}_b \arrow[r] \arrow[d] & \NeronPrym^{\vee,\circ}_b \arrow[r]\arrow[d] & 0 \\
0 \arrow[r] & \Neronellcurve_b	\arrow[r] & \NeronJac_b \arrow[r] & \NeronPrym^{\vee}_b & 
\end{tikzcd}	
\end{center}
The snake lemma gives an exact sequence $0 \rightarrow \Phi_{\intellcurve_b} \rightarrow \Phi_{\intJac_b} \rightarrow \Phi_{\intPrym^{\vee}_b}$. 
Since $\Phi_{\intellcurve_b}$ is trivial, the morphism of finite \'etale $k$-group schemes $\Phi_{\intJac_b} \rightarrow \Phi_{\intPrym^{\vee}_b}$ is injective. 
Since the $2$-primary parts have the same order, the induced morphism $\Phi_{\intJac_b}[2^{\infty}] \rightarrow \Phi_{\intPrym^{\vee}_b}[2^{\infty}]$ is an isomorphism, proving the proposition.

\end{proof}

The following important corollary will be the one that we will use later on.

\begin{corollary}\label{corollary: squarefree implies unramified selmer condition rhovee descent}
Let $p$ be a prime not dividing $N$. Let $b\in \intsmallB(\Z_p)$ with $\Delta(b)\neq 0$ and $p^2\nmid \Delta_{\hat{\ellcurve}}(b)$. 
Then the image of the $\rhodual$-descent map $\Prym_b(\Q_p)/\rhodual(\Prym^{\vee}_b(\Q_p)) \rightarrow \HH^1(\Q_p, \Prym^{\vee}_b[\rhodual])$ coincides with the subset of unramified classes $\HH^1_{\text{nr}}(\Q_p, \Prym^{\vee}_b[\rhodual])$.
\end{corollary}
\begin{proof}
	Since $\Delta_{\hat{\ellcurve}}(b)$ and $\Delta_{\ellcurve}(\hat{b})$ coincide up to a unit in $\Z_p$, we see that $p^2\nmid \Delta_{\ellcurve}(\hat{b})$.
	By Proposition \ref{proposition: DeltaE squarefree admissible}, $\hat{b}$ is admissible and by Tate's algorithm \cite[Lemma IV.9.5(a)]{Silverman-advancedtopicsarithmeticellcurves},  $\Neronellcurve_{\hat{b}}=\Neronellcurve_{\hat{b}}^{\circ}$. 
	Therefore by Proposition \ref{proposition: admissible implies rho exact} and Theorem \ref{theorem: summary bigonal construction}, $\rhodual$ induces an isomorphism $\Phi_{\Prym_b^{\vee}} \xrightarrow{\sim} \Phi_{\Prym_b}$. 
    Under these circumstances, the claim of the corollary is well-known and follows essentially from Lang's theorem; see \cite[Proposition 2.7(d)]{Cesnavicius-Selmergroupsflatcohomology}.
\end{proof}

\subsection{The case of square-free discriminant}\label{subsection: integral reps: the $2$-Selmer case}

In this section we analyze the orbits of $\intsmallV$ and $\intrhoV$ over points in $\intsmallB(\Z_p)$ and $\intrhoB(\Z_p)$ of square-free discriminant. 
This will be the first step in proving Theorem \ref{theorem: F4 integral representatives exist} and will be used in the proofs of Theorems \ref{theorem: average size rhovee selmer group} and \ref{theorem: average size strongly irred 2 selmer} when applying the square-free sieve.

\begin{lemma}\label{lemma: sqfree discriminant element of V F4}
    Let $R$ be a discrete valuation ring with residue field $k$ in which $N$ is a unit. 
    Let $K = \Frac R$ and let $\ord_K: K^{\times} \twoheadrightarrow \Z$ be the normalized discrete valuation.
    Let $x\in \intsmallV(R)$ with $b=\pi(x)\in \intsmallB(R)$ and suppose that $\ord_K \Delta(b)=1$.
    Then the reduction $x_k$ of $x$ in $\intsmallV(k)$ is regular and $\intsmallG(\bar{k})$-conjugate to $\sigma(b)_k$.
    In addition the $R$-group scheme $Z_{\intsmallG}(x)$ is quasi-finite \'etale and has special fibre of order $2^3$.
\end{lemma}

\begin{proof}
    We are free to replace $R$ by a discrete valuation ring $R'$ containing $R$ such that any uniformizer in $R$ is also a uniformizer in $R'$.
    Therefore we may assume that $R$ is complete and $k$ algebraically closed.

    Let $x_k=y_s+y_n$ be the Jordan decomposition of $x_k\in \intsmallV(k)$ as a sum of its semisimple and nilpotent parts. 
    Let $\intsmallh_{0,k}=\mathfrak{z}_{\intsmallh}(y_s)$ and $\intsmallh_{1,k}=\image(\Ad(y_s))$.
    Then $\intsmallh_{k}=\intsmallh_{0,k}\oplus \intsmallh_{1,k}$, where $\Ad(x_k)$ acts nilpotently on $\intsmallh_{0,k}$ and invertibly on $\intsmallh_{1,k}$.
    By Hensel's lemma, this decomposition lifts to an $\Ad(x)$-invariant decomposition of free $R$-modules $\intsmallh_R = \intsmallh_{0,R}\oplus \intsmallh_{1,R}$, where $\Ad(x)$ acts topologically nilpotently on $\intsmallh_{0,R}$ and invertibly on $\intsmallh_{1,R}$.
    There exists a unique closed subgroup $\mathsf{L}\subset \intsmallH_R$ with Lie algebra $\intsmallh_{0,R}$ such that $\mathsf{L}$ is $R$-smooth with connected fibres; this follows from an argument identical to the proof of \cite[Lemma 4.19]{Laga-E6paper}.
    Moreover the construction of $\mathsf{L}$ shows that $\mathsf{L}_k=Z_{\intsmallH}(y_s)$.
    
    The proof of Proposition \ref{proposition: special fibres neron models} shows that the curve $\intsmallcurve_{b,k}$ either has one node, or two nodes swapped by $\tau$.
    Therefore the affine surface $\mathcal{S}/k$ cut out by the equation
    $z^2+y^4+p_2(b)xy^2+p_6(b)y^2 = x^3+p_8(b)x+p_{12}(b)$ in $\A^3_k$ either has one ordinary double point, or two such double points swapped by $(x,y,z)\mapsto (x,-y,-z)$.
    The surface $\mathcal{S}$ is the fibre above $b_k\in \intsmallB(k)$ of a semi-universal deformation of a simple singularity of type $F_4$, in the sense of \cite[\S6.2]{Slodowy-simplesingularitiesalggroups}.
    The results of \cite[\S6.6]{Slodowy-simplesingularitiesalggroups} (in particular Propositions 2, 3 and the subsequent remark) imply that the derived group of $\mathsf{L}$ has type $A_1$ and the center $Z(\mathsf{L})$ of $\mathsf{L}$ has rank $3$.
    Moreover the restriction $\theta_{\mathsf{L}}$ of $\theta$ to $\mathsf{L}$ is a stable involution, in the sense that for each geometric point of $\Spec R$ there exists a maximal torus of $\mathsf{L}$ on which $\theta$ acts as $-1$, by \cite[Lemma 2.4]{Thorne-thesis}.
    There is an isomorphism $\mathsf{L}/Z(\mathsf{L})\simeq \PGL_2$ inducing an isomorphism $\intsmallh_{R,0}^{der}\simeq \intsmallh_{R,0}/\mathfrak{z}(\intsmallh_{R,0}) \simeq \liesl_{2,R}$ under which $\theta_{\mathsf{L}}$ corresponds to the involution $\xi = \Ad\left(\text{diag}(1,-1) \right)$.
    The lemma now follow easily from explicit calculations in $\liesl_{2,R}$ identical to \cite[Lemma 4.19]{Laga-E6paper}, which we omit.
\end{proof}

The following proposition and its corollary describe orbits in $\intsmallV$ of square-free discriminant. Their proofs are identical to the proofs of \cite[Proposition 4.20 and Corollary 4.21]{Laga-E6paper}, using Proposition \ref{proposition: special fibres neron models} and Lemma \ref{lemma: sqfree discriminant element of V F4}; they will be omitted.

\begin{proposition}\label{prop: integral reps squarefree discr F4 case}
	Let $R$ be a discrete valuation ring in which $N$ is a unit. Let $K = \Frac R$ and let $\ord_K: K^{\times} \twoheadrightarrow \Z$ be the normalized discrete valuation. 
	Let $b\in \intsmallB(R)$ and suppose that $\ord_K \Delta(b)\leq 1$. 
	Then:
	\begin{enumerate}
	   \item If $x\in \intsmallV_b(R)$, then $Z_{\intsmallG}(x)(K) = Z_{\intsmallG}(x)(R)$. 
		\item The natural map $\alpha\colon \intsmallG(R)\backslash \intsmallV_b(R) \rightarrow \intsmallG(K)\backslash \intsmallV_b(K)$ is injective and its image contains $\eta_b\left(\Prym_b(K)/2\Prym_b(K)\right)$. 
		\item If further $R$ is complete and has finite residue field then the image of $\alpha$ equals $\eta_b\left(\Prym_b(K)/2\Prym_b(K)\right)$. 
	\end{enumerate}
\end{proposition}


\begin{corollary}\label{corollary: extend groupoid dedekind scheme squarefree}
Let $X$ be a Dedekind scheme in which $N$ is a unit with function field $K$. 
	For every closed point $p$ of $X$ write $\ord_{p} \colon K^{\times} \twoheadrightarrow \Z$ for the normalized discrete valuation of $p$. 
	Let $b\in \intsmallB(X)$ be a morphism such that $\ord_{p}(\Delta(b))\leq 1$ for all $p$. 
	Let $P\in \intPrym_b(K)/2\intPrym_b(K)$ and let $\eta_b(P)\in G(K) \backslash V_b(K)$ be the corresponding orbit from Proposition \ref{proposition: spread out selmer group embedding}. 
	Then the object of $\GrLieE_{K,b}$, corresponding to $\eta_b(P)$ using Proposition \ref{proposition: G-orbits in terms of groupoids}, uniquely extends to an object of $\GrLieE_{X,b}$.
\end{corollary}

We now consider orbits of square-free discriminant in the representation $\intrhoV$. We will only need to consider the case of $\Z_p$; the representation-theoretic input of the following proposition has already been established by Bhargava and Shankar \cite{BS-2selmerellcurves}.

\begin{proposition}\label{prop: integral reps squarefree rho case}
	Let $p$ be a prime number not dividing $N$ and let $b\in \intsmallB(\Z_p)$ such that $\Delta(b)\neq 0$ and $p^2\nmid \Delta_{\hat{E}}(b)$. 
	Let $b^{\star} = \bigresolv(b) \in \intrhoB(\Z_p)$. Then:
	\begin{enumerate}
	   \item If $x\in \intrhoV_{b^{\star}}(\Z_p)$, then $Z_{\intrhoG}(x)(\Q_p) = Z_{\intrhoG}(x)(\Z_p)$. 
		\item The natural map $\alpha \colon \intrhoG(\Z_p)\backslash \intrhoV_{b^{\star}}(\Z_p) \rightarrow \rhoG(\Q_p)\backslash \rhoV_{b^{\star}}(\Q_p)$ is injective and its image equals $\eta^{\star}_{b}\left(\Prym_{b}(\Q_p)/\rhodual(\Prym^{\vee}_{b}(\Q_p))\right)$. 
	\end{enumerate}
\end{proposition}
\begin{proof}

    Let $E'/\Q_p$ be the elliptic curve with Weierstrass equation $y^2=x^3+I(\bigresolv(b))/9x-J(\bigresolv(b))/27$, where $\bigresolv\colon\intsmallB_S \rightarrow \intrhoB_S$ is the resolvent binary quartic map from \S\ref{subsection: the representation rhoV} and $I, J$ are the invariants of a binary quartic form of (\ref{equation: invariant I}) and (\ref{equation: invariant J}).
    By Equation (\ref{equation: bigonal ell curve}), Proposition \ref{proposition: comparison invariants ell curve V} and the choice of $N$, $\hat{\ellcurve}_b$ and $E'$ are quadratic twists, where the twisting is given by an element of $\Z[1/N]^{\times}$.
    We have chosen $E'$ so that by \cite[Theorem 3.2]{BS-2selmerellcurves}, there exists an injection $\eta'\colon E'(\Q_p)/2E'(\Q_p) \hookrightarrow \rhoG(\Q_p)\backslash \rhoV_{b^{\star}}(\Q_p)$ with the property that the composite $E'(\Q_p)/2E'(\Q_p) \xrightarrow{\eta'} \rhoG(\Q_p)\backslash \rhoV_{b^{\star}}(\Q_p) \hookrightarrow \HH^1(\Q_p, \hat{\ellcurve}_b[2]) =\HH^1(\Q_p, E'[2]) $ coincides with the $2$-descent map. (The second map comes from Proposition \ref{proposition: AIT rho selmer}.)
    
    Since $p$ is a unit in $\Z[1/N]$ and the discriminant of $\hat{\ellcurve}_b$ is not divisible by $p^2$, the same is true for the discriminant of $E'$.
    Therefore the Tamagawa number of $E'$ is $1$ \cite[Lemma IV.9.5(a)]{Silverman-advancedtopicsarithmeticellcurves}, hence the image of $E'(\Q_p)/2E'(\Q_p) \xrightarrow{\eta'} \rhoG(\Q_p)\backslash \rhoV_{b^{\star}}(\Q_p) \hookrightarrow \HH^1(\Q_p, \hat{\ellcurve}_b[2])$ coincides with the subgroup of unramified classes $\HH_{\text{nr}}^1(\Q_p, \hat{\ellcurve}_b[2])\subset \HH^1(\Q_p, \hat{\ellcurve}_b[2])$ \cite[Lemma 7.1]{BruinPoonenStoll}.

    Again by the fact that the discriminant of $E'$ is square-free, \cite[Proposition 3.18]{BS-2selmerellcurves} implies that $Z_{\intrhoG}(x)(\Q_p) = Z_{\intrhoG}(x)(\Z_p)$, that $\alpha$ is injective and that the image of the composite $\intrhoG(\Z_p)\backslash \intrhoV_{b^{\star}}(\Z_p) \rightarrow \rhoG(\Q_p)\backslash \rhoV_{b^{\star}}(\Q_p) \hookrightarrow  \HH^1(\Q_p,\hat{\ellcurve}_b[2])$ is $\HH_{\text{nr}}^1(\Q_p, \hat{\ellcurve}_b[2])$. 
	By the construction of $\eta_b^{\star}$ in the proof of Theorem \ref{theorem: inject rho-descent orbits}, it remains to prove that this subset of $\HH^1(\Q_p,\hat{\ellcurve}_b[2])$ coincides with the image of the $\rhodual$-descent map $\Prym_b(\Q_p)/\rhodual(\Prym^{\vee}(\Q_p))\rightarrow  \HH^1(\Q_p,\Prym_b^{\vee}[\rhodual])$ transported along the isomorphism $\HH^1(\Q_p,\Prym_b^{\vee}[\rhodual])\simeq \HH^1(\Q_p,\hat{\ellcurve}_b[2])$ afforded by Corollary \ref{corollary: subgroups Prym[2] using bigonal}. 
	Since the latter isomorphism preserves the unramified classes (which only depend on the Galois module $\hat{\ellcurve}_b[2] \simeq \Prym_b^{\vee}[\rhodual]$), this follows from Corollary \ref{corollary: squarefree implies unramified selmer condition rhovee descent}.
\end{proof}

\subsection{Integral representatives for \texorpdfstring{$\smallV$}{V}}

In this subsection we prove Theorem \ref{theorem: F4 integral representatives exist}. Our strategy is closely modelled on the strategy of proving \cite[Theorem 4.1]{Laga-E6paper}: we deform to the case of square-free discriminant using a Bertini type theorem over $\Z_p$ and using the compactified Prym variety. 
We have done all the necessary preparations and what follows is a routine adaptation of \cite[\S4.5]{Laga-E6paper}.
The following proposition and its proof are very similar to \cite[Corollary 4.23]{Laga-E6paper}. It establishes the existence of a deformation with good properties.  

\begin{proposition}\label{proposition: deform the point in the Prym general case integral}
	Let $p$ be a prime number not dividing $N$.
	Let $b\in \intsmallB(\Z_p)$ with $\Delta(b)\neq 0$ and $Q \in \Prym_b(\Q_p)$.
	Then there exists a morphism $\mathcal{X} \rightarrow \Z_p$ that is of finite type, smooth of relative dimension $1$ and with geometrically integral fibres, together with a point $x \in \mathcal{X}(\Z_p)$ satisfying the following properties. 
	\begin{enumerate}
		\item There exists a morphism $\tilde{b}: \mathcal{X} \rightarrow \intsmallB_{\Z_p}$ with the property that $\tilde{b}(x) = b$ and that the discriminant $\Delta(\tilde{b})$, seen as a map $\mathcal{X} \rightarrow \A^1_{\Z_p}$, is not identically zero on the special fibre and is square-free on the generic fibre of $\mathcal{X}$.
		\item Write $\mathcal{X}^{\rs}$ for the open subscheme of $\mathcal{X}$ where $\Delta(\tilde{b})$ does not vanish. Then there exists a morphism $\tilde{Q}: \mathcal{X}^{\rs} \rightarrow \intPrym$ lifting the morphism $ \mathcal{X}^{\rs} \rightarrow \smallB^{\rs}_{\Z_p}$ satisfying $\tilde{Q}(x_{\Q_p}) = Q$. 
	\end{enumerate}
\end{proposition}
\begin{proof}
	We apply \cite[Proposition 4.22]{Laga-E6paper} to the compactified Prym variety $\CPrym \rightarrow \smallB$ introduced in \S\ref{subsection: compactifications}.
	In \S\ref{subsection: integral structures} we have spread out $\CPrym$ to a scheme $\intCPrym \rightarrow \intsmallB_S$ with similar properties. (Recall that $S = \Z[1/N]$.)
	Define $\mathcal{D} $ to be the pullback of $\{\Delta = 0\} \subset \intsmallB_{\Z_p}$ along $\intCPrym_{\Z_p} \rightarrow \intsmallB_{\Z_p}$.
	Since the latter morphism is proper, we may extend $Q\in \Prym_b(\Q_p)  \subset \CPrym_b(\Q_p)$ to an element of $\intCPrym_b(\Z_p)$, still denoted by $Q$.
	We now claim that the triple $(\intCPrym_{\Z_p},\mathcal{D},Q)$ satisfies the assumptions of \cite[Proposition 4.22]{Laga-E6paper}.
	Indeed, the properties of $\intCPrym_{\Z_p}$ follow from Proposition \ref{proposition: good compactifications exist}. (Or rather the analogous properties obtained by spreading out in \S\ref{subsection: integral structures}.)
	Moreover $\intCPrym_{\F_p}$ is not contained in $\mathcal{D}$ since $\Delta$ is nonzero mod $p$ by our assumptions on $N$.
	Since $\intCPrym_{\Z_p} $ is $\intsmallB_{\Z_p}$-flat, $\mathcal{D}$ is a Cartier divisor. 
	Since the smooth locus of $\intCPrym_{\Q_p} \rightarrow \intsmallB_{\Q_p}$ has complement of codimension at least two and $\{\Delta = 0\} \subset \intsmallB_{\Q_p}$ is reduced, the scheme $\mathcal{D}_{\Q_p}$ is reduced too. 
	Finally $Q_{\Q_p} \not\in \mathcal{D}_{\Q_p}$ since $b$ has nonzero discriminant. 
	
	We obtain a closed subscheme $\mathcal{X}\hookrightarrow \intCPrym_{\Z_p}$ satisfying the conclusions of \cite[Proposition 4.22]{Laga-E6paper}. 
	Write $x\in \mathcal{X}(\Z_p)$ for the section corresponding to $Q$, $\widetilde{b}$ for the restriction of $\intCPrym_{\Z_p} \rightarrow \intsmallB_{\Z_p}$ to $\mathcal{X}$ and $\widetilde{Q}$ for the restriction of the inclusion $\mathcal{X} \hookrightarrow \intCPrym_{\Z_p}$ to $\mathcal{X}^{\rs}$. 
	Then the tuple $(\mathcal{X},x,\widetilde{b},\widetilde{Q})$ satisfies the conclusion of the proposition. 
	
\end{proof}

We now give the proof of Theorem \ref{theorem: F4 integral representatives exist}.
We keep the assumptions and notation of Proposition \ref{proposition: deform the point in the Prym general case integral} and assume that we have made a choice of $(\mathcal{X},x,\widetilde{b},\widetilde{Q})$ satisfying the conclusions of that proposition. 
The strategy is to extend the orbit $\eta_b(Q)$ (which corresponds to the point $x_{\Q_p}$) to larger and larger subsets of $\mathcal{X}$.


Let $y\in \mathcal{X}$ be a closed point of the special fibre with nonzero discriminant having an affine open neighborhood containing $x_{\Q_p}$. 
Let $R$ be the semi-local ring of $\mathcal{X}$ at $x_{\Q_p}$ and $y$. 
Since every projective module of constant rank over $R$ is free, we can apply Proposition \ref{proposition: spread out selmer group embedding} to obtain an orbit $\eta_{\widetilde{b}}(\widetilde{Q}) \in \intsmallG(R)\backslash \intsmallV_{\widetilde{b}}(R)$. This orbit spreads out to an element of $\intsmallG(U_1)\backslash \intsmallV_{\widetilde{b}}(U_1)$, where $U_1\subset \mathcal{X}$ is an open subset containing $x_{\Q_p}$ and intersecting the special fibre nontrivially. 
Under the bijection of Proposition \ref{proposition: G-orbits in terms of groupoids}, this defines an object $\mathcal{A}_1$ of $\GrLieE_{U_1,\widetilde{b}}$ such that the pullback of $\mathcal{A}_1$ along the point $x_{\Q_p}\in U_1(\Q_p)$ corresponds to the orbit $\eta_b(Q)$.

Let $U_2 = \mathcal{X}_{\Q_p}$. By Corollary \ref{corollary: extend groupoid dedekind scheme squarefree}, the restriction of $\mathcal{A}_1$ to $U_1\cap U_2$ extends to an object $\mathcal{A}_2$ of $\GrLieE_{U_2,\widetilde{b}}$. 
We can glue $\mathcal{A}_1$ and $\mathcal{A}_2$ to an object $\mathcal{A}_0$ of $\GrLieE_{U_0,\widetilde{b}}$, where $U_0 = U_1 \cup U_2$. 
The complement of $U_0$ has dimension zero since $\mathcal{X}_{\F_p}$ is irreducible. 
By Lemma \ref{lemma: extend GrLieE complement codim 2}, $\mathcal{A}_0$ extends to an object $\mathcal{A}_3$ of $\GrLieE_{\mathcal{X},\widetilde{b}}$. 
Let $\mathcal{A}_4\in \GrLieE_{\Z_p,b}$ denote the pullback of $\mathcal{A}_3$ along the point $x\in \mathcal{X}(\Z_p)$. 
Since $\HH^1(\Z_p, \intsmallG)$ is trivial by \cite[III.3.11(a)]{milne-etalecohomology} and Lang's theorem, Propositions \ref{proposition: G-torsors in terms of groupoids} and \ref{proposition: G-orbits in terms of groupoids} implies that $\mathcal{A}_4$ determines an element of $\intsmallG(\Z_p)\backslash\intsmallV_b(\Z_p)$ mapping to $\eta_b(Q)$ in $\smallG(\Q_p)\backslash \smallV_b(\Q_p)$. 
This completes the proof of Theorem \ref{theorem: F4 integral representatives exist}.

We conclude this subsection by stating a consequence for orbits of $\Z$.
We will need the following lemma, whose proof is identical to that of \cite[Proposition 5.7]{Thorne-Romano-E8}.
Write $\sh{E} \coloneqq \intsmallB(\Z) \cap \smallB^{\rs}(\Q)$ and $\sh{E}_p \coloneqq \intsmallB(\Z_p) \cap \smallB^{\rs}(\Q_p)$.
\begin{lemma}\label{lemma: integral reps at primes dividing N}
    Let $p$ be a prime (not necessarily coprime to $N$) and let $b_0\in \sh{E}_p$.
    Then there exists an integer $n\geq 1$ and an open compact neighborhood $W_p \subset \sh{E}_p$ of $b_0$ such that for all $b\in W_p$ and for all $y\in \Prym_{p^n\cdot b}(\Q_p)$, the orbit $\eta_{p^n\cdot b}(y)\in \intsmallG(\Q_p)\backslash \intsmallV_{p^n\cdot b}(\Q_p)$ of Theorem \ref{theorem: inject 2-descent into orbits} has a representative in $\intsmallV_{p^n\cdot b}(\Z_p)$.
\end{lemma}

\begin{corollary}\label{corollary: weak global integral representatives 2 case}
	Let $b_0 \in \sh{E}$. Then for each prime $p$ dividing $N$ we can find an open compact neighborhood $W_p$ of $b_0$ in $\sh{E}_p$ and an integer $n_p\geq 0$ with the following property. Let $M = \prod_{p\mid N} p^{n_p}$. Then for all $b\in \sh{E} \cap \left(\prod_{p\mid N} W_p \right)$ and for all $y \in \Sel_2(\Prym_{M\cdot b})$, the orbit $\eta_{M\cdot b}(y) \in \smallG(\Q) \backslash \smallV_{M\cdot b}(\Q)$ contains an element of $\intsmallV_{M\cdot b}(\Z)$. 
\end{corollary} 
\begin{proof}
    The group $\intsmallG$ has class number $1$: $\smallG(\A^{\infty})=\smallG(\Q)\cdot \intsmallG(\widehat{\Z})$ (Proposition \ref{proposition: class number 1 tamagawa}).
    Therefore an orbit $v\in \smallG(\Q) \setminus \smallV(\Q)$ has a representative in $\intsmallV(\Z)$ if and only if for every prime $p$ the associated $\smallG(\Q_p)$-orbit has a representative in $\intsmallV(\Z_p)$.
    The corollary follows from combining Theorem \ref{theorem: F4 integral representatives exist} and Lemma \ref{lemma: integral reps at primes dividing N}.

\end{proof}

\subsection{Integral representatives for \texorpdfstring{$\rhoV$}{V*}}\label{subsection: int reps rho selmer}

\begin{proof}[Proof of Theorem \ref{theorem: rho integral representatives exist}]
	 Let $A\in \Prym_b(\Q_p)$ be an element giving rise to a $\rhoG(\Q_p)$-orbit $\eta^{\star}_b(A)$ in $\rhoV_{b^{\star}}(\Q_p)$. 
	By construction of $\eta_b^{\star}$ we have a commutative diagram:
	\begin{center}
	\begin{tikzcd}
	\Prym_b(\Q_p)/2\Prym_b(\Q_p) \arrow[r,"\eta_b"] \arrow[d, two heads] & \smallG(\Q_p)\backslash \smallV_b(\Q_p)	\arrow[d, "\bigresolv"] \\
	\Prym_b(\Q_p)/\rhodual(\Prym^{\vee}_b(\Q_p)) \arrow[r,"\eta^{\star}_b"] & \rhoG(\Q_p) \backslash \rhoV_{b^{\star}}(\Q_p)
	\end{tikzcd}	
	\end{center}
	The diagram shows that the orbit $\eta^{\star}_b(A)$ is the image of the orbit $\eta_b(A)$ under the map $\bigresolv$. 
	By Theorem \ref{theorem: F4 integral representatives exist}, $\eta_b(A)$ has an integral representative $v\in \intsmallV(\Z_p)$. 
	Therefore, the element $\bigresolv(v)\in \intrhoV(\Z_p)$ is an integral representative of $\eta_b^{\star}(A)$.
	
\end{proof}

Again we state the following global corollary which follows from Lemma \ref{lemma: integral reps at primes dividing N}.

\begin{corollary}\label{corollary: weak global integral representatives rho case} 
	Let $b_0 \in \sh{E}$. Then for each prime $p$ dividing $N$ we can find an open compact neighborhood $W_p$ of $b_0$ in $\sh{E}_p$ and an integer $n_p\geq 0$ with the following property. Let $M = \prod_{p\mid N} p^{n_p}$. Then for all $b\in \sh{E} \cap \left(\prod_{p\mid N} W_p \right)$ with $b^{\star} = \bigresolv(b)$ and for all $y \in \Sel_{\rhodual}(\Prym^{\vee}_{M\cdot b})$, the orbit $\eta^{\star}_{M\cdot b}(y) \in \rhoG(\Q) \backslash \rhoV_{M\cdot b^{\star}}(\Q)$ contains an element of $\intrhoV_{M\cdot b^{\star}}(\Z)$. 
\end{corollary} 
\begin{proof}
    For each $p$ dividing $N$, let $W_p\subset \sh{E}_p$ be an open compact neighborhood of $b_0$ and $n_p\geq 0$ be an integer satisfying the conclusion of Lemma \ref{lemma: integral reps at primes dividing N}.
    Let $M=\prod_{p\mid N} n_p$, let $b\in \sh{E} \cap \left(\prod_{p\mid N} W_p \right)$ and let $y \in \Sel_{\rhodual}(\Prym^{\vee}_{M\cdot b})$ with corresponding orbit $\rhoG(\Q)\cdot v=\eta_{M\cdot b}^{\star}(y)$.
    The orbit $\rhoG(\Q_p)\cdot v$ lies in the image of $\eta^{\star}_{M\cdot b}$ so by an argument similar to the proof of Theorem \ref{theorem: rho integral representatives exist}, it is of the form $\rhoG(\Q_p)\cdot \bigresolv(w_p)$ for some $w_p\in \intsmallV(\Q_p)$ that lies in the image of $\eta_{M\cdot b}$.
    By Lemma \ref{lemma: integral reps at primes dividing N} we may assume that $w_p\in \intsmallV(\Z_p)$.
    Therefore $\rhoG(\Q_p)\cdot v$ has a representative in $\intrhoV(\Z_p)$ for every prime $p$.
    Since $\intrhoG=\PGL_2$ has class number one, $\intrhoG(\A^{\infty})=\intrhoG(\widehat{\Z})\intrhoG(\Q)$ so $\rhoG(\Q)\cdot v$ has a representative in $\intrhoV(\Z)$.
    
\end{proof}

\section{Counting integral orbits in \texorpdfstring{$\smallV$}{V}}\label{section: counting orbits in V}

In this section we will apply the counting techniques of Bhargava to provide estimates for the integral orbits of bounded height in our representation $(\intsmallG,\intsmallV)$.

\subsection{Heights} \label{subsection: heights}

Recall that $\intsmallB = \Spec \Z[p_2,p_6,p_8,p_{12}]$ and that $\smallpi: \intsmallV \rightarrow \intsmallB$ denotes the morphism of taking invariants.
For any $b\in \smallB(\Real)$ we define the \define{height} of $b$ by the formula
$$\height(b) \coloneqq \sup |p_i(b)|^{1/i}.$$
We define $\height(v) = \height(\smallpi(v))$ for any $v\in \smallV(\Real)$.
We have $\height(\lambda\cdot b) = |\lambda|\height(b)$ for all $\lambda\in \Real$ and $b\in \smallB(\Real)$. 
If $A$ is a subset of $\smallV(\Real)$ or $\smallB(\Real)$ and $X\in \Real_{>0}$ we write $A_{<X}\subset A$ for the subset of elements of height $<X$. 
For every such $X$, the set $\intsmallB(\Z)_{<X}$ is finite.

\subsection{Measures}\label{subsection: measures}

Let $\omega_{\smallG}$ be a generator for the $\Q$-vector space of left-invariant top differential forms on $\intsmallG$ over $\Q$. 
It is well-defined up to an element of $\Q^{\times}$ and it determines Haar measures $dg$ on $\smallG(\Real)$ and $\smallG(\Q_p)$ for each prime $p$.

\begin{proposition}\label{proposition: class number 1 tamagawa}
	\begin{enumerate}
		\item $\intsmallG$ has class number $1$: $\smallG(\A^{\infty}) = \smallG(\Q)\intsmallG(\widehat{\Z})$.
		\item The product $\vol\left(\intsmallG(\Z)\backslash \intsmallG(\Real) \right) \cdot \prod_p \vol\left(\intsmallG(\Z_p)\right)$ converges absolutely and equals $2$, the Tamagawa number of $\smallG$.
	\end{enumerate}
\end{proposition}
\begin{proof}
	The group $\intsmallG$ is the Zariski closure of $\smallG$ in $\GL(\intsmallV)$ and in a suitable basis of $\intsmallV$, $\smallG$ contains a maximal $\Q$-split torus consisting of diagonal matrices in $\GL(\intsmallV)$. 
	Therefore $\intsmallG$ has class number $1$ by \cite[Theorem 8.11; Corollary 2]{PlatonovRapinchuk-Alggroupsandnumbertheory} and the fact that $\Q$ has class number one. 
	The first part implies that the product in the second part equals the Tamagawa number $\tau(\smallG)$ of $\smallG\simeq (\Sp_6\times \SL_2)/\mu_2$. 
	Now use the identities $\tau(\smallG)=2\tau(\Sp_6\times \SL_2)$ \cite[Theorem 2.1.1]{Ono-relativetheorytamagawa} and $\tau(\Sp_6\times \SL_2)=1$ (because $\Sp_6$ and $\SL_2$ are simply connected).
\end{proof}

We can decompose the measure $dg$ on $\smallG(\Real)$ using the Iwasawa decomposition. 
Fix, once and for all, a maximal compact subgroup $K\subset G(\Real)$.
Let $\smallP = \smallT \smallN \subset \smallG$ be the Borel subgroup corresponding to the root basis $S_{\smallG}$, with unipotent radical $\smallN$.
Let $\smallPopp = \smallT \smallNopp \subset \smallG$ be the opposite Borel subgroup. 
Then the natural product maps 
\begin{equation*}
	\smallNopp(\Real)\times \smallT(\Real)^{\circ}\times K \rightarrow \smallG(\Real) 
	,\; \smallT(\Real)^{\circ} \times \smallNopp(\Real) \times K \rightarrow \smallG(\Real)
\end{equation*} 
are diffeomorphisms. 
If $t\in \smallT(\Real)$, let $\delta_{\smallG}(t) = \prod_{\beta\in \Phi_{\smallG}^-} \beta(t) = \det \Ad(t)|_{\Lie \smallNopp(\Real)}$. (Here $\Phi_{\smallG}^-$ denotes the subset of negative roots.)
The following result follows from well-known properties of the Iwasawa decomposition; see \cite[Chapter 3;\S1]{Lang-SL2R}.

\begin{lemma}\label{lemma: Haar measure iwasawa decomposition}
	Let $dt, dn, dk$ be Haar measures on $\smallT(\Real)^{\circ}, \smallNopp(\Real), K$ respectively. Then the assignment 
	\begin{align*}
		f \mapsto 
		\int_{t\in \smallT(\Real)^{\circ} } \int_{n\in \smallNopp(\Real)} \int_{k\in K } f(tnk)\, dk\, dn\, dt = 
		\int_{t\in \smallT(\Real)^{\circ} } \int_{n\in \smallNopp(\Real)} \int_{k\in K } f(ntk)\delta_{\smallG}(t)^{-1} \, dk\, dn\, dt
	\end{align*}
	defines a Haar measure on $\smallG(\Real)$.
\end{lemma}

We now fix Haar measures on the groups $\smallT(\Real)^{\circ}, K$ and $\smallNopp(\Real)$, as follows. 
We give $\smallT(\Real)^{\circ}$ the measure pulled back from the isomorphism $\prod_{\beta\in S_{\smallG}} \beta \colon \smallT(\Real)^{\circ} \rightarrow \Real_{>0}^4$, where $\Real_{>0}$ gets its standard Haar measure $d^{\times} \lambda = d\lambda/\lambda$.
We give $K$ its probability Haar measure. 
Finally we give $\smallNopp(\Real)$ the unique Haar measure $dn$ such that the Haar measure on $\smallG(\Real)$ from Lemma \ref{lemma: Haar measure iwasawa decomposition} coincides with $dg$.

Next we introduce measures on $\smallV$ and $\smallB$.
Let $\omega_{\smallV}$ be a generator for the free rank one $\Z$-module of left-invariant top differential forms on $\intsmallV$.
Then $\omega_{\smallV}$ is uniquely determined up to sign and it determines Haar measures $dv$ on $\smallV(\Real)$ and $\smallV(\Q_p)$ for every prime $p$.
We define the top form $\omega_{\smallB} = dp_2\wedge dp_6 \wedge dp_8 \wedge dp_{12}$ on $\intsmallB$.
It defines measures $db$ on $\smallB(\Real)$ and $\smallB(\Q_p)$ for every prime $p$. 

\begin{lemma}\label{lemma: relations different forms on V,G,B}
There exists a unique rational number $W_0\in \Q^{\times}$ with the following property.
Let $k/\Q$ be a field extension, let $\mathfrak{c}$ a Cartan subalgebra of $\smallh_k$ contained in $\smallV_k$, and let $\mu_{\mathfrak{c}}\colon \smallG_k \times \mathfrak{c} \rightarrow \smallV_k$ be the natural action map. 
Then $\mu^*_{\mathfrak{c}}\omega_{\smallV} = W_0 \omega_{\smallG}\wedge \pi|_{\mathfrak{c}}^*\omega_{\smallB}$.
\end{lemma}
\begin{proof}
    The proof is identical to that of \cite[Proposition 2.13]{Thorne-E6paper}.
    Here we use the fact that the sum of the invariants equals the dimension of the representation: $2+6+8+12 = 28 = \dim_{\Q} \smallV$. 
\end{proof}

\begin{lemma}\label{lemma: the constants W0 and W}
	Let $W_0\in \Q^{\times}$ be the constant of Lemma \ref{lemma: relations different forms on V,G,B}. Then:
	\begin{enumerate}
		\item Let $\intsmallV(\Z_p)^{\rs} = \intsmallV(\Z_p)\cap \smallV^{\rs}(\Q_p)$ and define a function $m_p: \intsmallV(\Z_p)^{\rs} \rightarrow \Real_{\geq 0}$ by the formula
		\begin{equation}
		m_p(v) = \sum_{v' \in \intsmallG(\Z_p)\backslash\left( \smallG(\Q_p)\cdot v\cap \intsmallV(\Z_p) \right)} \frac{\#Z_{\intsmallG}(v)(\Q_p)  }{\#Z_{\intsmallG}(v)(\Z_p) } .
		\end{equation}
		Then $m_p(v)$ is locally constant. 
		\item Let $\intsmallB(\Z_p)^{\rs} = \intsmallB(\Z_p)\cap \smallB^{\rs}(\Q_p)$ and let $\psi_p: \intsmallV(\Z_p)^{\rs}  \rightarrow \Real_{\geq 0}$ be a bounded, locally constant function which satisfies $\psi_p(v) = \psi_p(v')$ when $v,v'\in \intsmallV(\Z_p)^{\rs}$ are conjugate under the action of $\smallG(\Q_p)$. 
		Then we have the formula 
		\begin{equation}
		\int_{v\in \intsmallV(\Z_p)^{\rs}} \psi_p(v) \mathrm{d} v = |W_0|_p \vol\left(\intsmallG(\Z_p)\right) \int_{b\in \intsmallB(\Z_p)^{\rs}} \sum_{v\in \smallG(\Q_p)\backslash \intsmallV_b(\Z_p) } \frac{m_p(v)\psi_p(v)  }{\# Z_{\intsmallG }(v)(\Q_p)} \mathrm{d} b .
		\end{equation}
		
	\end{enumerate}
\end{lemma}
\begin{proof}
	The proof is identical to that of \cite[Proposition 3.3]{Romano-Thorne-ArithmeticofsingularitiestypeE}, using Lemma \ref{lemma: relations different forms on V,G,B}. 
\end{proof}

\subsection{Fundamental sets}\label{subsection: fundamental sets}


Let $K\subset \smallG(\Real)$ be the maximal compact subgroup fixed in \S\ref{subsection: measures}.
For any $c\in \Real_{>0}$, define $T_c \coloneqq \{t\in \smallT(\Real)^{\circ} \mid \forall \beta\in S_{\smallG} ,\, \beta(t) \leq c   \}$.
A \define{Siegel set} is, by definition, any subset $\Siegel_{\omega, c} \coloneqq \omega \cdot T_c \cdot K$, where $\omega \subset \smallNopp(\Real)$ is a compact subset and $c>0$.

\begin{proposition}\label{proposition: properties Siegel set}
	\begin{enumerate}
		\item For every $\omega\subset \smallNopp(\Real)$ and $c>0$, the set $$\{\gamma\in \intsmallG(\Z) \mid \gamma \cdot \Siegel_{\omega,c} \cap \Siegel_{\omega,c}\neq \emptyset \}$$ is finite. 
		\item We can choose $\omega\subset \smallNopp(\Real)$ and $c>0$ such that $\intsmallG(\Z) \cdot \Siegel_{\omega,c} = \smallG(\Real)$.
	\end{enumerate}
\end{proposition}
\begin{proof}
	The first part follows from the Siegel property \cite[Corollaire 15.3]{Borel-introductiongroupesarithmetiques}.
	By \cite[Theorem 4.15]{PlatonovRapinchuk-Alggroupsandnumbertheory}, the second part is reduced to proving that $\smallG(\Q)=\smallP(\Q)\cdot \intsmallG(\Z)$.
	 This follows from \cite[\S6, Lemma 1(b)]{Borel-densityMaximalityarithmetic}, using that (in the terminology of that paper) the lattice $\intsmallV$ is special with respect to the pinning $(\smallT,\smallP,\{X_{\alpha}\})$. 
\end{proof}

Now fix $\omega \subset \smallNopp(\Real)$ and $c>0$ so that $\Siegel_{\omega,c}$ satisfies the conclusions of Proposition \ref{proposition: properties Siegel set}.
By enlarging $\omega$, we may assume that $\Siegel_{\omega,c}$ is semialgebraic. 
We drop the subscripts and for the remainder of \S\ref{section: counting orbits in V} we write $\Siegel$ for this fixed Siegel set. 
The set $\Siegel$ will serve as a fundamental domain for the action of $\intsmallG(\Z)$ on $\smallG(\Real)$. 

A $\intsmallG(\Z)$-coset of $\smallG(\Real)$ may be represented more than once in $\Siegel$, but by keeping track of the multiplicities this will not cause any problems.
The surjective map $\varphi\colon \Siegel \rightarrow \intsmallG(\Z)\backslash \smallG(\Real)$ has finite fibres and if $g\in \Siegel$ we define $\mu(g) \coloneqq \# \varphi^{-1}(\varphi(g))$.
The function $\mu \colon \Siegel \rightarrow \mathbb{N}$ is uniformly bounded by $\mu_{\max} \coloneqq \# \{\gamma \in \intsmallG(\Z) \mid \gamma \Siegel \cap \Siegel \neq \emptyset \}  $ and has semialgebraic fibres.
By pushing forward measures via $\varphi$, we obtain the formula
\begin{equation}\label{equation: weighted volume Siegel set}
	\int_{g\in \Siegel} \mu(g)^{-1} \,dg = \vol\left(\intsmallG(\Z) \backslash \smallG(\Real) \right).
\end{equation}

We now construct special subsets of $\smallV^{\rs}(\Real)$ which serve as our fundamental domains for the action of $\smallG(\Real)$ on $\smallV^{\rs}(\Real)$. 
By the same reasoning as in \cite[\S 2.9]{Thorne-E6paper}, we can find open subsets $L_1,\dots, L_k$ of $\{b\in \smallB^{\rs}(\Real)\mid \height(b)=1 \}$ and sections $s_i \colon L_i \rightarrow \smallV(\Real)$ of the map $\pi\colon \smallV \rightarrow \smallB $ satisfying the following properties:
\begin{itemize}
	\item For each $i$, $L_i$ is connected and semialgebraic and $s_i$ is a semialgebraic map with bounded image. 
	\item Set $\Lambda = \Real_{>0}$. Then we have an equality 
	\begin{equation}
		\smallV^{\rs}(\Real) = \bigcup_{i=1}^k \smallG(\Real) \cdot \Lambda \cdot s_i(L_i). 
	\end{equation}
\end{itemize}

If $v\in s_i(L_i)$ let $r_i = \# Z_{\smallG}(v)(\Real)$; this integer is independent of the choice of $v$. 
We record the following change-of-measure formula, which follows from Lemma \ref{lemma: relations different forms on V,G,B}.

\begin{lemma}\label{lemma: change of variables section}
Let $\phi\colon \smallV(\Real) \rightarrow \C$ be a continuous function of compact support and $i\in \{1,\dots,k\}$.
Let $G_0\subset \smallG(\Real)$ be a measurable subset and let $m_{\infty}(v)$ be the cardinality of the fibre of the map $G_0 \times\Lambda\times  L_i \rightarrow \smallV(\Real), (g,\lambda,l)\mapsto g\cdot\lambda\cdot  s(l)$ above $v\in \smallV(\Real)$. 
Then 
\begin{align*}
\int_{v\in G_0 \cdot \Lambda \cdot s_i(L_i)} f(v)m_{\infty}(v) \,dv = |W_0| \int_{b\in \Lambda\cdot L_i}\int_{g\in G_0} f(g\cdot s_i(b)) \,dg \,db,
\end{align*}
where $W_0\in \Q^{\times}$ is the scalar of Lemma \ref{lemma: relations different forms on V,G,B}.
\end{lemma}

\subsection{Counting integral orbits in \texorpdfstring{$\smallV$}{V}}\label{subsection: counting integral orbits in V}

For any $\intsmallG(\Z)$-invariant subset $A\subset \intsmallV(\Z)$, define 
$$N(A,X) \coloneqq \sum_{v\in \intsmallG(\Z)\backslash A_{<X}} \frac{1}{\# Z_{\intsmallG}(v)(\Z)}.$$
(Recall that $A_{<X}$ denotes the elements of $A$ of height $<X$.)
Let $k$ be a field of characteristic not dividing $N$.
We say an element $v\in \intsmallV(k)$ with $b=\smallpi(v)$ is:
\begin{itemize}
	\item \define{$k$-reducible} if $\Delta(b)=0$ or if it is $\intsmallG(k)$-conjugate to the Kostant section $\smallsigma(b)$, and \define{$k$-irreducible} otherwise.
	\item \define{Almost $k$-reducible} if $\Delta(b)=0$ or if $\bigresolv(v)$ is $\intrhoG(k)$-conjugate to $\bigresolv(\smallsigma(b))$, and \define{strongly $k$-irreducible} otherwise.
	\item \define{$k$-soluble} if $\Delta(b)\neq 0$ and it lies in the image of the map $\eta_b: \intPrym_b(k)/2\intPrym_b(k) \rightarrow \intsmallG(k)\backslash \intsmallV_b(k)$ of Theorem \ref{theorem: inject 2-descent into orbits}.
\end{itemize}

We note that every strongly $k$-irreducible element is $k$-irreducible by Lemma \ref{lemma: reducible implies almost reducible}.
For any $A\subset \intsmallV(\Z)$, write $A^{irr}\subset A$ for the subset of $\Q$-irreducible elements and $A^{sirr} \subset A$ for the subset of strongly $\Q$-irreducible elements. 
Write $\smallV(\Real)^{sol} \subset \smallV(\Real)$ for the subset of $\Real$-soluble elements.

\begin{theorem}\label{theorem: counting R-soluble elements, no congruence}
	We have
	\begin{displaymath}
	N(\intsmallV(\Z)^{sirr} \cap \smallV(\Real)^{sol},X) = \frac{|W_0|}{4}\vol\left(\intsmallG(\Z)\backslash \smallG(\Real)\right) \vol\left(\smallB(\Real)_{<X}  \right)+ o\left(X^{28}\right),
	\end{displaymath}
	where $W_0\in \Q^{\times}$ is the scalar of Lemma \ref{lemma: relations different forms on V,G,B}.
\end{theorem}

We first explain how to reduce Theorem \ref{theorem: counting R-soluble elements, no congruence} to Proposition \ref{prop: counting sections}.
Recall that there exists $\G_m$-actions on $\smallV$ and $\smallB$ such that the morphism $\smallpi: \smallV \rightarrow \smallB$ is $\G_m$-equivariant and that we write $\Lambda = \Real_{>0}$. 
By an argument identical to \cite[Lemma 5.5]{Laga-E6paper}, the subset $\smallV(\Real)^{sol}\subset \smallV^{\rs}(\Real)$ is open and closed in the Euclidean topology.
Therefore by discarding some of the subsets $L_1,\dots,L_k$ of \S\ref{subsection: fundamental sets}, we may write 
$\smallV(\Real)^{sol} = \bigcup_{i\in J} \smallG(\Real)\cdot \Lambda \cdot s_i(L_i)$
for some $J\subset \{1,\dots,k\}$.
Moreover for every $b\in \smallB^{\rs}(\Real)$ we have equalities $$\#\left(\smallG(\Real)\backslash \smallV_b(\Real)^{sol}\right)/\#Z_{\smallG}(\sigma(b))(\Real)= \#\left(\Prym_b(\Real)/2\Prym_b(\Real)\right)/\#\Prym_b[2](\Real)=1/4,    $$
where the first follows from the definition of $\Real$-solubility and Proposition \ref{proposition: iso centralizer in G and prym[2]}, and the second is a general fact about real abelian surfaces.
Therefore by the inclusion-exclusion principle, to prove Theorem \ref{theorem: counting R-soluble elements, no congruence} it suffices to prove the following proposition. 

For any subset $I$ of $\{1,\dots,k\}$, write $L_I=\pi\left( \cap_{i\in I} \smallG(\Real)\cdot s_i(L_i) \right)$.
Write $s_I$ for the restriction of $s_i$ to $L_I$ and write $r_I=r_i$ for some choice of $i\in I$. (The section $s_I$ may depend on $i$ but the number $r_I$ does not if $L_I$ is non-empty.)

\begin{proposition}\label{prop: counting sections}
    In the above notation, let $(L, s, r)$ be $(L_I, s_I, r_I)$ for some $I\subset \{1,\dots, k\}$.
	Then
	$$N(\smallG(\Real)\cdot\Lambda\cdot s(L) \cap \intsmallV(\Z)^{sirr},X) = \frac{|W_0|}{r}\vol\left(\intsmallG(\Z) \backslash \smallG(\Real) \right)\vol((\Lambda\cdot L)_{<X}) +o\left(X^{28} \right).$$

\end{proposition}

So to prove Theorem \ref{theorem: counting R-soluble elements, no congruence} it remains to prove Proposition \ref{prop: counting sections}. For the latter we will follow the general orbit-counting techniques established by Bhargava, Shankar and Gross \cite{BS-2selmerellcurves, Bhargava-Gross-hyperellcurves} closely.
The only notable differences are that we work with a Siegel set instead of a true fundamental domain and that we have to carry out a case-by-case analysis for cutting off the cusp in \S\ref{subsection: cutting off cusp}. 
For the remainder of \S\ref{section: counting orbits in V} we fix a triple $(L,s,r)$ as above with $L\neq \emptyset$.

\subsection{First reductions}

We first reduce Proposition \ref{prop: counting sections} to estimating the number of (weighted) lattice points in a region of $\smallV(\Real)$. 
Recall that $\Siegel$ denotes the Siegel set fixed in \S\ref{subsection: fundamental sets} and it comes with a multiplicity function $\mu\colon \Siegel \rightarrow \mathbb{N}$.
Because $\intsmallG(\Z)\cdot \Siegel = \smallG(\Real)$, every element of $\smallG(\Real)\cdot \Lambda \cdot s(L)$ is $\intsmallG(\Z)$-equivalent to an element of $\Siegel \cdot \Lambda \cdot s(L)$.
In fact, we can be more precise about how often a $\intsmallG(\Z)$-orbit will be represented in $\Siegel \cdot \Lambda\cdot s(L)$. Let $\nu \colon \Siegel \cdot \Lambda\cdot s(L) \rightarrow \Real_{>0}$ be the `weight' function defined by 
\begin{equation}\label{equation: weight function nu}
	x\mapsto \nu(x) \coloneqq \sum_{\substack{g\in \Siegel \\ x\in g\cdot \Lambda\cdot s(L) }} \mu(g)^{-1}.
\end{equation}
Then $\nu$ takes only finitely many values and has semialgebraic fibres. 
We now claim that if every element of $\Siegel \cdot \Lambda \cdot s(L)$ is weighted by $\nu$, then the $\intsmallG(\Z)$-orbit of an element $x\in \smallG(\Real) \cdot \Lambda \cdot s(L)$ is represented exactly $\#Z_{\smallG}(x)(\Real)/\#Z_{\intsmallG}(x)(\Z)$ times. 
More precisely, for any $x \in \smallG(\Real) \cdot \Lambda \cdot s(L)$ we have 
\begin{equation}\label{equation: multiplicity orbit in fund domain}
	\sum_{x' \in \intsmallG(\Z) \cdot x \cap \Siegel \cdot \Lambda\cdot s(L) } \nu(x') = \frac{\#Z_{\smallG}(x)(\Real)}{\#Z_{\intsmallG}(x)(\Z)}.
\end{equation}
This follows from an argument similar to \cite[p. 202]{BS-2selmerellcurves}, by additionally keeping track of the multiplicity function $\mu$.

In conclusion, for any $\intsmallG(\Z)$-invariant subset $A\subset \intsmallV(\Z) \cap \smallG(\Real) \cdot \Lambda\cdot s(L)$ we have 
\begin{equation}\label{equation: N(S,a) in terms of lattice points fund domain}
	N(A,X) = \frac{1}{r} \#\left[A \cap (\Siegel\cdot \Lambda\cdot s(L))_{<X}  \right],
\end{equation}
with the caveat that elements on the right-hand-side are weighted by $\nu$. 
(Recall that $r = \#Z_{\smallG}(v)(\Real)$ for some $v\in s(L) $.)

\subsection{Averaging and counting lattice points}

We consider an averaged version of (\ref{equation: N(S,a) in terms of lattice points fund domain}) and obtain a useful expression for $N(A,X)$ (Lemma \ref{lemma: bhargavas trick}) using a trick due to Bhargava. 
Then we use this expression to count orbits lying in the `main body' of $\smallV$ using geometry-of-numbers techniques, see Proposition \ref{proposition: counting lattice points main body}.

Fix a compact, semialgebraic subset $G_0 \subset \smallG(\Real) \times \Lambda$ of non-empty interior, that in addition satisfies $K\cdot G_0 = G_0$, $\vol(G_0) = 1$ and the projection of $G_0$ onto $\Lambda$ is contained in $[1,K_0]$ for some $K_0>1$. 
Moreover we may suppose that $G_0$ is of the form $G_0'\times [1,K_0]$ where $G_0'$ is a subset of $\smallG(\Real)$. 
Equation (\ref{equation: N(S,a) in terms of lattice points fund domain}) still holds when $L$ is replaced by $hL$ for any $h\in \smallG(\Real)$, by the same argument given above.
Thus for any $\intsmallG(\Z)$-invariant $A\subset \intsmallV(\Z) \cap \smallG(\Real) \cdot \Lambda\cdot s(L)$ we obtain
\begin{equation}\label{equation: counting smallV using latttice points}
	N(A,X) = \frac{1}{r} \int_{h\in G_0} \# \left[ A\cap (\Siegel\cdot\Lambda \cdot hs(L))_{<X} \right] \, dh.
\end{equation}
We use Equation (\ref{equation: counting smallV using latttice points}) to \emph{define} $N(A,X)$ for any subset $A\subset \intsmallV(\Z) \cap \smallG(\Real) \cdot \Lambda\cdot s(L)$ which is not necessarily $\intsmallG(\Z)$-invariant. 
We can rewrite this integral using the decomposition $\Siegel = \omega \cdot T_c \cdot K$ and an argument similar to \cite[\S2.3]{BS-2selmerellcurves}, which we omit. 
We obtain:

\begin{lemma}\label{lemma: bhargavas trick}
	Given $X\geq 1$, $n\in \smallNopp(\Real)$, $t\in \smallT(\Real)$ and $\lambda \in \Lambda$, define $B(n,t,\lambda,X) \coloneqq (nt\lambda G_0 \cdot s(L))_{<X}$.
	Then for any subset $A\subset \intsmallV(\Z) \cap (\smallG(\Real)\cdot \Lambda\cdot s(L) ) $ we have
	\begin{equation}\label{equation: lemma bhargavas trick}
		N(A,X) = \frac{1}{r}\int_{\lambda = K_0^{-1}}^{X} \int_{t\in T_c} \int_{n\in \omega} \#\left[A \cap B(n,t,\lambda,X)  \right]\mu(nt)^{-1} \delta_{\smallG}(t)^{-1} \, dn \, dt \, d^{\times}\lambda,
	\end{equation}
	where an element $v\in A\cap B(n,t,\lambda,X)$ on the right hand side is counted with weight $\#\{h\in G_0 \mid v\in nt\lambda h \cdot s(L)) \}$.

\end{lemma}

Before estimating the integrand of (\ref{equation: lemma bhargavas trick}) by counting lattice points in the bounded regions $B(n,t,\lambda,X)$, we first need to handle the so-called cuspidal region after recalling and introducing some notation. 

Recall that any $v\in \smallV(\Q)$ can be decomposed as $\sum v_{\alpha}$ where $v_{\alpha}$ lies in the weight space corresponding to $\alpha \in \Phi_{\smallV}$.  
In \S\ref{subsection: a criterion for reducibility} we have defined the subspace $\smallV(M)\subset \smallV$ of elements $v$ with $v_{\alpha} = 0$ for all $\alpha \in M$.   
Define $S(M) \coloneqq \smallV(M)(\Q)\cap \intsmallV(\Z)$.
Recall that $\alpha_0 \in \Phi_{\smallV}$ denotes the highest root of $\Phi_{\smallH}$; it is maximal with respect to the partial ordering on $\Phi_{\smallV}$ defined by (\ref{equation: partial ordering}) in \S\ref{subsection: an explicit description of V}. 
We define $S(\alpha_0)$ as the \define{cuspidal region} and $\intsmallV(\Z)\setminus S(\alpha_0)$ as the \define{main body} of $\smallV$. 

The next proposition, proved in \S\ref{subsection: cutting off cusp}, says that the number of strongly irreducible elements in the cuspidal region is negligible. 
\begin{proposition}\label{prop: cutting off cusp}
	There exists $\delta >0$ such that $N(S(\alpha_0)^{sirr},X) = O(X^{28-\delta})$. 
\end{proposition}

Having dealt with the cuspidal region, we may now count lattice points in the main body using the following proposition \cite[Theorem 1.3]{BarroeroWidmer-lattice}, which strengthens a well-known result of Davenport \cite{Davenport-onaresultofLipschitz}.

\begin{proposition}\label{proposition: count lattice points barroero}
	Let $m,n\geq 1$ be integers, and let $Z\subset \Real^{m+n}$ be a semialgebraic subset. 
	For $T\in \Real^m$, let $Z_T = \{x\in \Real^n\mid (T,x) \in Z\}$, and suppose that all such subsets $Z_T$ are bounded.
	Then for any unipotent upper-triangular matrix $u\in \GL_n(\Real)$, we have
	\begin{align*}
		\#(Z_T \cap u\Z^n) = \vol(Z_T)+O(\max\{1,\vol(Z_{T,j}\}),
	\end{align*}
	where $Z_{T,j}$ runs over all orthogonal projections of $Z_T$ to any $j$-dimensional coordinate hyperplane $(1\leq j \leq n-1)$. 
	Moreover, the implied constant depends only on $Z$. 
\end{proposition}

\begin{proposition}\label{proposition: counting lattice points main body}
	Let $A = \intsmallV(\Z)\cap (\smallG(\Real)\cdot \Lambda\cdot s(L))$.
	Then 
	\begin{equation*}
		N(A \setminus S(\alpha_0) ,X) = \frac{|W_0|}{r}\vol\left(\intsmallG(\Z) \backslash \smallG(\Real) \right)\vol((\Lambda\cdot L)_{<X})+o(X^{28}).
	\end{equation*}
\end{proposition}
\begin{proof}
	This follows from estimating the set $\#[(A\setminus S(\alpha_0)) \cap B(n,t,\lambda,X)]$ using Proposition \ref{proposition: count lattice points barroero}, together with Lemmas \ref{lemma: change of variables section} and \ref{lemma: bhargavas trick} and Formula \eqref{equation: weighted volume Siegel set}; we omit the details. 
	(See \cite[Proposition 4.6]{Romano-Thorne-ArithmeticofsingularitiestypeE} for a similar proof.)
\end{proof}

\subsection{End of the proof of Proposition \ref{prop: counting sections}}

The following proposition is proven in \S\ref{subsection: estimates reducibility stabilizers}.

\begin{proposition}\label{proposition: estimates red}
	Let $\smallV^{alred}$ denote the subset of almost $\Q$-reducible elements $v\in \intsmallV(\Z)$ with $v\not\in S(\alpha_0)$. 
	Then $N(\smallV^{alred},X) = o(X^{28})$. 
\end{proposition}

We now finish the proof of Proposition \ref{prop: counting sections}.
Again let $A = \intsmallV(\Z)\cap (\smallG(\Real)\cdot \Lambda\cdot s(L))$. 
Then 
\begin{equation*}
	N(A^{sirr},X) = N(A^{sirr} \setminus S(\alpha_0),X)+N(S(\alpha_0)^{sirr},X)
\end{equation*}
The second term on the right-hand-side is $o(X^{28})$ by Proposition \ref{prop: cutting off cusp}, and $N(A^{sirr} \setminus S(\alpha_0),X)= N(A \setminus S(\alpha_0),X)+o(X^{28})$ by Proposition \ref{proposition: estimates red}.
Using Proposition \ref{proposition: counting lattice points main body}, we obtain
\begin{align*}
	N(A^{sirr},X) = \frac{|W_0|}{r}\vol\left(\intsmallG(\Z) \backslash \smallG(\Real) \right)\vol((\Lambda\cdot L)_{<X})+o(X^{28}).
\end{align*}

This completes the proof of Proposition \ref{prop: counting sections}, hence also that of Theorem \ref{theorem: counting R-soluble elements, no congruence}.

\subsection{Congruence conditions}

We now introduce a weighted version of Theorem \ref{theorem: counting R-soluble elements, no congruence}. If $w\colon \intsmallV(\Z) \rightarrow \Real$ is a function and $A\subset \intsmallV(\Z)$ is a $\intsmallG(\Z)$-invariant subset we define 
\begin{equation}\label{definition N with congruence conditions}
N_w(A,X) \coloneqq \sum_{\substack{v\in \intsmallG(\Z)\backslash A \\ \height(v)<X}} \frac{w(v)}{\# Z_{\intsmallG}(v)(\Z)}.
\end{equation}
We say a function $w$ is \define{defined by finitely many congruence conditions} if $w$ is obtained from pulling back a function $\bar{w} \colon \intsmallV(\Z/M\Z) \rightarrow \Real$ along the projection $\intsmallV(\Z) \rightarrow \intsmallV(\Z/M\Z)$ for some $M \geq 1$. 
For such a function write $\mu_w$ for the average of $\bar{w}$ where we put the uniform measure on $\intsmallV(\Z/M\Z)$. 
The following theorem follows immediately from the proof of Theorem \ref{theorem: counting R-soluble elements, no congruence}, compare \cite[\S2.5]{BS-2selmerellcurves}.

\begin{theorem}\label{theorem: counting finitely many congruence}
	Let $w\colon \intsmallV(\Z) \rightarrow \Real $ be defined by finitely many congruence conditions. Then 
	\begin{displaymath}
	N_w(\intsmallV(\Z)^{sirr} \cap \smallV(\Real)^{sol},X) = \mu_w \frac{|W_0|}{4}\vol\left(\intsmallG(\Z)\backslash \smallG(\Real)\right) \vol\left(\smallB(\Real)_{<X}  \right)+ o\left(X^{28}\right),
	\end{displaymath}
	where $W_0\in \Q^{\times}$ is the scalar of Lemma \ref{lemma: relations different forms on V,G,B}.
\end{theorem}

Next we will consider infinitely many congruence conditions. 
Suppose we are given for each prime $p$ a $\intsmallG(\Z_p)$-invariant function $w_p: \intsmallV(\Z_p) \rightarrow [0,1]$ with the following properties:
\begin{itemize}
	\item The function $w_p$ is locally constant outside the closed subset $\{v\in \intsmallV(\Z_p) \mid \Delta(v) = 0\} \subset \intsmallV(\Z_p)$. 
	\item For $p$ sufficiently large, we have $w_p(v) = 1$ for all $v \in \intsmallV(\Z_p)$ such that $p^2 \nmid \Delta(v)$. 
\end{itemize}
In this case we can define a function $w: \intsmallV(\Z) \rightarrow [0,1]$ by the formula $w(v) = \prod_{p} w_p(v)$ if $\Delta(v) \neq 0$ and $w(v) = 0$ otherwise. Call a function $w: \intsmallV(\Z) \rightarrow [0,1]$ defined by this procedure \define{acceptable}.

\begin{theorem}\label{theorem: counting infinitely many congruence conditions}
	Let $w: \intsmallV(\Z) \rightarrow [0,1]$ be an acceptable function. Then
	\begin{displaymath}
	N_w(\intsmallV(\Z)^{irr}\cap \smallV^{sol}(\Real) ,X) \leq \frac{|W_0|}{4} \left(\prod_p \int_{\intsmallV(\Z_p)} w_p(v) \mathrm{d} v \right)  \vol\left(\intsmallG(\Z) \backslash \smallG(\Real) \right) \vol\left(\smallB(\Real)_{<X}  \right) + o(X^{28}). 
	\end{displaymath}
\end{theorem}
\begin{proof}
	This inequality follows from Theorem \ref{theorem: counting finitely many congruence}; the proof is identical to the first part of the proof of \cite[Theorem 2.21]{BS-2selmerellcurves}. 
\end{proof}

\subsection{Estimates on reducibility and stabilizers}\label{subsection: estimates reducibility stabilizers}

In this subsection we give the proof of Proposition \ref{proposition: estimates red} and the following proposition which will be useful in \S\ref{section: proof of the main theorems}. 
\begin{proposition}\label{proposition: estimates bigstab}
	Let $\smallV^{bigstab}$ denote the subset of strongly $\Q$-irreducible elements $v\in \intsmallV(\Z)$ with $\#Z_{\smallG}(v)(\Q)>1$. 
	Then $N(\smallV^{bigstab},X) = o(X^{28})$.
\end{proposition}

By the same reasoning as \cite[\S10.7]{Bhargava-Gross-hyperellcurves} it will suffice to prove Lemma \ref{lemma: red and bigstab mod p} below, after having introduced some notation. 

Let $N$ be the integer of \S\ref{subsection: integral structures} and let $p$ be a prime not dividing $N$. We define $\smallV_p^{alred}\subset \intsmallV(\Z_p)$ to be the set of vectors whose reduction mod $p$ is almost $\F_p$-reducible.
We define $\smallV_p^{bigstab} \subset \intsmallV(\Z_p)$ to be the set of vectors $v\in \intsmallV(\Z_p)$ such that $p | \Delta(v)$ or the image $\bar{v}$ of $v$ in $\intsmallV(\F_p)$ has nontrivial stabilizer in $\intsmallG(\F_p)$. 

\begin{lemma}\label{lemma: red and bigstab mod p}
	We have 
	$$\lim_{Y\rightarrow +\infty} \prod_{N<p<Y} \int_{\smallV_p^{alred}} \,dv = 0,$$
	and similarly
	$$\lim_{Y\rightarrow +\infty} \prod_{N<p<Y} \int_{\smallV_p^{bigstab}} \,dv = 0.$$ 
\end{lemma}
\begin{proof}
	The proof is very similar to the proof of \cite[Lemma 5.7]{Laga-E6paper} which is in turn based on the proof of \cite[Proposition 6.9]{Thorne-Romano-E8}. 
	We first treat the case of $\smallV_p^{alred}$.
	We have the formula
	\begin{align}\label{equation: red and bigstab 1}
	\int_{\smallV_p^{alred}} \,dv =  \frac{1}{\# \intsmallV(\F_p)}\# \{v\in \intsmallV(\F_p)\mid v \text{ is almost }\F_p\text{-reducible} \}.   
	\end{align}
	Since $\#\intsmallV(\F_p) = \#\intsmallV^{\rs}(\F_p)+O(p^{27})$, it suffices to prove that there exists a nonnegative $\delta <1$ with the property that 
	$$ \frac{1}{\# \intsmallV^{\rs}(\F_p)}\# \{v\in \intsmallV^{\rs}(\F_p)\mid v \text{ is almost }\F_p\text{-reducible} \}<\delta   $$
	for all $p$ large enough.
	If $b\in \intsmallB^{\rs}(\F_p)$, Proposition \ref{proposition: spread out orbit parametrization galois} and the triviality of $\HH^1(\F_p, \intsmallG)$ (Lang's theorem) show that $\intsmallV_b(\F_p)$ is partitioned into $\#\HH^1(\F_p, \intPrym_b[2])$ many orbits, each of size $\#\intsmallG(\F_p)/\#\intPrym_b[2](\F_p)$. 
	Since $\#\intPrym_b[2](\F_p) = \#\intPrym_b(\F_p)/2\intPrym_b(\F_p) = \#\HH^1(\F_p, \intPrym_b[2])$, we have $\#\intsmallV^{\rs}(\F_p) = \#\intsmallG(\F_p)\#\intsmallB^{\rs}(\F_p)$. 
	Moreover by Corollary \ref{corollary: almost reducible equivalences} (or rather a similar statement for $\Z[1/N]$-algebras, which continues to hold by the same proof), an orbit corresponding to an element of $\HH^1(\F_p, \intPrym_b[2])$ is almost $\F_p$-reducible if and only if its image in $\HH^1(\F_p, \hat{\intellcurve}_b[2])$ is trivial. 
	Therefore the left-hand-side of (\ref{equation: red and bigstab 1}) equals
	\begin{align}\label{equation: red and bigstab 2}
		\frac{1}{\#\intsmallB^{\rs}(\F_p)}\sum_{b\in \intsmallB^{\rs}(\F_p) } \frac{\# \ker\left(\HH^1(\F_p,\intPrym_b[2]) \rightarrow \HH^1(\F_p, \hat{\intellcurve}_b[2])  \right)  }{\# \intPrym_b[2](\F_p) }.
	\end{align}
	We have $\# \ker\left(\HH^1(\F_p,\intPrym_b[2]) \rightarrow \HH^1(\F_p, \hat{\intellcurve}_b[2])  \right)\leq \#\HH^1(\F_p, \intellcurve_b[2])$ by Corollary \ref{corollary: subgroups Prym[2] using bigonal}.
	Since $\# \HH^1(\F_p, \intellcurve_b[2]) =\#\intellcurve_b[2]$, the quantity (\ref{equation: red and bigstab 2}) is bounded above by
	\begin{align}\label{equation: proof estimates red bigstab}
	\frac{1}{\#\intsmallB^{\rs}(\F_p)}\sum_{b\in \intsmallB^{\rs}(\F_p) } \frac{\# \intellcurve_b[2](\F_p)  }{\# \intPrym_b[2](\F_p) } .
	\end{align}
	Each summand in (\ref{equation: proof estimates red bigstab}) is the inverse of an integer; let $\eta_p$ be the proportion of $b\in \intsmallB^{\rs}(\F_p)$ where this summand equals $1$.  
	Then the quantity (\ref{equation: proof estimates red bigstab}) is $\leq \eta_p + (1-\eta_p)/2 = 1/2+\eta_p/2$. 
	So it suffices to prove that $\eta_p \rightarrow \eta$ for some $\eta<1$. 
    In the notation of Proposition \ref{proposition: monodromy J[2]}, let $C\subset W_{\mathrm{E}}^{\zeta}$ be the subset of elements such that
	\begin{align*}
		\frac{\#\left((1+\zeta)\splitLambda/2\splitLambda\right)^w}{\#\left((\splitLambda/2\splitLambda)^{\zeta}\right)^w  } = 1.
	\end{align*}
	Then \cite[Proposition 9.15]{Serre-lecturesonNx(p)} applied to the $W_{\mathrm{E}}^{\zeta}$-torsor $\smallt^{\rs}\rightarrow \smallB^{\rs}$ from Proposition \ref{proposition: monodromy J[2]} implies that
	$$
	\frac{1}{\# \intsmallB^{\rs}(\F_p)}\#\left\{b\in \intsmallB^{\rs}(\F_p) \mid  \frac{\# \intellcurve_b[2](\F_p)  }{\# \intPrym_b[2](\F_p) } = 1 \right\}  = \frac{\#C}{\#W_{\mathrm{E}}^{\zeta}}+O(p^{-1/2}).
	$$
	Since $1\not\in C$, this implies that $\eta<1$.
	
	Next we briefly treat the case of $\smallV_p^{bigstab}$, referring to \cite[Lemma 5.7]{Laga-E6paper} for more details. By a similar argument to the one above, it suffices to find an element $w\in W_{\mathrm{E}}^{\zeta}$ with $\left((\splitLambda/2\splitLambda)^{\zeta}\right)^w = 0$. 
	This can be achieved by taking a Coxeter element of $W_{\mathrm{E}}$ fixed by $\zeta$: the end of the proof of \cite[Lemma 5.7]{Laga-E6paper} shows that such an element has no nonzero fixed vector on $\splitLambda/2\splitLambda$ hence the same is true for its restriction to the $\zeta$-fixed points.
	An example of such a Coxeter element is $w_1w_6w_2w_3w_5w_4$, using Bourbaki notation \cite[Planche V]{Bourbaki-Liealgebras} for labelling the simple roots of $E_6$.

\end{proof}

We explain why Lemma \ref{lemma: red and bigstab mod p} implies Propositions \ref{proposition: estimates red} and \ref{proposition: estimates bigstab}.
We first claim that if $v\in \intsmallV(\Z)$ with $b=\smallpi(v)$ is almost $\Q$-reducible, then for each prime $p$ not dividing $N$ the reduction of $v$ in $\intsmallV(\F_p)$ is almost $\F_p$-reducible.
Indeed, either $\Delta(b)=0$ in $\F_p$ (in which case $v$ is almost $\F_p$-reducible), or $p\nmid \Delta(b)$ and $\bigresolv(v)$ is $\rhoG(\Q)$-conjugate to $\bigresolv(\sigma(b))$. 
In the latter case Proposition \ref{prop: integral reps squarefree rho case} implies that $\bigresolv(v)$ is $\rhoG(\Z_p)$-conjugate to $\bigresolv(\sigma(b))$, so their reductions are $\intrhoG(\F_p)$-conjugate, proving the claim.
By a congruence version of Proposition \ref{proposition: counting lattice points main body}, for every subset $L\subset \smallB(\Real)$ considered in Proposition \ref{prop: counting sections} and for every $Y>0$ we obtain the estimate:
\begin{equation*}
    N(\smallV^{alred}\cap \smallG(\Real)\cdot \Lambda \cdot s(L),X)\leq C \left(\prod_{N<p<Y} \int_{\smallV_p^{alred}} \,dv\right)\cdot X^{28} +o(X^{28}),
\end{equation*}
where $C>0$ is a constant independent of $Y$.
By Lemma \ref{lemma: red and bigstab mod p}, the product of the integrals converges to zero as $Y$ tends to infinity, so $ N(\smallV^{alred}\cap \smallG(\Real)\cdot \Lambda \cdot s(L),X)=o(X^{28})$.
Since this holds for every such subset $L$, we obtain Proposition \ref{proposition: estimates red}.

Note that we have not used Theorem \ref{theorem: counting R-soluble elements, no congruence} in this argument, but we may use it now to prove Proposition \ref{proposition: estimates bigstab}. 
Again the reduction of an element of $\smallV^{bigstab}$ modulo $p$ lands in $\smallV_p^{bigstab}$ if $p$ does not divide $N$, by Proposition \ref{prop: integral reps squarefree discr F4 case}.
Since $\lim_{X\rightarrow +\infty} N(\smallV^{bigstab},X)/X^{28}$ is $O(\prod_{N<p<Y}\int_{\smallV_p^{alred}} dv) $ by Theorem \ref{theorem: counting finitely many congruence} and the product of the integrals converges to zero by Lemma \ref{lemma: red and bigstab mod p}, this proves Proposition \ref{proposition: estimates bigstab}.

\subsection{Cutting off the cusp}\label{subsection: cutting off cusp}

In this section we prove Proposition \ref{prop: cutting off cusp}. 
We continue to use the notation introduced above its statement.
We will follow the proof of the $E_6$ case \cite[Proposition 3.6]{Thorne-E6paper} using simplifications from the proof of \cite[Theorem 6.2]{Thorne-Romano-E8}. 
We first reduce the statement to a combinatorial result, after introducing some notation. 

If $(M_0, M_1)$ is a pair of disjoint subsets of $\Phi_{\smallV}$ we define $S(M_0,M_1) = \{v\in \intsmallV(\Z) \mid \forall {\alpha}\in M_0, v_{\alpha}=0; \forall {\alpha}\in M_1, v_{\alpha} \neq 0  \}$.
Let $\mathcal{C}$ be the collection of non-empty subsets $M_0\subset \Phi_{\smallV}$ such that if $\alpha \in M_0$ and $\beta\geq \alpha$ then $\beta\in M_0$. 
(We have fixed a partial ordering on $\Phi_{\smallV}$ in \S\ref{subsection: an explicit description of V}, Equation (\ref{equation: partial ordering}).)
Given a subset $M_0 \in\mathcal{C}$ we define $\lambda(M_0) \coloneqq \{ \alpha\in \Phi_{\smallV}\setminus M_0 \mid M_0 \cup\{\alpha\} \in \mathcal{C}\}$, i.e. the set of maximal elements of $\Phi_{\smallV}\setminus M_0$. 

By definition of $\mathcal{C}$ and $\lambda$ we see that $S(\{\alpha_0\}) = \cup_{M_0\in \mathcal{C}} S(M_0,\lambda(M_0))$. 
Therefore to prove Proposition \ref{prop: cutting off cusp}, it suffices to prove that for each $M_0\in \mathcal{C}$, either $S(M_0,\lambda(M_0))^{sirr}=\emptyset$ or $N(S(M_0,\lambda(M_0)),X)=o(X^{28})$. 
By the same logic as \cite[Proposition 3.6 and \S5]{Thorne-E6paper} (itself based on a trick due to Bhargava), the estimate $N(S(M_0,\lambda(M_0)),X)=o(X^{28})$ holds if there exists a subset $M_1\subset \Phi_{\smallV}\setminus M_0$ and a function $f\colon M_1 \rightarrow \Real_{\geq 0}$ with $\sum_{\alpha\in M_1} f(\alpha) <\#M_0$ such that
\begin{align*}
	\sum_{\alpha\in \Phi_{\smallG}^+\setminus M_0 }\alpha+\sum_{\alpha\in M_1}f(\alpha)\alpha
\end{align*}
has strictly positive coordinates with respect to the basis $S_{\smallG}$. 
It will thus suffice to prove the following proposition, which is the analogue of \cite[Proposition 29]{Bhargava-Gross-hyperellcurves}.
Recall that we write $\alpha = \sum_{i=1}^4 n_i(\alpha) \beta_i$ for every $\alpha\in X^*(\smallT)\otimes \Q$.

\begin{proposition}\label{proposition: combinatorial cutting off the cusp}
Let $M_0 \in \mathcal{C}$ be a subset such that $\smallV(M_0)(\Q)$ contains strongly $\Q$-irreducible elements. 
Then there exists a subset $M_1\subset \Phi_{\smallV} \setminus M_0$ and a function $f\colon M_1 \rightarrow \Real_{\geq 0}$ satisfying the following conditions: 
\begin{itemize}
		\item We have $\sum_{\alpha \in M_1} p(\alpha) < \# M_0$.
		\item For each $i = 1,\dots, 4$ we have $\sum_{\alpha \in \Phi_{\smallG}^+} n_i(\alpha)- \sum_{\alpha \in M_0} n_i(\alpha) + \sum_{\alpha \in M_1} p(\alpha) n_i(\alpha) >0$. 
\end{itemize}
\end{proposition}

The proof of Proposition \ref{proposition: combinatorial cutting off the cusp} will be given after some useful lemmas. 
We will use the notation of Table \ref{table 3} to label the elements of $\Phi_{\smallV}$.

\begin{lemma}\label{lemma: cutting cusp 1}
Let $M_0\in \mathcal{C}$ and suppose that $\smallV(M_0)(\Q)^{sirr}\neq \emptyset$. 
Then $M_0\subset \{1,2,3,4,5,6,7,8,9,10,13\}$ and $\{9,10\} \not\subset M_0$. 
\end{lemma}
\begin{proof}
	Let $M_0$ be such a subset. Suppose that $11\in M_0$.
	Since $M_0 \in \mathcal{C}$ we have $\{1,2,3,4,5,7,8,11\}\subset M_0$.  
	By Lemma \ref{lemma: Q-reducibility conditions explicit weights}, this implies that $\smallV(M_0)(\Q)^{sirr} = \emptyset$, contradiction. 
	The same argument involving the other three subsets of Lemma \ref{lemma: Q-reducibility conditions explicit weights} shows that $15\not\in M_0$, $\{9,10\}\not\subset M_0$ and $14\not\in M_0$. 
	Therefore $M_0$ is contained in the subset of $\alpha\in \Phi_{\smallV}$ with the property that $\alpha \not\leq 11, \alpha\not\leq 14$ and $\alpha\not\leq 15$, which is easily checked to be $\{1,2,3,4,5,6,7,8,9,10,13\}$. 
\end{proof}

For the reader's convenience we give the Hasse diagram of the subset $\{1,2,3,4,5,6,7,8,9,10,13\}$ with respect to the partial ordering on $\Phi_{\smallV}$. 

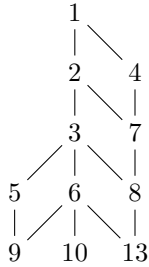
\begin{center}
\begin{tikzpicture}[scale = 0.8]
	\node (1) at (0,0) {$1$};
	\node (2) at (0,-1) {$2$};
	\node (4) at (1,-1) {$4$};
	\node (3) at (0,-2) {$3$};
	\node (7) at (1,-2) {$7$};
	\node (5) at (-1,-3) {$5$};
	\node (6) at (0,-3) {$6$};
	\node (8) at (1,-3) {$8$};
	\node (9) at (-1,-4) {$9$};
	\node (10) at (0,-4) {$10$};
	\node (13) at (1,-4) {$13$};
	\draw[shorten <= 3pt, shorten >=3pt] (1) -- (2);
	\draw[shorten <= 3pt, shorten >=3pt] (2) -- (3);
	\draw[shorten <= 3pt, shorten >=3pt] (3) -- (6);
	\draw[shorten <= 3pt, shorten >=3pt] (1) -- (4);
	\draw[shorten <= 3pt, shorten >=3pt] (4) -- (7);
	\draw[shorten <= 3pt, shorten >=3pt] (7) -- (8);
	\draw[shorten <= 3pt, shorten >=3pt] (8) -- (13);
	\draw[shorten <= 3pt, shorten >=3pt] (2) -- (7);
	\draw[shorten <= 3pt, shorten >=3pt] (3) -- (8);
	\draw[shorten <= 3pt, shorten >=3pt] (6) -- (13);
	\draw[shorten <= 3pt, shorten >=3pt] (6) -- (10);
	\draw[shorten <= 3pt, shorten >=3pt] (3) -- (5);
	\draw[shorten <= 3pt, shorten >=3pt] (5) -- (9);
	\draw[shorten <= 3pt, shorten >=3pt] (6) -- (9);
\end{tikzpicture}
\end{center}

We say a subset $M_0\in \mathcal{C}$ is \define{good} if there exists a subset $M_1\subset \Phi_{\smallV} \setminus M_0$ and a function $f\colon M_1\rightarrow \Real_{\geq 0}$ satisfying the conclusions of Proposition \ref{proposition: combinatorial cutting off the cusp}.
The following lemma is a slight generalization of \cite[Lemma 6.6]{Thorne-Romano-E8}; its proof is identical. 

\begin{lemma}\label{lemma: cutting cusp 2}
	Suppose that $M_0', M_0''\in \mathcal{C}$ with $M_0''\subset M_0'$, that $M_1' \subset \Phi_{\smallV}\setminus M_0'$, and that there exists a function $f'\colon M_1'\rightarrow \Real_{\geq 0}$ satisfying the conditions of Proposition \ref{proposition: combinatorial cutting off the cusp}.
	If there exists a function $g\colon (M_0'\setminus M_0'')\rightarrow M_1'$ such that 
	\begin{enumerate}
		\item $\alpha \geq g(\alpha)$ for all $\alpha\in M_0'\setminus M_0''$, 
		\item $f'(\alpha)-\#g^{-1}(\alpha)\geq 0$ for all $\alpha\in M_1'$,
	\end{enumerate}
	then any $M_0 \in \mathcal{C}$ such that $M_0''\subset M_0 \subset M_0'$ is good. 
\end{lemma}
\begin{proof}
	Given such a subset $M_0$, define $f\colon M_1' \rightarrow \Real$ by $f(\alpha) = f'(\alpha)-\#[g^{-1}(\alpha)\cap(M_0'\setminus M_0)]$.
	The second condition on $g$ implies that $f$ takes values in $\Real_{\geq 0 }$. We have
	\begin{align*}
		\sum_{\alpha \in M_1'}f(\alpha) = \sum_{\alpha\in M_1'}f'(\alpha)-\#[M_0'\setminus M_0] < \#M_0.
	\end{align*}
	Moreover 
	\begin{align*}
		\sum_{\alpha\in \Phi_{\smallG}^+\setminus M_0 }\alpha+\sum_{\alpha\in M_1'}f(\alpha)\alpha = \left(\sum_{\alpha\in \Phi_{\smallG}^+\setminus M'_0 }\alpha+\sum_{\alpha\in M_1'}f'(\alpha)\alpha\right)+ \sum_{\alpha\in M_0'\setminus M_0 }(\alpha-g(\alpha)).
	\end{align*}
	The first term on the right-hand-side has positive coordinates with respect to $S_{\smallG}$ by assumption on $f'$ and the second term has nonnegative coordinates by the first condition on $g$. 
\end{proof}

Table \ref{table: explicit estimates cusp data} gives examples of $M_0'\in \mathcal{C}$ together with a subset $M_1' \subset \Phi_{\smallV} \setminus M_0'$ and a function $f'\colon M_1'\rightarrow \Real_{\geq 0}$ satisfying the conclusions of Proposition \ref{proposition: combinatorial cutting off the cusp}.
The third column denotes the coordinates of $2\sum_{\alpha \in \Phi_{\smallV}\setminus M'_0} \alpha $ with respect to the basis $S_{\smallG}$. 
The validity this table can be easily checked in conjunction with Table \ref{table 3}. 
For example, checking the second row amounts to checking that the nonnegative reals $(v_4,v_{10},v_{14}) = (0,3\frac{1}{2},1\frac{1}{2})$ have the property that $v_4+v_{10}+v_{14}<6$ and that the vector 
$$
(2v_4+2v_{10}-2v_{14}+4  ,4v_4+8 , 3v_4-v_{10}+v_{14}+4, -v_4 +v_{10}+v_{14}-4 )
$$
has strictly positive entries.

\begin{table}
\centering	
\begin{tabular}{|c | c | c | c | c | c |}
\hline
$M_0'$ & $M_1'$ & Weights & $f'\colon M'_1 \rightarrow \Real_{\geq 0}$\\
\hline
$1,2,3,5,6,10$ & $4,12$ & $2,8,6,-4$ & $(0,5)$ \\
$1,2,3,5,6,9$ & $4,10,14$ & $4,8,4,-4$ & $(0,3\frac{1}{2},1\frac{1}{2})$\\
$1,2,3,4,5,6,7,8,10,13$ & $9,11,15$ & $-6,-2,0,0$ & $(4\frac{3}{4},1\frac{1}{8},3\frac{1}{8})$ \\
$1,2,3,4,5,6,7,8,9,13$ & $10,11,14$ & $-4,-2,-2,0$ & $(3\frac{3}{4},4\frac{3}{8},1\frac{5}{8})$ \\
$1,2,3,4,5,6,7,8,9$ & $10,11,13,14$ & $-2,-2,-1,-1$ & $(3\frac{7}{32},1\frac{5}{16},1\frac{9}{32},2\frac{7}{8})$ \\
\hline
\end{tabular}
\caption{Examples of good $M_0'$}
\label{table: explicit estimates cusp data}
\end{table}

\begin{proof}[Proof of Proposition \ref{proposition: combinatorial cutting off the cusp}]
	Let $M_0\in \mathcal{C}$ be such a subset. 
	By Lemma \ref{lemma: cutting cusp 1}, $M_0 \subset \{1,2,3,4,5,6,7,8,9,10,13\}$ and $\{9,10\}\not\subset M_0$.
	We prove that $M_0$ is good by considering various cases together with the information of Table \ref{table: explicit estimates cusp data}. 
	If $\#M_0\leq 2$ then $M_0 = \{1\}, \{1,2\}$ or $\{1,4\}$ so taking $M_1 = {3}$ and $f(3) = 1/2$ shows that $M_0$ is good. 
	We may assume for the remainder of the proof that $\#M_0\geq 3$, which implies that $\{1,2\} \subset M_0$. 
	
	Case 1:  $4\not\in M_0$ and $9\not\in M_0$. Then $\{1,2\} \subset M_0\subset \{1,2,3,5,6,10\}$. We apply Lemma \ref{lemma: cutting cusp 2} with $(M_0',M_1',f')$ given by the first row of Table \ref{table: explicit estimates cusp data}, $M_0'' = \{1,2\}$ and $g\colon (M'_0 \setminus M_0'') \rightarrow M_1'$ given by $3,5,6,10 \mapsto 12$. 

	Case 2: $4\not\in M_0$ and $9\in M_0$. Then $10 \not\in M_0$ by Lemma \ref{lemma: cutting cusp 1}, hence $M_0 = \{1,2,3,5,6,9\}$. 
	The second row of Table \ref{table: explicit estimates cusp data} shows that $M_0$ is good.
	
	Case 3: $4\in M_0$ and $9\not\in M_0$. Then $\{1,2,4\} \subset M_0 \subset \{1,2,3,4,5,6,7,8,10,13\}$.
	We apply Lemma \ref{lemma: cutting cusp 2} with $(M_0',M_1',f')$ given by the third row of Table \ref{table: explicit estimates cusp data}, $M_0'' = \{1,2,4\}$ and $g\colon (M'_0 \setminus M_0'') \rightarrow M_1'$ given by $3,5,6\mapsto 9; 7\mapsto 11; 8,10,13\mapsto 15$. 
		
	Case 4: $4\in M_0$ and $9\in M_0$. 
	Lemma \ref{lemma: cutting cusp 1} then implies that $10\not\in M_0$.
	If $13\in M_0$, then $M_0 = \{1,2,3,4,5,6,7,8,9,13\}$, which is good by the fourth row of Table \ref{table: explicit estimates cusp data}.
	If $13\not\in M_0$, then $\{1,2,3,4,5,6,9\} \subset M_0 \subset \{1,2,3,4,5,6,7,8,9\}$.
	We apply Lemma \ref{lemma: cutting cusp 2} with $(M_0',M_1',f')$ given by the fifth row of Table \ref{table: explicit estimates cusp data}, $M_0'' = \{1,2,3,4,5,6,9\}$ and $g\colon (M'_0 \setminus M_0'') \rightarrow M_1'$ given by $7\mapsto 11$; $8\mapsto 13$. 
\end{proof}

\section{Counting integral orbits in \texorpdfstring{$\rhoV$}{V*}}\label{section: counting orbits in rhoV}

In this section we count integral orbits in the representation $\intrhoV$.
Since $\intrhoV$ is essentially the space of binary quartic forms, the methods here will be very similar to the ones employed by Bhargava and Shankar \cite{BS-2selmerellcurves}.

\subsection{The local Selmer ratio of a self-dual isogeny}

The following lemma is presumably well-known; it generalizes the observation that $\# (E(\Q_p)/nE(\Q_p)) = |n|_p^{-1} \#E(\Q_p)[n]$ if $E/\Q_p$ is an elliptic curve. 
It follows from local duality theorems.


\begin{lemma}\label{lemma: selmer ratio selfdual}
Let $k = \Real$ or $\Q_p$ for some $p$ and let $K$ be a finite extension of $k$. Write $| \cdot|_{k}: k^{\times} \rightarrow \Real_{>0}$ for the normalized absolute value of $k$.
Let $A$ be an abelian variety over $K$ with dual $A^{\vee}$. Let $\lambda: A\rightarrow A^{\vee}$ be a self-dual isogeny. 
Then the degree of $\lambda$ is a square number $m^2$ for some $m\in \Z_{\geq 1}$. Consider the quantity
$$c(\lambda) := \frac{\#\left(A^{\vee}(K)/\lambda(A(K)) \right)}{\#A[\lambda](K)}. $$
Then $c(\lambda) = |m|_{k}^{-[K:k]}$.

\end{lemma}
\begin{proof}
	The selfduality of $\lambda$ implies that there is a perfect alternating pairing $A[\lambda]\times A[\lambda] \rightarrow \G_{m,K} , $
so the degree of $\lambda$ is a square number $m^2$ for some $m\in \Z_{\geq 0}$.
This pairing induces a pairing on Galois cohomology
$\HH^1(K,A[\lambda])\times \HH^1(K,A[\lambda]) \rightarrow \HH^2(K,\G_m)\hookrightarrow \Q/\Z$
which is also perfect and alternating and the image of the descent map $A^{\vee}(K)/\lambda(A(K))\rightarrow \HH^1(K,A[\lambda])$ is a maximal isotropic subspace \cite[Proposition 4.10]{PoonenRains-maximalisotropic}.
This implies that
\begin{equation}\label{equation: selmer ratio first equation}
\left(\#\frac{A^{\vee}(K)}{\lambda(A(K))}\right)^2 = \# \HH^1(K,A[\lambda]). 
\end{equation}
By the local Euler characteristic formula \cite[Theorems I.2.8 and I.2.13]{Milne-ADT}, we obtain the equality 
\begin{equation}\label{equation: selmer ratio lemma second equation}
\#\HH^1(K,A[\lambda]) = \#\HH^0(K,A[\lambda]) \#\HH^2(K,A[\lambda]) |m|^{-2[K:\Q_p]}_{k} .
\end{equation}
We have $\HH^0(K,A[\lambda])=A[\lambda](K)$ and local Tate duality implies that $\HH^2(K,A[\lambda])\simeq \HH^0(K,A[\lambda])^{\vee} \simeq A[\lambda](K)^{\vee}\simeq A[\lambda](K)$ too. The lemma follows from combining Equations (\ref{equation: selmer ratio first equation}) and (\ref{equation: selmer ratio lemma second equation}). 
\end{proof}

\begin{corollary}\label{corollary: selmer ratios rho}
	Let $k = \Real$ or $\Q_p$ for some $p$. 
	If $b\in \smallB^{\rs}(k)$, then in the notation of Lemma \ref{lemma: selmer ratio selfdual} the quantities $c(\rho\colon \Prym_b\rightarrow \Prym_b^{\vee})$ and $c(\rhodual\colon \Prym_b^{\vee} \rightarrow \Prym_b)$ coincide and equal
	\begin{displaymath}
	\begin{cases}
	1/2 & \text{if }k = \Real,\\
	2 & \text{if }k = \Q_2, \\
	1 & \text{else}.
	\end{cases}
	\end{displaymath}
\end{corollary}
\begin{proof}
	We only treat the case of $c(\rho)$, the case of $c(\rhodual)$ being analogous. 
	The isogeny $\rho$ is self-dual by Lemma \ref{lemma: rho is self-dual} and its kernel is isomorphic to $\ellcurve_b[2]$, which has order $4$.
	Apply Lemma \ref{lemma: selmer ratio selfdual}.
\end{proof}

\subsection{Heights, measures and fundamental sets} \label{subsection: heights, measures and fundamental sets}

We discuss objects and notions analogous to those of \S\ref{subsection: heights} and \S\ref{subsection: measures} in the context of the representation $\intrhoV$.
Recall from \S\ref{subsection: the representation rhoV} that we have defined the isomorphism $\bigresolv\colon \smallB\rightarrow \rhoB$ which was spread out to an isomorphism $\intsmallB_S\rightarrow \intrhoB_S$ in \S\ref{subsection: integral structures}.
We will use $\bigresolv$ to transport definitions of $\intsmallB_S$ to $\intrhoB_S$.
For example, we set $\Delta^{\star}\coloneqq\Delta\circ \bigresolv^{-1}$, $\Delta^{\star}_{\ellcurve}\coloneqq\Delta_{\ellcurve}\circ \bigresolv^{-1}$, $\Delta^{\star}_{\hat{\ellcurve}}\coloneqq\Delta_{\hat{\ellcurve}}\circ \bigresolv^{-1}$, all elements of $S[\intrhoB]$.
Furthermore we define $\intrhoBrs\coloneqq \intrhoB_S[\left(\Delta^{\star}\right)^{-1}]$, and we define $\intrhoVrs$ as the preimage of $\intrhoBrs$ under $\rhopi\colon \intrhoV\rightarrow \intrhoB$.
The $S$-schemes $\intrhoBrs$ and $\intrhoVrs$ have generic fibre $\rhoBrs$ and $\rhoVrs$.

For any $b\in \rhoB(\Real)$ we define the \define{height} of $b$ by the formula
$$\height(b) = \height(\bigresolv^{-1}(b)),$$
where the height on $\smallB(\Real)$ is defined in \S\ref{subsection: heights}.
We define $\height(v) = \height(\rhopi(v))$ for any $v\in \rhoV(\Real)$.
If $A$ is a subset of $\rhoV(\Real)$ or $\rhoB(\Real)$ and $X\in \Real_{>0}$ we write $A_{<X}\subset A$ for the subset of elements height $<X$. 

Let $\omega_{\rhoG}$ be a generator for the free rank $1$ $\Z$-module of left-invariant top differential forms on $\intrhoG$ over $\Z$. 
It is well-defined up to sign and it determines Haar measures $dg$ on $\rhoG(\Real)$ and $\rhoG(\Q_p)$ for each prime $p$.

Let $\omega_{\rhoV}$ be a generator for the free rank one $\Z$-module of left-invariant top differential forms on $\intrhoV$.
Then $\omega_{\rhoV}$ is uniquely determined up to sign and it determines Haar measures $dv$ on $\rhoV(\Real)$ and $\rhoV(\Q_p)$ for every prime $p$.
We define the top form $\omega_{\rhoB}$ on $\rhoB$ as the pullback of the form $\omega_{\smallB}$ from \S\ref{subsection: measures} under the isomorphism $\bigresolv^{-1} \colon \rhoB \rightarrow \smallB$.
It defines measures $db$ on $\rhoB(\Real)$ and $\rhoB(\Q_p)$ for every prime $p$.

\begin{lemma}\label{lemma: the constants W1 rho case}
	There exists a constant $W_1\in \Q^{\times}$ with the following properties:
	\begin{enumerate}
		\item Let $\intrhoV(\Z_p)^{\rs} = \intrhoV(\Z_p)\cap \rhoVrs(\Q_p)$ and define a function $m_p: \intrhoV(\Z_p)^{\rs} \rightarrow \Real_{\geq 0}$ by the formula
		\begin{equation}
		m_p(v) = \sum_{v' \in \intrhoG(\Z_p)\backslash\left( \rhoG(\Q_p)\cdot v\cap \intrhoV(\Z_p) \right)} \frac{\#Z_{\intrhoG}(v)(\Q_p)  }{\#Z_{\intrhoG}(v)(\Z_p) } .
		\end{equation}
		Then $m_p(v)$ is locally constant. 
		\item Let $\intrhoB(\Z_p)^{\rs} = \intrhoB(\Z_p)\cap \rhoBrs(\Q_p)$ and let $\psi_p: \intrhoV(\Z_p)^{\rs}  \rightarrow \Real_{\geq 0}$ be a bounded, locally constant function which satisfies $\psi_p(v) = \psi_p(v')$ when $v,v'\in \intrhoV(\Z_p)^{\rs}$ are conjugate under the action of $\rhoG(\Q_p)$. 
		Then we have the formula 
		\begin{equation}
		\int_{v\in \intrhoV(\Z_p)^{\rs}} \psi_p(v) \mathrm{d} v = |W_1|_p \vol\left(\intrhoG(\Z_p)\right) \int_{b\in \intrhoB(\Z_p)^{\rs}} \sum_{v\in \rhoG(\Q_p)\backslash \intrhoV_b(\Z_p) } \frac{m_p(v)\psi_p(v)  }{\# Z_{\intrhoG }(v)(\Q_p)} \mathrm{d} b .
		\end{equation}

	\end{enumerate}
\end{lemma}
\begin{proof}
	The proof is the same as that of \cite[Proposition 3.3]{Romano-Thorne-ArithmeticofsingularitiestypeE}, using the fact that the sum of the weights of the $\G_m$-action on $\rhoV$ equals $28$, the sum of the invariants of $\rhoB$.
\end{proof}

We henceforth fix a constant $W_1\in \Q^{\times}$ satisfying the properties of Lemma \ref{lemma: the constants W1 rho case}.

We now construct special subsets of $\rhoVrs(\Real)$ which serve as our fundamental domains for the action of $\rhoG(\Real)$ on $\rhoVrs(\Real)$. 
We have to slightly alter the fundamental sets used in \cite{BS-2selmerellcurves} because the discriminant we use here to define $\rhoBrs(\Real)$ is larger than the discriminant used by Bhargava and Shankar.

In the notation of \cite[\S2.1]{BS-2selmerellcurves}, if $i=0,1,2+, 2-$, let
$\rhoV(\Real)^{(i)}\subset \rhoV(\Real)$
be the subset of elements $(b_2,b_6,q)$ where the binary quartic form $q$ has $4-2i$ real roots, and is positive/negative definite if $i=2+$/$2-$ respectively.
For each such $i$, an open semialgebraic subset $K^{(i)}\subset \{b\in \rhoB(\Real) \mid b_2=b_6=0 \text{ and } \height(b)=1\}$ and a section $s_{(i)}\colon K^{(i)} \rightarrow \rhoV(\Real)^{(i)}$ of $\rhopi$ is given in \cite[Table 1]{BS-2selmerellcurves}, with the property that (if $\Lambda = \Real_{>0}$):
$$\{v\in \rhoV(\Real) \mid b_2=b_6=0 \text{ and }\Delta^{\star}_{\hat{\ellcurve}}(\rhopi(v))\neq 0\} = \bigsqcup_{i=0,1,2\pm}\rhoG(\Real) \cdot \Lambda \cdot s_{(i)}(L^{(i)}). $$
Let $M^{(i)}\subset \rhoV(\Real)$ be the subset of elements $v=(b_2,b_6,q) $ satisfying $\height(v)=1$ and $q\in \Lambda \cdot K^{(i)}$.
Let $L_1,\dots,L_k$ be the connected components of $\rhopi(M^{(i)})\cap \rhoBrs(\Real)$ for all $i=0,1,2\pm$; they are also the connected components of $\{b\in \rhoBrs(\Real)\mid \height(b)=1 \}$.

By construction, the open subsets $L_1,\dots, L_k$ of $\{b\in \rhoBrs(\Real)\mid \height(b)=1 \}$ come equipped with sections $s_i \colon L_i \rightarrow \rhoVrs(\Real)$ of $\rhopi\colon \rhoV(\Real) \rightarrow \rhoB(\Real) $ and satisfy the following properties, which follow from the corresponding properties of $K^{(i)}$:
\begin{itemize}
	\item For each $i$, $L_i$ is connected and semialgebraic and $s_i$ is a semialgebraic map with bounded image. 
	\item Set $\Lambda = \Real_{>0}$. Then we have an equality 
	\begin{equation}\label{equation: sections of G on V rho case}
		\rhoVrs(\Real) = \bigsqcup_{i=1}^k \rhoG(\Real) \cdot \Lambda \cdot s_i(L_i). 
	\end{equation}
\end{itemize}
If $v\in s_i(L_i)$ let $r_i = \# Z_{\rhoG}(v)(\Real)$; this integer is independent of the choice of $v$. 



\subsection{Counting integral orbits in \texorpdfstring{$\rhoV$}{V*}}\label{subsection: count integral orbits in rhoV}

For any $\intrhoG(\Z)$-invariant subset $A\subset \intrhoV(\Z)$ and function $w\colon \intrhoV(\Z) \rightarrow \Real$, define 
$$N_w(A,X) \coloneqq \sum_{v\in \intrhoG(\Z)\backslash A_{<X}} \frac{w(v)}{\# Z_{\intrhoG}(v)(\Z)}.$$

Let $k$ be a field of characteristic not dividing $N$. 
We say an element $v\in \intrhoV(k)$ with $b^{\star} = \rhopi(v)$ and $b = \bigresolv^{-1}(b^{\star})$ is:
\begin{itemize}
	\item \define{$k$-reducible} if $\Delta(b)=0$ or if $v$ is $\intrhoG(k)$-conjugate to $\bigresolv(\smallsigma(b))$, and \define{$k$-irreducible} otherwise. 
	\item \define{$k$-soluble} if $\Delta(b)\neq 0$ and $v$ lies in the image of the map $\eta^{\star}_b\colon \intPrym_b(k)/\rhodual(\intPrym_b^{\vee}(k)) \rightarrow \intrhoG(k)\backslash \intrhoV_{b^{\star}}(k)$ of Theorem \ref{theorem: inject rho-descent orbits}. 
\end{itemize}

In more classical language, an element $v = (b_2,b_6,q) \in \intsmallV(k)$ is $k$-reducible if $\rhodisc(v)=0$ or the binary quartic form $q$ has a $k$-rational linear factor. (This follows from Lemma \ref{lemma: reducible implies almost reducible}.)
The definition of $k$-soluble elements introduced here does not relate in a direct way to the notion of solubility used in the more classical sense as in \cite{BS-2selmerellcurves}; it depends not only on $q$ but also on $b_2$ and $b_6$.

For any $A\subset \intrhoV(\Z)$ write $A^{irr}\subset A$ for the subset of $\Q$-irreducible elements. 
Write $\rhoV(\Real)^{sol} \subset \rhoV(\Real)$ for the subset of $\Real$-soluble elements. 
We say a function $w \colon \intrhoV(\Z)\rightarrow \Real$ is defined by \define{finitely many congruence conditions} if it is the pullback of a function $\bar{w}\colon \intrhoV(\Z/M\Z) \rightarrow \Real$ and for such $w$ write $\mu_w$ for the average of $\bar{w}$ when $\intrhoV(\Z/M\Z)$ is given the uniform probability measure. 
Recall that $W_1$ denotes the constant fixed in \S\ref{subsection: heights, measures and fundamental sets} satisfying the conclusions of Lemma \ref{lemma: the constants W1 rho case}.

\begin{theorem}\label{theorem: counting finite congruence rho case}
	Let $w\colon \intrhoV(\Z) \rightarrow \Real$ be a function defined by finitely many congruence conditions. Then
	$$
	N_w(\intrhoV(\Z)^{irr} \cap \rhoV(\Real)^{sol},X) = \mu_w \frac{|W_1|}{2}\vol\left(\intrhoG(\Z) \backslash \intrhoG(\Real) \right)\vol(\smallB(\Real)_{<X}) +o(X^{28}).
	$$
\end{theorem}
\begin{proof}
    For every $b\in \rhoBrs(\Real)$ and $v\in  \rhoV_b(\Real)$ we have equalities $$\#\left(\rhoG(\Real)\backslash \rhoV_b(\Real)^{sol}\right)/\#Z_{\rhoG}(v)(\Real)= \#\left(\Prym_b(\Real)/\rhodual(\Prym^{\vee}_b(\Real))\right)/\#\Prym_b^{\vee}[\rhodual](\Real)=1/2,    $$
    where the first follows from the definition of $\Real$-solubility and Lemma \ref{lemma: centralizer resolv bin quartic bigonal ell curve}, and the second from Corollary \ref{corollary: selmer ratios rho}.
    
    By an argument identical to that of \cite[Lemma 5.5]{Laga-E6paper}, the subset $\rhoV(\Real)^{sol}$ is open and closed in $\intrhoVrs(\Real)$.
    Using the decomposition (\ref{equation: sections of G on V rho case}) and discarding those sections which do not contain $\Real$-soluble elements, it suffices to prove that for each $L_i$ we have
	$$N_w(\rhoG(\Real)\cdot\Lambda\cdot s_i(L_i) \cap \intrhoV(\Z)^{irr},X) = \mu_w \frac{|W_1|}{r_i}\vol\left(\intrhoG(\Z) \backslash \rhoG(\Real) \right)\vol((\Lambda\cdot L_i)_{<X}) +o\left(X^{28} \right).$$
	This may be proved in exactly the same way as Proposition \ref{prop: counting sections} using the results of \cite[\S2]{BS-2selmerellcurves}; we omit the details. 
\end{proof}

Next we consider infinitely many congruence conditions.
The key input is the following uniformity estimate, which follows immediately from the one obtained by Bhargava and Shankar \cite[Theorem 2.13]{BS-2selmerellcurves}.
We recall from \S\ref{subsection : discriminant polynomial}, \S\ref{subsection: the representation rhoV} that $\rhoDelta = \rhoDelta_{\ellcurve}\rhoDelta_{\hat{\ellcurve}}$ (up to a unit in $\Z[1/N]$) and that $\rhoDelta_{\hat{\ellcurve}}$ coincides with the usual discriminant of the binary quartic form $q$ (again up to a unit in $\Z[1/N]$).

\begin{proposition}\label{proposition: uniformity estimate}
    For a prime $p$ not dividing $N$, let $\mathcal{W}_p(\rhoV)$ denote the subset of $v\in \intrhoV(\Z)^{irr}$ such that $p^2 \mid \rhodisc_{\hat{E}}(v)$. For any $M>N$ we have 
	\begin{align}\label{equation: uniformity estimate}
		\lim_{X\rightarrow \infty} N\left(\cup_{p>M}\mathcal{W}_p(\rhoV),X \right)/X^{28} = O(1/\log M)
	\end{align}
	where the implied constant is independent of $M$.
\end{proposition}

Suppose we are given for each prime $p$ a $\intrhoG(\Z_p)$-invariant function $w_p: \intrhoV(\Z_p) \rightarrow [0,1]$ with the following properties:
\begin{itemize}
	\item The function $w_p$ is locally constant outside a closed subset of $\intrhoV(\Z_p)$ of measure zero. 
	\item For $p$ sufficiently large and not dividing $N$, we have $w_p(v) = 1$ for all $v \in \intrhoV(\Z_p)$ such that $p^2 \nmid \rhodisc_{\hat{E}}(v)$ and $\rhodisc(v)\neq 0$. 
\end{itemize}
In this case we can define a function $w: \intrhoV(\Z) \rightarrow [0,1]$ by the formula $w(v) = \prod_{p} w_p(v)$ if $\rhodisc(v) \neq 0$ and $w(v) = 0$ otherwise. Call a function $w: \intrhoV(\Z) \rightarrow [0,1]$ defined by this procedure \define{acceptable}.

\begin{theorem}\label{theorem: counting infinitely many congruence conditions rho case}
	Let $w: \intrhoV(\Z) \rightarrow [0,1]$ be an acceptable function. Then
	\begin{displaymath}
	N_w(\intrhoV(\Z)^{irr}\cap \rhoV(\Real)^{sol} ,X) = \frac{|W_1|_{\infty}}{2} \left(\prod_p \int_{\intrhoV(\Z_p)} w_p(v) \mathrm{d} v \right)  \vol\left(\intrhoG(\Z) \backslash \rhoG(\Real) \right) \vol\left(\smallB(\Real)_{<X}  \right) + o(X^{28}). 
	\end{displaymath}
\end{theorem}
\begin{proof}
    Our definition of an acceptable function slightly differs from the one the one employed in \cite[\S2.7]{BS-2selmerellcurves}, since we only require that for sufficiently large primes $p$, $w_p(v)=1$ if $p^2\nmid \rhodisc_{\hat{E}}(v)$ \emph{and} $\rhodisc(v)\neq 0$. 
    Let $S\subset \intrhoV(\Z)$ be the subset of $b$ with $\rhodisc(b)=0$.
	Bearing in mind that $N(S,X)=o(X^{28})$ and the closure of $S$ in $\intrhoV(\Z_p)$ is of measure zero, the proof of the theorem is identical to that of \cite[Theorem 2.21]{BS-2selmerellcurves}, using Proposition \ref{proposition: uniformity estimate}.
\end{proof}

\begin{remark}
We obtain an equality in Theorem \ref{theorem: counting infinitely many congruence conditions rho case} whereas in Theorem \ref{theorem: counting infinitely many congruence conditions} we only obtain an upper bound. 
This is because the proof of Theorem \ref{theorem: counting infinitely many congruence conditions rho case} relies on the uniformity estimate for $\Delta_{\hat{E}}^*$ of Proposition \ref{proposition: uniformity estimate}.
We expect a similar uniformity estimate to hold for $\intsmallV$ with respect to $\Delta$, but this is not known.
    
\end{remark}

\section{Proof of the main theorems} \label{section: proof of the main theorems}

In this section we combine all the previous results and prove the main theorems of the introduction. 
To calculate the average size of the $\rho$-Selmer group in \S\ref{subsection: main thrm rho selmer}, we reduce it to calculating the average size of the $\rhodual$-Selmer group using the bigonal construction (Theorem \ref{theorem: summary bigonal construction}).
To make this reduction step precise, we consider the effect of changing the parameter space by an automorphism in \S\ref{subsection: changing the parameter space}.

\subsection{Changing the parameter space}\label{subsection: changing the parameter space}


This section is based on a remark of Poonen and Stoll \cite[Remark 8.11]{PoonenStoll-Mosthyperellipticnorational}.
Let $n\geq 1$ be an integer and $\genB = \A^n_{\Z}$ be affine $n$-space with coordinates $x_1,\dots,x_n$. Suppose that $\mathcal{B}$ is equipped with a $\G_m$-action such that $\lambda\cdot  x_i = \lambda^{d_i}x_i$ for some set of positive weights $d_1\leq \dots \leq d_n$; let $d$ be their sum.

\begin{definition}
Let $T$ be a subset of $\genB(\Real) \times \prod_{p}\genB(\Z_p)$ of the form $T_{\infty} \times \prod_p T_p$.
We say $T$ is a \define{generalized box} if 
\begin{itemize}
	\item $T_{\infty}\subset \genB(\Real)$ is open, bounded and semialgebraic.
	\item For each prime number $p$, $T_p\subset \genB(\Z_p)$ is open and compact and for all but finitely many $p$ we have $T_p = \genB(\Z_p)$.
\end{itemize}
If in addition $T_{\infty}$ is a product of intervals $(a_1,b_1)\times \cdots \times (a_n,b_n)$, we say $T$ is a \define{box}.
\end{definition}
Let $T$ be a generalized box. We define 
\begin{align*}
\sh{E}_{T,<X} \coloneqq \{b\in \genB(\Q) \mid b\in X\cdot T_{\infty} \text{ and } b\in T_p \text{ for all } p \}.
\end{align*}
Every element of $\sh{E}_{T,<X}$ lies in $\genB(\Z)$.
We note that if $\genB = \intsmallB$ and $T = [-1,1]^6 \times \prod \intsmallB(\Z_p)$ then $\sh{E}_{T,<X}$ coincides with the elements of $\intsmallB(\Z)$ of height bounded by $X$. 

The top form $dx_1\wedge \dots \wedge dx_n$ defines measures on $\genB(\Real)$ and $\genB(\Z_p)$ for every prime $p$ which satisfy $\vol(\genB(\Z_p)) = 1$ for all $p$. 
We define the volume of a generalized box $T$ by $\vol(T_{\infty}) \prod_p \vol(T_p)$; the previous sentence shows this is well-defined.


We first show that the volume is well-behaved under $\G_m$-equivariant automorphisms.

\begin{lemma}\label{lemma: automorphism preserves volume}
	Let $\phi \colon \genB \rightarrow \genB$ be a $\G_m$-equivariant morphism such that $\phi_{\Q}\colon \genB_{\Q}\rightarrow \genB_{\Q}$ is an isomorphism. 
	Let $T$ be a generalized box of $\genB$. Then $\phi(T)$ is a generalized box, $\phi(\sh{E}_{T,<X}) = \sh{E}_{\phi(T),<X}$ and moreover $\vol(\phi(T)) = \vol(T)$. 
\end{lemma}
\begin{proof}
	The first two claims follow from the definitions; it remains to compute the volume of $\phi(T)$. 
	Up to an element of $\Q^{\times}$, the form $\omega = dx_1\wedge \dots \wedge dx_n$ is the unique nonzero $n$-form of $\genB_{\Q}$ that is homogeneous of degree $d$.
	Since the pullback $\phi^*\omega$ has the same properties, $\phi^*\omega = a\cdot \omega$ for some $a\in \Q^{\times}$. 
	The lemma follows from the product formula $|a|\prod_p |a|_p=1$. 
\end{proof}

Let $f\colon \genB(\Q) \rightarrow \Real_{\geq 0}$ be a function such that $f(\lambda\cdot b) = f(b)$ for all $\lambda\in \Q^{\times}$ and $b\in \genB(\Q)$.
Let $C \in \Real_{\geq 0}$ be a constant. 
We say \define{$\Eq(f,C)$ holds for the generalized box $T$} if 
\begin{equation}\label{equation: estimate generalized box}
	\sum_{b\in \sh{E}_{T,<X}} f(b) = C \vol(T) X^d+o(X^d)
\end{equation}
as $X\rightarrow +\infty$.
We say \define{$\Eq^{\leq}(f,C)$ holds for $T$} if in (\ref{equation: estimate generalized box}) the equality is replaced by $\leq$.

\begin{proposition}\label{proposition: automorphism preserves estimates}
	Let $f,C$ and $\phi$ be as above and $\bullet\in \{\emptyset,\leq \}$.
	Suppose that $\Eq^{\bullet}(f,C)$ holds for all boxes of $\genB$. 
	Then $\Eq^{\bullet}(f\circ \phi,C)$ holds for all generalized boxes of $\genB$.
\end{proposition}
\begin{proof}
	We only consider the case of $\Eq^{\leq}(f,C)$, the case of $\Eq(f,C)$ being analogous. 
	By approximating the infinite component using rectangles, $\Eq^{\leq}(f,C)$ holds for all generalized boxes of $\genB$.
	If $T$ is a generalized box then by Lemma \ref{lemma: automorphism preserves volume}, $\phi(T)$ is a generalized box with the same volume as $T$ and $\phi(\sh{E}_{T,<X}) = \sh{E}_{\phi(T),<X}$. 
	So 
	\begin{align*}
		\sum_{\sh{E}_{T,<X}} f(\phi(b)) = \sum_{\sh{E}_{\phi(T),<X}}f(b) 
		 \leq C \vol(\phi(T)) X^d+o(X^d) 
		= C\vol(T) X^d +o(X^d),
	\end{align*}
	proving the proposition.

\end{proof}

\begin{remark}\label{remark: changing variables elliptic curves}

Most orbit-counting results using the geometry-of-numbers methods as employed in \S \ref{section: counting orbits in V} are valid for any generalized box, with the same proof.
Proposition \ref{proposition: automorphism preserves estimates} shows that for these counting results, the choice of homogeneous coordinates of $\genB$ is irrelevant.  
For example, consider the family of elliptic curves 
\begin{equation}\label{equation: ell curves fam 1 remark}
y^2+p_2xy+p_6y = x^3+p_8x+p_{12}.
\end{equation}
After applying a homogeneous change of coordinates we obtain the family 
\begin{equation}\label{equation: ell curves fam 2 remark}
(y+p_2x+p_6)^2 = x^3+p_8x+p_{12}.
\end{equation}
The results of \cite{BS-2selmerellcurves, BS-3Selmer, BS-4Selmer, BS-5Selmer} are valid for any box of $\A^2_{(p_8,p_{12})}$ hence trivially for any box of $\A^4_{(p_2,p_6,p_8,p_{12})}$ parametrizing elliptic curves in Family (\ref{equation: ell curves fam 2 remark}). Proposition \ref{proposition: automorphism preserves estimates} shows that these results remain valid for any generalized box for the elliptic curves in Family (\ref{equation: ell curves fam 1 remark}) too. 
\end{remark}

\subsection{The average size of the \texorpdfstring{$\rho$}{rho}-Selmer group}\label{subsection: main thrm rho selmer}

Recall that $\sh{E}\subset \intsmallB(\Z)$ denotes the subset of elements $b$ with $\Delta(b)\neq 0$.
We say a subset $\mathcal{F} \subset \sh{E}$ is \define{defined by finitely many congruence conditions} if it is the preimage of a subset of $\intsmallB(\Z/M\Z)$ under the mod $M$ reduction map $\sh{E} \rightarrow \intsmallB(\Z/M\Z)$.

\begin{theorem}\label{theorem: average size rhovee selmer group}
	Let $\mathcal{F}\subset \sh{E}$ be a subset defined by finitely many congruence conditions. Then 
	\begin{equation*}
	\lim_{X\rightarrow \infty} \frac{1}{\# \mathcal{F}_{<X}} \sum_{b\in \mathcal{F}_{<X}} \# \Sel_{\rhodual}\Prym^{\vee}_b  	= 3.
	\end{equation*}
\end{theorem}

The proof is very similar to the proof of \cite[Theorem 6.1]{Laga-E6paper}; we include it here for completeness.
Note that we obtain an equality here and not just an upper bound using the uniformity estimate of Proposition \ref{proposition: uniformity estimate} combined with Proposition \ref{prop: integral reps squarefree rho case}.
We first prove a local statement.
Recall that there is a $\G_m$-action on $\intsmallB$ which satisfies $\lambda\cdot p_i = \lambda^i p_i$, and that $\sh{E}_p$ denotes the subset of elements $b$ of $\intsmallB(\Z_p)$ with $\Delta(b)\neq 0$, equipped with the $p$-adic subspace topology.
Also let $\mathcal{F}_p$ be the closure of $\mathcal{F}$ in $\sh{E}_p$.

\begin{proposition}\label{proposition: local result of main theorem rho selmer}
	Let $b_0 \in \mathcal{F}$. Then we can find for each prime $p$ dividing $N$ an open compact neighborhood $W_p$ of $b_0$ in $\sh{E}_p$ with the following property. Let $\mathcal{F}_W = \mathcal{F} \cap \left(\prod_{p | N} W_p \right)$. Then we have 
	\begin{equation*}
	\lim_{X\rightarrow \infty} \frac{ \sum_{b\in \mathcal{F}_W,\; \height(b)<X }\# \Sel_{\rhodual}\Prym^{\vee}_b   }{\#  \{b \in \mathcal{F}_W \mid \height(b) < X \}}	=3.
	\end{equation*}
\end{proposition}
\begin{proof}
	Choose sets $W_p$ and integers $n_p\geq 0$ for $p| N$ satisfying the conclusion of Corollary \ref{corollary: weak global integral representatives rho case}. We assume after shrinking the $W_p$ that they satisfy $W_p \subset \mathcal{F}_p$. 
	If $p$ does not divide $N$, set $W_p = \mathcal{F}_p$ and $n_p = 0$. Let $M = \prod_{p} p^{n_p}$. 
	
	For $v\in \intrhoV(\Z)$ with $b^{\star} = \rhopi(v)$ and $\bigresolv^{-1}(b^{\star}) = b \in \smallB(\Q)$, define $w(v) \in \Q_{\geq 0}$ by the following formula:
	\begin{displaymath}
	w(v) = 
	\begin{cases}
	\left( \sum_{v'\in \intrhoG(\Z)\backslash \left( \intrhoG(\Q)\cdot v \cap \intrhoV(\Z) \right)}  \frac{\# Z_{\intrhoG}(v')(\Q)}{\# Z_{\intrhoG}(v')(\Z)} \right)^{-1} & \text{if }b\in p^{n_p}\cdot W_p \text{ and } \rhoG(\Q_p)\cdot v \in \eta^{\star}_{b}(\Prym_b(\Q_p)/\rhodual(\Prym_b(\Q_p))) \text{ for all }p, \\
	0 & \text{otherwise.}
	\end{cases}
	\end{displaymath}
	Define $w'(v)$ by the formula $w'(v) = \#Z_{\intrhoG}(v)(\Q)\cdot  w(v)$. 
	Corollaries \ref{corollary: Selrho embeds} and \ref{corollary: weak global integral representatives rho case} imply that if $b\in M \cdot \mathcal{F}_{W}$, non-identity elements in the $\rhodual$-Selmer group of $\Prym^{\vee}_b$ correspond bijectively to $\rhoG(\Q)$-orbits in $\rhoV_{b^{\star}}(\Q)$ that intersect $\intrhoV(\Z)$ nontrivially, that are $\Q$-irreducible and that are soluble at $\Real$ and $\Q_p$ for all $p$. In other words, we have the formula: 
	\begin{equation}\label{equation: selmer count vs orbit count}
	\sum_{\substack{b \in \mathcal{F}_W \\ \height(b) <X}}\left( \#\Sel_{\rhodual}(\Prym^{\vee}_b)-1 \right) 
	= \sum_{\substack{b \in M\cdot\mathcal{F}_W \\ \height(b) <M \cdot X}}\left( \#\Sel_{\rhodual}(\Prym^{\vee}_b)-1 \right)
	= N_{w'}(\intrhoV(\Z)^{irr}\cap\rhoV(\Real)^{sol}  ,M \cdot X).
	\end{equation}
	Since the number of $\intrhoG(\Z)$-orbits of $v\in \intrhoV(\Z)_{<X}$ with $Z_{\rhoG}(v)(\Q)\neq 1$ is negligible \cite[Lemma 2.4]{BS-2selmerellcurves}, we have
	\begin{equation}\label{equation: compare w and w'}
	 N_{w'}(\intrhoV(\Z)^{irr}\cap \rhoV(\Real)^{sol}  ,M \cdot X) =  N_{w}(\intrhoV(\Z)^{irr}\cap \rhoV(\Real)^{sol},M \cdot X) + o(X^{28}).
	\end{equation}
	It is more convenient to work with $w(v)$ than with $w'(v)$ because $w(v)$ is an acceptable function in the sense of \S\ref{subsection: count integral orbits in rhoV}. 
	Indeed, for $v\in \intrhoV(\Z_p)$ with $\rhopi(v)=b^{\star}$ and $b =\bigresolv^{-1}(b^{\star})$, define $w_p(v) \in \Q_{\geq 0}$ by the following formula:
	\begin{displaymath}
	w_p(v) = 
	\begin{cases}
	\left( \sum_{v'\in \intrhoG(\Z_p)\backslash \left( \intrhoG(\Q_p)\cdot v \cap \intrhoV(\Z_p) \right)}  \frac{\# Z_{\intrhoG}(v')(\Q_p)}{\# Z_{\intrhoG}(v')(\Z_p)} \right)^{-1} & \text{if }b\in p^{n_p}\cdot W_p \text{ and } \rhoG(\Q_p)\cdot v \in \eta^{\star}_{b}(\Prym_b(\Q_p)/\rhodual(\Prym^{\vee}_b(\Q_p)) ), \\
	0 & \text{otherwise.}
	\end{cases}
	\end{displaymath}
	Then \cite[Proposition 3.6]{BS-2selmerellcurves} shows that $w(v)  =\prod_pw_p(v)$ for all $v\in\intrhoV(\Z)$. The remaining properties for $w(v)$ to be acceptable follow from Part 1 of Lemma \ref{lemma: the constants W1 rho case} and Proposition \ref{prop: integral reps squarefree rho case}.
	Moreover using Lemma \ref{lemma: the constants W1 rho case} we obtain the formula
	\begin{equation}\label{equation: mass formula w}
	\int_{v\in \intrhoV(\Z_p)} w_p(v) d v = |W_1|_p \vol\left(\intrhoG(\Z_p) \right) \int_{b \in p^{n_p}\cdot {W_p}} \frac{\#\Prym_b(\Q_p)/\rhodual(\Prym^{\vee}_b(\Q_p))}{\#\Prym^{\vee}_b[\rhodual](\Q_p)}d b.
	\end{equation}
	Using the equality $\#\Prym_b(\Q_p)/\rhodual(\Prym^{\vee}_b(\Q_p)) = |1/2|_p  \#\Prym^{\vee}_b[\rhodual](\Q_p)$ of Corollary \ref{corollary: selmer ratios rho}, we see that the integral on the right hand side equals $|1/2|_p\vol(p^{n_p}\cdot W_p)=|1/2|_pp^{-28n_p} \vol(W_p)$.
	Combining the identities (\ref{equation: selmer count vs orbit count}) and (\ref{equation: compare w and w'}) shows that 
	\begin{align*}
	\lim_{X\rightarrow +\infty} X^{-28} \sum_{\substack{b \in \mathcal{F}_W \\ \height(b) <X}}\left( \#\Sel_{\rhodual }(\Prym^{\vee}_b)-1 \right)
	& = \lim_{X\rightarrow +\infty} X^{-28}N_w(\intrhoV(\Z)^{irr}\cap \rhoV(\Real)^{sol},M \cdot X).\\
	\end{align*}
	(That is, the limit on the left-hand-side exists if and only if the limit on the right-hand-side exists, and in that case their values coincide.)
	By Theorem \ref{theorem: counting infinitely many congruence conditions rho case} and the estimate $\vol(\smallB(\Real)_{<X})= 2^4 X^{28}+o(X^{28})$, the right-hand-side equals
	 \begin{displaymath}
	 \frac{|W_1|}{2} \left(\prod_p \int_{\intrhoV(\Z_p)} w_p(v) d v \right)  \vol\left(\intrhoG(\Z) \backslash \rhoG(\Real) \right) 2^4M^{28}.
	 \end{displaymath}
	 Using (\ref{equation: mass formula w}) and the remarks thereafter this simplifies to
	\begin{displaymath}
	\vol\left(\intrhoG(\Z)\backslash \intrhoG(\Real) \right) \prod_p \vol\left(\intrhoG(\Z_p)\right) 2^4\prod_{p} \vol(W_p).
	\end{displaymath}
	On the other hand, since $\mathcal{F}_W$ is defined by congruence conditions we have
	\begin{equation}\label{equation: estimate subset finitely many congruence}
	\lim_{X\rightarrow+\infty} \frac{\#  \{b \in \mathcal{F}_W \mid \height(b) < X \}}{X^{28}} = 2^4\prod_p \vol(W_p).
	\end{equation}
	We conclude that 
	\begin{displaymath}
	\lim_{X\rightarrow\infty} \frac{ \sum_{b\in \mathcal{F}_W,\; \height(b)<X } \left(\# \Sel_{\rhodual }\Prym^{\vee}_b-1 \right)  }{\#  \{b \in \mathcal{F}_W \mid \height(b) < X \}}	= \vol\left(\intrhoG(\Z)\backslash \intrhoG(\Real) \right) \cdot \prod_p \vol\left(\intrhoG(\Z_p)\right).
	\end{displaymath}
	Since the Tamagawa number of $\rhoG = \PGL_2$ is $2$, the proposition follows. 
\end{proof}

To deduce Theorem \ref{theorem: average size rhovee selmer group} from Proposition \ref{proposition: local result of main theorem rho selmer}, choose for each $i\in \Z_{\geq 1}$ open compact subsets $W_{p,i} \subset\sh{E}_p$ (for $p$ dividing $N$) such that if $\mathcal{F}_{W_i} = \mathcal{F} \cap \left( \prod_{p | N} W_{p,i} \right)$, then $W_i$ satisfies the conclusion of Proposition \ref{proposition: local result of main theorem rho selmer} and we have a countable partition $\mathcal{F} = \mathcal{F}_{W_1}\sqcup \mathcal{F}_{W_2} \sqcup \cdots$. 
By an argument identical to the proof of \cite[Theorem 7.1]{Thorne-Romano-E8}, we see that for any $\varepsilon >0$, there exists $k\geq 1$ such that 
\begin{displaymath}
\limsup_{X\rightarrow+\infty}  \frac{ \sum_{\substack{b \in \sqcup_{i\geq k} \mathcal{F}_{W_i} , \height(b) < X  }} \left(\#\Sel_{\rhodual}\Prym^{\vee}_b -1\right)      }{ \# \{b \in \mathcal{F} \mid \height(b) < X  \}  }<\varepsilon.
\end{displaymath}
Using Proposition \ref{proposition: local result of main theorem rho selmer} this implies that 
\begin{align*}
\limsup_{X\rightarrow+\infty}  \frac{ \sum_{\substack{b\in \mathcal{F} , \height(b) < X  }}  \left(\#\Sel_{\rhodual}\Prym^{\vee}_b -1\right)      }{ \# \{b \in \mathcal{F} \mid \height(b) < X  \}  } &\leq 2 \limsup_{X\rightarrow+\infty}\frac{\# \{b \in \sqcup_{i<k} \mathcal{F}_{W_i} \mid \height(b) < X \}  }{ \# \{b \in \mathcal{F} \mid \height(b) < X  \}  } +\varepsilon \\
&\leq 2+\varepsilon.  
\end{align*}
Since the above inequality is true for any $\varepsilon >0$, the expression on the left is bounded above by $2$.
Similarly we obtain $$\liminf_{X\rightarrow \infty}\frac{ \sum_{\substack{b\in \mathcal{F} , \height(b) < X  }}  \left(\#\Sel_{\rhodual}\Prym^{\vee}_b -1\right)      }{ \# \{b \in \mathcal{F} \mid \height(b) < X  \}  } \geq 2.$$
Combining the last two inequalities concludes the proof of Theorem \ref{theorem: average size rhovee selmer group}.

Using the bigonal construction from \S\ref{subsection: the bigonal construction}, we immediately obtain the average size of $\Sel_{\rho}\Prym_b$. 
\begin{theorem}\label{theorem: average size rho selmer group}
	Let $\mathcal{F}\subset \sh{E}$ be a subset defined by finitely many congruence conditions. 
	Then the average size of $\Sel_{\rho}\Prym_b$ for $b\in \mathcal{F}$, when ordered by height, exists and equals $3$. 
\end{theorem}
\begin{proof}
    In the notation of \S\ref{subsection: changing the parameter space}, Theorem \ref{theorem: average size rhovee selmer group} remains valid for any box of $\intsmallB$, by an identical proof.
    Since $\Sel_{\rho}\Prym_{\hat{b}}\simeq \Sel_{\rhodual}\Prym_b^{\vee}$ (Theorem \ref{theorem: summary bigonal construction}), the theorem follows from Proposition \ref{proposition: automorphism preserves estimates} applied to the automorphism $\chi\colon \smallB \rightarrow \smallB$.
\end{proof}

\subsection{The average size of the 2-Selmer group}\label{subsection: main theorem 2 selmer}

If $b\in \smallB^{\rs}(\Q)$, let $\Sel^{\natural}_2 \Prym_b \subset \Sel_2 \Prym_b$ be the subset of elements whose image under the embedding $\Sel_2 \Prym_b \hookrightarrow \smallG(\Q) \backslash \smallV_b(\Q)$ is strongly $\Q$-irreducible (as defined in \S\ref{subsection: counting integral orbits in V}). 
By Corollary \ref{corollary: almost reducible equivalences}, it coincides with the subset of $\Sel_2\Prym_b$ whose image in $\Sel_{\rhodual}\Prym^{\vee}_b$ under $\rho$ is nontrivial.

\begin{theorem}\label{theorem: average size strongly irred 2 selmer}
	Let $\mathcal{F}\subset \sh{E}$ be a subset defined by finitely many congruence conditions (see \S\ref{subsection: main thrm rho selmer}).
	Then the average size of $\Sel_2^{\natural}\Prym_b$ for $b\in \mathcal{F}$, when ordered by height, is bounded above by $2$. 
\end{theorem}
\begin{proof}
    The proof is very similar to that of Theorem \ref{theorem: average size rhovee selmer group}, using the results of \S\ref{subsection: twisting and embedding the selmer group}, \S\ref{subsection: integral reps: the $2$-Selmer case} and \S\ref{section: counting orbits in V}. 
    We give a brief sketch. 
    Again it suffices to prove that for each $b_0\in \mathcal{F}$ and for every prime $p$ dividing $N$, we can find an open compact neighborhood $W_p$ of $b_0$ in $\mathcal{F}_p$ such that the average size of $\Sel_2^{\natural}\Prym_b$ is bounded above by $2$ in the family $\mathcal{F}_W\coloneqq \mathcal{F} \cap \left(\prod_{p | N} W_p \right)$.
    Choose sets $W_p\subset \mathcal{F}_p$ and integers $n_p\geq 0$ for $p\mid N$ satisfying the conclusion of Corollary \ref{corollary: weak global integral representatives 2 case}.
    Set $W_p=\mathcal{F}_p$ and $n_p=0$ if $p$ does not divide $N$. Let $M=\prod_p p^{n_p}$.
    For $v\in \intsmallV(\Z)$ with $\smallpi(v)=b$, define $w(v) \in \Q_{\geq 0}$ by the following formula:
	\begin{displaymath}
	w(v) = 
	\begin{cases}
	\left( \sum_{v'\in \intsmallG(\Z)\backslash \left( \intsmallG(\Q)\cdot v \cap \intsmallV(\Z) \right)}  \frac{\# Z_{\intsmallG}(v')(\Q)}{\# Z_{\intsmallG}(v')(\Z)} \right)^{-1} & \text{if }b\in p^{n_p}\cdot W_p \text{ and } \smallG(\Q_p)\cdot v \in \eta_{b}(\Prym_b(\Q_p)/2\Prym_b(\Q_p) \text{ for all }p , \\
	0 & \text{otherwise.}
	\end{cases}
	\end{displaymath}
    Then Corollaries \ref{corollary: Sel2 embeds} and \ref{corollary: weak global integral representatives 2 case} and Proposition \ref{proposition: estimates bigstab} imply that \begin{equation*}
	\sum_{\substack{b \in \mathcal{F}_W \\ \height(b) <X}} \#\Sel_{2}^{\natural}(\Prym_b)
	= N_{w}(\intsmallV(\Z)^{sirr}\cap\smallV(\Real)^{sol}  ,M \cdot X)+o(X^{28}).
	\end{equation*}
	Similar to the proof of Theorem \ref{theorem: average size rhovee selmer group}, the function $w(v)$ decomposes into a product of local terms $\prod_p w_p(v)$ and is acceptable by Part 1 of Lemma \ref{lemma: the constants W0 and W} and Proposition \ref{prop: integral reps squarefree discr F4 case}.
	By Theorem \ref{theorem: counting infinitely many congruence conditions} and the evaluation of the integrals $\int_{\intsmallV(\Z_p)}w_p(v)dv$ using Lemma \ref{lemma: the constants W0 and W}, we obtain the estimate 
	\begin{equation*}
	    N_{w}(\intsmallV(\Z)^{sirr}\cap\smallV(\Real)^{sol}  ,M \cdot X)\leq  \vol\left(\intsmallG(\Z)\backslash \intsmallG(\Real) \right) \prod_p \vol\left(\intsmallG(\Z_p)\right) \prod_{p} \vol(W_p)\vol(\smallB(\Real)_{<X})+o(X^{28}).
	\end{equation*}
	The result now follows from Equation (\ref{equation: estimate subset finitely many congruence}) and the fact that the Tamagawa number of $\smallG$ is $2$ (Proposition \ref{proposition: class number 1 tamagawa}).
\end{proof}
    
To obtain a bound on the full $2$-Selmer group of $\Prym_b$, we use the results of \S \ref{subsection: main thrm rho selmer}.
For every $b\in \smallB^{\rs}(\Q)$, the factorization of isogenies $[2] = \rhodual \circ \rho$ gives rise to an exact sequence 
\begin{align*}
	\Sel_{\rho} \Prym_b \rightarrow \Sel_2 \Prym_b \rightarrow \Sel_{\rhodual} \Prym^{\vee}_b .
\end{align*}

We obtain the inequality $$\#\Sel_2\Prym_b \leq \#\Sel_2^{\natural}\Prym_b+\# \Sel_{\rho}\Prym_b.$$
Our last result then follows from Theorems \ref{theorem: average size rho selmer group} and \ref{theorem: average size strongly irred 2 selmer}:

\begin{theorem}\label{theorem: the average size of the 2-Selmer group}
	Let $\mathcal{F}\subset \sh{E}$ be a subset defined by finitely many congruence conditions.
	Then the average size of $\Sel_2\Prym_b$ for $b\in \mathcal{F}$, when ordered by height, is bounded above by $5$.  
\end{theorem}

\begin{remark}
	For every $b\in \smallB^{\rs}(\Q)$ we have an exact sequence 
	\begin{align*}
	\hat{\ellcurve}_b[2](\Q) \rightarrow \Sel_{\rho} \Prym_b \rightarrow \Sel_2 \Prym_b \rightarrow \Sel_{\rhodual} \Prym^{\vee}_b .
\end{align*}
	Moreover an easy Hilbert irreducibility argument shows that the average size of $\#\hat{\ellcurve}_b[2](\Q)$ for $b\in \mathcal{F}$ is $1$. 
	We conclude that the average size of $\Sel_2 \Prym_b$ equals the sum of the average sizes of $\Sel_{\rho}\Prym_b$ (which is $3$ by Theorem \ref{theorem: average size rho selmer group}) and $\Sel_2^{\natural}\Prym_b$, provided the latter quantity exists.
\end{remark}

\bibliographystyle{alpha}


\begin{footnotesize}
\textsc{Jef Laga  }\;  \texttt{jcsl5@cam.ac.uk}   \newline
\textsc{Department of Pure Mathematics and Mathematical Statistics, Wilberforce Road, Cambridge, CB3 0WB, UK}
\end{footnotesize}

\end{document}